\documentclass[11pt]{article}
\usepackage[T1]{fontenc}
\usepackage[utf8]{inputenc}
\IfFileExists{lmodern.sty}{\usepackage{lmodern}}{}
\usepackage{amsmath,amssymb,amsthm,mathtools,mathrsfs}
\usepackage[a4paper,margin=28mm]{geometry}
\usepackage[protrusion=true,expansion=false]{microtype}
\usepackage{graphicx,booktabs,array}
\usepackage{float}
\usepackage{enumitem}
\usepackage{xcolor}
\usepackage{hyperref}
\hypersetup{
  colorlinks=true,
  linkcolor=blue!55!black,
  citecolor=blue!55!black,
  urlcolor=blue!65!black,
  pdfauthor={Yutian Li},
  pdftitle={Pólya's Conjecture for the Neumann Eigenvalues on Euclidean Balls},
  pdfsubject={Sharp Neumann eigenvalue counting bounds on Euclidean balls},
  pdfkeywords={Neumann Laplacian, Pólya conjecture, Euclidean ball,
    Dini condition, Robin problem, Bessel zeros, lattice point counting}
}
\allowdisplaybreaks
\setlist{nosep}
\numberwithin{equation}{section}
\newtheorem{theorem}{Theorem}[section]
\newtheorem{proposition}[theorem]{Proposition}
\newtheorem{lemma}[theorem]{Lemma}
\newtheorem{corollary}[theorem]{Corollary}
\theoremstyle{definition}

\newtheorem*{unnumberedremark}{Remark}

\title{Pólya's Conjecture for the Neumann Eigenvalues on Euclidean Balls}
\author{Yutian Li\\[3pt]
\small School of Mathematics and Statistics\\
\small \mbox{Nanfang College, Guangzhou}, 510970, Guangdong, China\\
\small \href{mailto:liyt@nfu.edu.cn}{\texttt{liyt@nfu.edu.cn}}}
\date{}
\begin{document}
\maketitle

\begin{abstract}
We prove Pólya's conjectured lower bound for the Neumann eigenvalue
counting function of Euclidean balls.  If
\(B_R^d\subset\mathbb R^d\) is the ball of radius \(R\), then, for every
\(d\ge2\), \(R>0\), and \(E\ge0\),
\[
N_{B_R^d}^{<}(E)
\ge
\frac{\omega_d}{(2\pi)^d}|B_R^d|E^{d/2}
=
\frac{(R\sqrt E)^d}{2^d\Gamma(\frac d2+1)^2},
\]
where \(\omega_d\) is the volume of the unit \(d\)-ball and
\(N_{B_R^d}^{<}(E)\) counts Neumann eigenvalues strictly below \(E\).
Combined with the Dirichlet theorem for balls, this settles both Pólya
inequalities for Euclidean balls in every dimension \(d\ge2\).  In the disk
case, the proof replaces a computer-assisted finite-frequency step by
explicit Rayleigh--Ritz estimates.

In dimensions \(d\ge3\), the radial Neumann condition is a Dini condition
rather than a derivative-zero Bessel condition.  A strict comparison with
an auxiliary Robin problem transfers a derivative-zero Bessel phase estimate
to the physical Neumann spectrum.  The problem then becomes a comparison between a
multiplicity-weighted phase staircase and an integral equal to the Weyl
term.  Variational trial spaces control low frequencies; finitely many
radial levels and beta-integral estimates cover the intermediate range;
and a uniform phase estimate treats high frequencies.  All finite
computations for \(2\le d\le6\) are printed in the paper.  For \(d\ge7\),
one compact two-parameter estimate is verified in exact rational arithmetic
by the ancillary program.
\end{abstract}



\tableofcontents
\clearpage

\section{Introduction and main result}

For a bounded Euclidean domain, Pólya conjectured that the leading term in
Weyl's law bounds the Dirichlet counting function from above and the
Neumann counting function from below, at every energy.  For Euclidean
balls, Filonov--Levitin--Polterovich--Sher \cite{FLPS} proved the Dirichlet
inequality in every dimension and the Neumann inequality for the disk.
They identified the higher-dimensional Neumann problem as the remaining
ball case and announced a subsequent treatment \cite[p.~133]{FLPS}.
Frank and Larson \cite[Section~1]{FrankLarson} also refer to work in
preparation by the same authors concerning the Neumann case in dimensions
\(d\ge3\).  We have not had access to that work, and the proof presented
here was obtained independently.  It establishes the Neumann inequality
for balls in dimensions \(d\ge3\).  We therefore do not attempt to compare
the two approaches.

For a bounded Lipschitz domain \(\Omega\subset\mathbb R^d\), write
\begin{equation*}
N_\Omega^{<}(E)
:=
\#\{j:\mu_j^N(\Omega)<E\},
\end{equation*}
where the Neumann eigenvalues are counted with multiplicity and
\(\mu_1^N(\Omega)=0\).  Weyl's law states that
\begin{equation*}
N_\Omega^{<}(E)
\sim
\frac{\omega_d}{(2\pi)^d}|\Omega|E^{d/2}
\qquad(E\to\infty).
\end{equation*}
The Neumann side of Pólya's conjecture asserts that the leading term is a
lower bound for every \(E\ge0\).  Our main result is the following.

\begin{theorem}[Neumann Pólya inequality for balls]
\label{thm:main}
For every integer \(d\ge2\), every \(R>0\), and every \(E\ge0\),
\begin{equation*}
N_{B_R^d}^{<}(E)
\ge
\frac{\omega_d}{(2\pi)^d}|B_R^d|E^{d/2}
=
\frac{(R\sqrt E)^d}{2^d\Gamma(\frac d2+1)^2}.
\end{equation*}
Here \(\omega_d=\pi^{d/2}/\Gamma(\frac d2+1)\) is the volume of the unit
ball in \(\mathbb R^d\).
\end{theorem}

\begin{corollary}[Pólya's conjecture for Euclidean balls]
\label{cor:polya-balls}
Both the Dirichlet and Neumann Pólya inequalities hold for Euclidean balls
in every dimension \(d\ge2\).
\end{corollary}

\begin{proof}
The Dirichlet inequality is \cite[Theorem~1.2]{FLPS}; the Neumann inequality
is Theorem \ref{thm:main}.
\end{proof}

If
\(
N_\Omega^{\le}(E):=\#\{j:\mu_j^N(\Omega)\le E\},
\)
then \(N_\Omega^{\le}(E)\ge N_\Omega^{<}(E)\), so the Neumann conclusion
also holds with the non-strict counting convention.  We use the strict count
because the Robin comparison in Proposition \ref{prop:strict-robin} is
strict.  For the disk, the explicit Rayleigh--Ritz estimates in Section
\ref{sec:d2} replace the finite computer-assisted step in \cite{FLPS}.  In
eigenvalue form, Theorem \ref{thm:main} is equivalent to
\begin{equation*}
\mu_{n+1}^N(B_R^d)
\le
\frac{4\pi^2n^{2/d}}{(\omega_d|B_R^d|)^{2/d}},
\qquad n\ge1;
\end{equation*}
see Corollary \ref{cor:eigenvalue-form}.  Thus the Weyl-sharp constant gives
a pointwise eigenvalue bound at every index.

On the unit ball, set
\begin{equation*}
\mathcal N_d^{<}(x)
:=
N_{B_1^d}^{<}(x^2),
\qquad
W_d(x)
:=
\frac{x^d}{2^d\Gamma(\frac d2+1)^2}.
\end{equation*}
It is enough to prove \(\mathcal N_d^{<}(x)\ge W_d(x)\) for \(x\ge0\);
the scaling to \(B_R^d\) is carried out in Section
\ref{sec:scaling}.

\subsection{Background and the obstruction}

Pólya formulated the conjecture in 1954 \cite{Polya1954} and later proved
both inequalities for tiling domains, with an additional regularity
assumption in the Neumann case that Kellner removed
\cite{Polya,Kellner}.  The Li--Yau and Kröger inequalities give the sharp
semiclassical constants after one integration in the energy variable
\cite{LiYau,Kroger,FrankLarson}; the counting-function inequality is the
endpoint before this averaging.  A ball does not tile Euclidean space, so
Pólya's argument is unavailable.  Nor does the Neumann result follow from
the Dirichlet theorem: the interlacing inequality
\(\mu_{k+1}^N\le\lambda_k^D\) combines with the Dirichlet bound in the
wrong direction.

There is also a specifically higher-dimensional obstruction.  In the
angular sector \(\ell\), separation of variables gives the radial condition
\begin{equation*}
kJ_{\nu}'(k)-\frac{d-2}{2}J_{\nu}(k)=0,
\qquad
\nu=\ell+\frac{d-2}{2},
\end{equation*}
where \(J_\nu\) is the Bessel function of the first kind.  This is a Dini
condition, reducing to \(J_\nu'(k)=0\) only when \(d=2\).  The
derivative-zero phase estimate used in the disk therefore does not apply to
the physical radial spectrum directly.

Uniform enclosures for Bessel phases were subsequently obtained in
\cite{FLPSPhase}.  Two-term Weyl remainder estimates for balls and
spherical shells \cite{GJWY} imply, for each fixed Neumann ball, the Pólya
lower bound above a sufficiently large frequency; they do not cover every
energy or provide a threshold uniform in the dimension.  Further related
results include the planar convex-domain estimate of
\cite{FilonovConvex}, the Dirichlet theorem for annuli
\cite{FLPSAnnuli}, improvements for balls and cylinders in \cite{GMWZ},
and Pólya-type or quantitative inequalities in other geometries
\cite{FreitasMaoSalavessa,HeWang,JiangLin}.

\subsection{Strategy of the proof}

The argument has three steps.

\paragraph{Strict Robin transfer.}
We place the physical problem in a Robin family whose endpoint has the
derivative-zero condition.  Strict monotonicity of every radial branch
gives
\begin{equation*}
\lambda_{\ell,k}^{N}
<
\bigl(j'_{\nu,k}\bigr)^2
\qquad(\ell\ge0,\ k\ge1).
\end{equation*}
Proposition \ref{prop:strict-robin} proves this comparison variationally.
Its strictness is essential at threshold frequencies and transfers the
Bessel-phase lower bound of \cite[Proposition~3.1,
equation~(3.3)]{FLPS} to \(\mathcal N_d^{<}\).

\paragraph{Discrete--continuous reduction.}
Let \(\Xi_x\) be the decreasing inverse of the Bessel action and let
\(H_d\) denote cumulative angular multiplicity.  Summation by radial level
gives the two expressions
\begin{equation*}
\mathcal N_d^{<}(x)
\ge
P_d(x)
=
\sum_{n\ge1}H_d\!\left(\Xi_x(n-\tfrac34)\right),
\qquad
W_d(x)
=
\frac2{(d-1)!}\int_0^{x/\pi}\Xi_x(t)^{d-1}\,dt.
\end{equation*}
Thus, after the Robin transfer, no further spectral input is needed.  The
problem is to prove that a \(3/4\)-shifted staircase dominates the exact
integral representing the Weyl term.

\paragraph{The critical scale and the frequency ranges.}
For \(d\ge5\), write \(x=ad^{3/2}\).  For the radial levels used below, the
inverse-phase estimate
\begin{equation*}
\Xi_x\!\left(n-\tfrac34\right)\ge x-\alpha_nx^{1/3},
\qquad
\alpha_n=
\left(\frac{n-\frac34}{c_0}\right)^{2/3},
\qquad
c_0=\frac{2\sqrt2}{3\pi}.
\end{equation*}
At this scale the loss is \(O(d^{1/2})\), hence a relative
\(O(d^{-1})\) error.  Raising the angular cutoff to a power comparable to
\(d\) turns each fixed radial level into a nonzero limiting contribution.
Indeed, the finite-level lower bound of \eqref{eq:L.2.5} satisfies, for
fixed \(M\) and \(a\ge4/5\),
\begin{equation*}
L_{d,M}(a)
\xrightarrow[d\to\infty]{}
\frac{\sqrt{2\pi}}{a}
\left(1-\frac1{24a^2}\right)
\sum_{n=1}^{M}\exp\!\left(-\alpha_n a^{-2/3}\right).
\end{equation*}
This limit motivates the decomposition; the proof itself establishes
uniform finite-dimensional inequalities.  Power trials control low
frequencies, finite-level and beta-integral estimates cover the
intermediate range, and a uniform phase estimate treats high frequencies.
The dimensions \(d=2,3,4\) use sharper variants adapted to their angular
multiplicities.  The exact overlapping ranges are recorded in Table
\ref{tab:frequency-ranges} after the common estimates.

\subsection{Further consequences}

Two auxiliary results merit separate mention.  In dimension two, a shifted
floor-sum estimate and explicit
Rayleigh--Ritz bounds replace the computer-assisted finite-frequency step
of \cite{FLPS}; see Section \ref{sec:d2}.  In dimension three, Theorem
\ref{thm:d3-quarter} proves a weighted quarter-shift inequality for the
inverse Bessel action at every \(x\ge5\).  The final section also records
the pointwise eigenvalue bound of Corollary \ref{cor:eigenvalue-form} and a
Neumann-bracketing transfer to domains that tile a ball, including
half-balls and orthant sectors; see Corollary \ref{cor:ball-tiles}.

\subsection{Organization and exact computation}

Section \ref{sec:spectral-reduction} proves the strict Robin transfer and
reduces the theorem to the weighted phase sum.  Section
\ref{sec:common-estimates} develops the estimates used in several
dimensions.  Sections \ref{sec:d2}--\ref{sec:dge5} treat, respectively,
\(d=2\), \(d=3\), \(d=4\), and \(d\ge5\).  Section
\ref{sec:scaling} restores the radius and derives the consequences of the
counting-function inequality.

Apart from standard Bessel identities, the only special-function
inequality imported into the proof is the phase bound of
\cite[Proposition~3.1, equation~(3.3)]{FLPS}.  Every finite calculation for
\(2\le d\le6\) is printed in the paper.  For \(d\ge7\), Proposition
\ref{supp:aggregate-proposition} reduces one compact two-parameter estimate
to finite lists of exact-rational Taylor coefficients and remainder bounds.
The ancillary program verifies these lists using integer and rational
arithmetic.  Decimal output is diagnostic only and is not used in any
inequality.

\section{Spectral reduction}
\label{sec:spectral-reduction}

\subsection{Separation of variables}

Throughout this section assume \(d\ge3\).  We work on the unit ball
and use the frequency variable \(x=\sqrt E\).  Put

\begin{equation}
\delta_d:=\frac{d-2}{2},
\qquad
\nu_{d,\ell}:=\ell+\delta_d.
\label{eq:I.1}
\end{equation}
The multiplicity of the spherical harmonics of degree \(\ell\) is

\begin{equation}
\kappa_{d,\ell}
:=
\binom{\ell+d-1}{d-1}
-
\binom{\ell+d-3}{d-1},
\qquad \ell=0,1,2,\ldots,
\label{eq:I.2}
\end{equation}
where a binomial coefficient is zero when its upper entry is smaller than
its lower entry.  A regular radial solution of positive frequency \(k\) in
this sector is

\begin{equation}
u(r)=r^{-\delta_d}J_{\nu_{d,\ell}}(kr).
\label{eq:I.3}
\end{equation}
Consequently, the Neumann boundary condition is

\begin{equation}
kJ_{\nu_{d,\ell}}'(k)-\delta_d J_{\nu_{d,\ell}}(k)=0.
\label{eq:I.4}
\end{equation}
This is a Dini equation, not \(J_{\nu_{d,\ell}}'(k)=0\), unless \(d=2\).  We use the standard Bessel identities and zero conventions of \cite{DLMF,Watson}.

\begin{proposition}[Strict Robin transfer]
\label{prop:strict-robin}
For every integer \(d\ge3\), every angular degree \(\ell\ge0\), and
every radial index \(k\ge1\), the Neumann eigenvalue is strictly
smaller than the corresponding \(\gamma=\delta_d\) Robin eigenvalue:
\[
\lambda_{\ell,k}^{N}
<
\lambda_{\ell,k}^{\operatorname{Robin}(\delta_d)}
=
(j'_{\nu_{d,\ell},k})^2.
\]
The indexing includes the zero mode when \((\ell,k)=(0,1)\).
\end{proposition}

\begin{proof}

In \(L^2((0,1),r^{d-1}dr)\), use the Friedrichs form domain

\begin{equation*}
\mathfrak V_{d,\ell}:=
\left\{
u\in AC_{\mathrm{loc}}(0,1]:
\int_0^1
\left(
|u'|^2+
\left(1+\frac{\ell(\ell+d-2)}{r^2}\right)|u|^2
\right)r^{d-1}\,dr<\infty
\right\}.
\end{equation*}
This is the Friedrichs realization selecting the regular solution at the
singular endpoint \(r=0\); the trace \(u(1)\) is well defined.  The only
limit-circle case relevant here is \(d=3,\ell=0\).  There the singular
branch behaves as \(u(r)\sim r^{2-d-\ell}=r^{-1}\).  Although it is
square-integrable against \(r^{d-1}dr\), its derivative has infinite
form energy:
\[
\int_0^1 |u'(r)|^2r^2\,dr=\infty.
\]
Thus the form domain again selects the regular branch.  On this common
domain define

\begin{equation}
\begin{aligned}
\mathfrak q_{\ell,\gamma}[u]
:={}&
\int_0^1
\left(
|u'(r)|^2+
\frac{\ell(\ell+d-2)}{r^2}|u(r)|^2
\right)r^{d-1}\,dr
+\gamma |u(1)|^2 .
\end{aligned}
\label{eq:I.5}
\end{equation}
The case \(\gamma=0\) is Neumann.  The case
\(\gamma=\delta_d\) has boundary condition

\begin{equation}
u'(1)+\delta_d u(1)=0,
\label{eq:I.6}
\end{equation}
which becomes \(J_{\nu_{d,\ell}}'(k)=0\) after \eqref{eq:I.3} is substituted.

Every fixed-\(\ell\) radial realization associated with the common closed
form domain has compact resolvent and simple eigenvalues.  Index them increasingly by

\begin{equation*}
0=\lambda_{0,1}(0)<\lambda_{0,2}(0)<\cdots,
\qquad
0<\lambda_{\ell,1}(0)<\lambda_{\ell,2}(0)<\cdots
\quad(\ell\ge1).
\end{equation*}
The forms constitute a type-(B) analytic family in the Robin parameter
\cite[Chapter VII, \S4]{Kato}.
The min--max principle first gives monotonicity in \(\gamma\); simplicity
and the Hellmann--Feynman formula give, for every eigenbranch,

\begin{equation}
\frac{d}{d\gamma}\lambda_{\ell,k}(\gamma)
=
\frac{|u_{\ell,k,\gamma}(1)|^2}
{\displaystyle\int_0^1
|u_{\ell,k,\gamma}(r)|^2r^{d-1}\,dr}
>0.
\label{eq:I.7}
\end{equation}
The numerator cannot vanish: the Robin condition would then also give
\(u'(1)=0\), forcing the ODE solution to vanish identically.

It remains to identify the Robin spectrum.  Since
\(\delta_d>0\), the form \(\mathfrak q_{\ell,\delta_d}\) is nonnegative,
so it has no negative eigenvalue.  At zero energy the regular solution is
\(u(r)=r^\ell\), and
\[
u'(1)+\delta_du(1)=\ell+\delta_d=\nu_{d,\ell}>0.
\]
Hence zero is not a Robin eigenvalue.  The positive Robin eigenvalues are
therefore exactly \(\{(j'_{\nu_{d,\ell},k})^2:k\ge1\}\).  Thus

\begin{equation}
\lambda_{\ell,k}^{N}
<
\lambda_{\ell,k}^{\operatorname{Robin}(\delta_d)}
=
(j'_{\nu_{d,\ell},k})^2.
\label{eq:I.8}
\end{equation}
The indexing includes the zero mode in the \(\ell=0\) Neumann
sector; it lies strictly below the first Robin eigenvalue.
\end{proof}

\subsection{A phase lower bound}

For \(x>0\), define

\begin{equation}
G_x(z)
:=
\frac{\sqrt{x^2-z^2}-z\arccos(z/x)}{\pi},
\qquad 0\le z\le x,
\label{eq:I.9}
\end{equation}
and set \(G_x(z)=0\) for \(z>x\).

\begin{proposition}[Filonov--Levitin--Polterovich--Sher
{\cite[Proposition 3.1, equation (3.3)]{FLPS}}]
\label{prop:flps-phase}
Let \(\nu\ge0\) be real and \(x>0\).  With \(j'_{\nu,k}\) denoting the
\(k\)-th positive zero of \(J_\nu'\), except for the convention
\(j'_{0,1}=0\), one has
\begin{equation}
\#\{k:j'_{\nu,k}\le x\}
\ge
\left\lfloor G_x(\nu)+\frac34\right\rfloor .
\label{eq:I.10}
\end{equation}
For \(\nu>x\), the right-hand side is zero by the extension of \(G_x\).
\end{proposition}

The proposition applies to every integer or half-integer
\(\nu_{d,\ell}\) arising here; when \(d\ge3\), these orders satisfy
\(\nu_{d,\ell}\ge1/2\).

Combining \eqref{eq:I.2}, \eqref{eq:I.8}, and \eqref{eq:I.10} gives a lower
bound directly for the strict Neumann count:

\begin{equation}
\mathcal N_d^{<}(x)
:=
\#\{j:\mu_j^N(B_1^d)<x^2\}
\ge
P_d(x),
\label{eq:I.11}
\end{equation}
where

\begin{equation}
P_d(x)
:=
\sum_{\ell\ge0}\kappa_{d,\ell}
\left\lfloor
G_x(\ell+\delta_d)+\frac34
\right\rfloor .
\label{eq:I.12}
\end{equation}
If a derivative zero equals \(x\), the corresponding Neumann
eigenfrequency is already strictly below \(x\).  Thus \eqref{eq:I.11}
remains valid for the strict counting function.

\subsection{The discrete level sum and the Weyl term}

Let \(\Xi_x:[0,x/\pi]\to[0,x]\) be the decreasing inverse of \(G_x\), and
extend it by zero for \(t\ge x/\pi\).  Set

\begin{equation}
t_n=n-\frac34,\qquad n=1,2,\ldots .
\label{eq:I.13}
\end{equation}
Then

\begin{equation}
P_d(x)=\sum_{n\ge1}H_d(\Xi_x(t_n)),
\label{eq:I.14}
\end{equation}
where

\begin{equation}
H_d(r):=\sum_{\ell+\delta_d\le r}\kappa_{d,\ell}.
\label{eq:I.15}
\end{equation}
Indeed,
\[
\left\lfloor G_x(\nu)+\frac34\right\rfloor
=
\sum_{n\ge1}\mathbf 1_{\{t_n\le G_x(\nu)\}},
\qquad
G_x(\nu)\ge t\ \Longleftrightarrow\ \nu\le\Xi_x(t).
\]
Summing first in \(\nu=\ell+\delta_d\) gives \eqref{eq:I.14}.  Both
inequalities are weak, so equality points are counted on both sides.
Terms with \(t_n>x/\pi\) vanish because \(\Xi_x(t_n)=0\) and
\(H_d(0)=0\) for \(d\ge3\).

If
\(L=\lfloor r-\delta_d\rfloor\ge0\), telescoping \eqref{eq:I.2} gives

\begin{equation}
H_d(r)
=
\binom{L+d-1}{d-1}
+
\binom{L+d-2}{d-1}.
\label{eq:I.16}
\end{equation}
For \(r<\delta_d\), \(H_d(r)=0\).

For later use, denote the integer values of this function by
\begin{equation}
\mathsf A_d(m):=\sum_{\ell=0}^{m}\kappa_{d,\ell}
=
\binom{m+d-1}{d-1}
+
\binom{m+d-2}{d-1}
=H_d(m+\delta_d).
\label{eq:L.1.3}
\end{equation}
The target Weyl term in frequency variables is

\begin{equation}
W_d(x)
:=
\frac{x^d}{2^d\Gamma(d/2+1)^2}.
\label{eq:I.17}
\end{equation}
It has the exact inverse-action representation

\begin{equation}
W_d(x)
=
\frac2{(d-1)!}
\int_0^{x/\pi}\Xi_x(t)^{d-1}\,dt.
\label{eq:I.18}
\end{equation}
The high-frequency estimate below proves that the discrete sum \eqref{eq:I.14}
strictly exceeds the integral \eqref{eq:I.18} above an explicit threshold.
Section \ref{sec:d3} proves the sharper all-frequency comparison when \(d=3\).
From this point onward, the proof concerns only the weighted phase sum
\(P_d\); no further spectral input is required.

\section{Common estimates for the weighted phase sum}
\label{sec:common-estimates}

\subsection{The phase function and its moments}

For later use we collect the elementary identities for the action in one
place.  Direct differentiation of \eqref{eq:I.9} gives
\begin{equation}
G_x'(z)
=-\frac1\pi\arccos\frac zx,
\qquad
G_x''(z)
=\frac1{\pi\sqrt{x^2-z^2}}>0
\quad(0<z<x).
\label{eq:D3.5}
\end{equation}
Consequently, \(G_x\) is strictly decreasing, convex, and
\(1/2\)-Lipschitz.  It is also increasing in the scale \(x\): for fixed
\(z<x\),

\begin{equation}
\frac{\partial G_x(z)}{\partial x}
=
\frac{\sqrt{x^2-z^2}}{\pi x}>0.
\label{eq:D3.6}
\end{equation}

We shall use the two exact integrals

\begin{equation}
\int_0^xG_x(z)\,dz=\frac{x^2}{8},
\qquad
\int_0^x2zG_x(z)\,dz=\frac{2x^3}{9\pi}.
\label{eq:D3.7}
\end{equation}
For the first identity, scale \(z=xu\), use
\(\int_0^1\sqrt{1-u^2}\,du=\pi/4\), and integrate
\(\int_0^1u\arccos u\,du\) by parts.  For the second identity,

\begin{equation}
\int_0^1u\sqrt{1-u^2}\,du=\frac13
\label{eq:D3.8}
\end{equation}
and integration by parts gives

\begin{equation}
\int_0^1u^2\arccos u\,du
=
\frac13\int_0^1\frac{u^3}{\sqrt{1-u^2}}\,du
=\frac29.
\label{eq:D3.9}
\end{equation}
Substitution in \eqref{eq:I.9} proves \eqref{eq:D3.7}.

For \(d\ge3\), the \(\ell=0\) order is
\(\nu_{d,0}=\delta_d>0\).  The convention \(j'_{0,1}=0\) is used only in
the disk section.

\subsection{Power trials}

The same elementary trial functions will be used in dimensions four and
higher.  Let \(L_\ell=\ell(\ell+d-2)\).

\begin{lemma}[Power-trial bound]
\label{lem:power-trial}
For \(d\ge2\) and \(\ell\ge1\), set
\[
s_\ell=\left(\frac{L_\ell}{2}\right)^{1/3},
\qquad
Q_d(L_\ell)
=\frac{s_\ell^2(2s_\ell+1)(2s_\ell+d)}{2s_\ell+d-2}.
\]
Then the first Neumann eigenvalue in the degree-\(\ell\) sector satisfies
\[
\lambda_{\ell,1}^N<Q_d(L_\ell).
\]
Moreover, \(Q_d(L)\) is strictly increasing for \(L>0\).
\end{lemma}

\begin{proof}
For a normalized spherical harmonic \(Y_{\ell,j}\), radial integration gives
\[
\|r^\gamma Y_{\ell,j}\|_2^2=\frac1{2\gamma+d},
\qquad
\|\nabla(r^\gamma Y_{\ell,j})\|_2^2
=\frac{\gamma^2+L_\ell}{2\gamma+d-2}.
\]
Thus the Rayleigh quotient at \(\gamma=s_\ell\) is \(Q_d(L_\ell)\).
Equality in the variational bound would make the trial function a Neumann
eigenfunction, but its normal derivative at \(r=1\) is
\(s_\ell Y_{\ell,j}\ne0\).  The inequality is therefore strict.

For \(L=2s^3\), logarithmic differentiation gives
\[
\frac{d\log Q_d}{d\log L}
=\frac13\left(
2+\frac{2s}{2s+1}+\frac{2s}{2s+d}
-\frac{2s}{2s+d-2}\right)>0,
\]
which proves the monotonicity.
\end{proof}

\subsection{A shifted floor-sum estimate}

Used in dimensions two and three, the following \(3/4\)-shifted convex
floor lemma sharpens \cite[Theorem 6.1]{FLPS}: the same block decomposition
retains the needed integral over \([0,M]\) and the term \(M/4\).

\begin{lemma}[Convex \(3/4\)-shifted floor sum]
\label{lem:retained-floor}

Let \(f:[0,b]\to[0,\infty)\) be strictly decreasing, convex, and
\(1/2\)-Lipschitz, with

\begin{equation}
f(b)=0,
\qquad
f(0)\ge\frac14.
\label{eq:D3.25}
\end{equation}
Let

\begin{equation}
M=\left\lfloor f^{-1}\!\left(\frac14\right)\right\rfloor+1,
\label{eq:D3.26}
\end{equation}
and assume \(M\le b\).  Then

\begin{equation}
\sum_{r=0}^{\lfloor b\rfloor}
\left\lfloor f(r)+\frac34\right\rfloor
\ge
\int_0^M f(s)\,ds+\frac M4.
\label{eq:D3.27}
\end{equation}
\end{lemma}

Appendix \ref{supp:retained-tail-proof} gives the block-and-telescoping proof.

For the finite-level estimates, put

\begin{equation*}
c_0:=\frac{2\sqrt2}{3\pi},
\qquad
\alpha_n:=\left(\frac{n-\frac34}{c_0}\right)^{2/3}
\quad(n\ge1).
\end{equation*}
Appendix \ref{supp:constant-proofs} collects the rational enclosures for
\(\alpha_1,\ldots,\alpha_4\), \(\pi\), and the required radicals.

\subsection{A uniform high-frequency estimate}

Put

\begin{equation}
p:=d-1\ge2,
\qquad
A_p:=\frac{p(p-1)(p-2)}{24}.
\label{eq:II.1}
\end{equation}
We prove

\begin{equation}
P_d(x)>W_d(x)
\quad\text{for}\quad
x\ge X_p:=(24p)^{3/2}.
\label{eq:II.2}
\end{equation}

\subsubsection{A continuous lower envelope for the angular multiplicity}

In \eqref{eq:I.16}, put

\begin{equation}
w:=L+\frac p2.
\label{eq:II.3}
\end{equation}

If \(p=2m\), direct factorization gives

\begin{equation}
p!H_d(r)
=
2w^2\prod_{j=1}^{m-1}(w^2-j^2).
\label{eq:II.4}
\end{equation}

If \(p=2m+1\), it gives
\nopagebreak

\begin{equation}
p!H_d(r)
=
2w\prod_{j=1}^{m}
\left(w^2-\left(j-\frac12\right)^2\right).
\label{eq:II.5}
\end{equation}
The empty product in \eqref{eq:II.4} covers \(p=2\), when \(H_3(r)=w^2\).
The sums of the squared offsets are exactly

\begin{equation}
\sum_{j=1}^{m-1}j^2=A_p
\quad(p=2m),
\qquad
\sum_{j=1}^{m}\left(j-\frac12\right)^2=A_p
\quad(p=2m+1).
\label{eq:II.6}
\end{equation}
Every factor in \eqref{eq:II.4}--\eqref{eq:II.5} is nonnegative.  Applying

\begin{equation}
\prod_i(1-u_i)\ge1-\sum_i u_i,
\qquad 0\le u_i\le1,
\label{eq:II.7}
\end{equation}
therefore yields

\begin{equation}
H_d(r)\ge\frac2{p!}(w^p-A_pw^{p-2}).
\label{eq:II.8}
\end{equation}
For \(s\ge0\), define

\begin{equation}
\psi_p(s):=
\begin{cases}
s^2,&p=2,\\
\max\{0,s^p-A_ps^{p-2}\},&p\ge3,
\end{cases}
\qquad
\phi_p(r):=\psi_p((r-\tfrac12)_+).
\label{eq:II.9}
\end{equation}
The function \(\psi_p\) is nondecreasing and convex.  For \(p\ge3\), on
the part where its polynomial branch is nonnegative,

\begin{equation}
\begin{aligned}
\frac d{ds}(s^p-A_ps^{p-2})
&=s^{p-3}\bigl(ps^2-(p-2)A_p\bigr)>0,\\
\frac {d^2}{ds^2}(s^p-A_ps^{p-2})
&=s^{p-4}
\bigl(p(p-1)s^2-(p-2)(p-3)A_p\bigr)>0.
\end{aligned}
\label{eq:II.10}
\end{equation}
At the transition \(s^2=A_p\), the right derivative is positive, so
taking the positive part preserves convexity.

Writing \(r=L+\delta_d+u\), \(0\le u<1\), gives

\begin{equation}
r-\frac12=w-1+u<w.
\label{eq:II.11}
\end{equation}
Monotonicity of \(\psi_p\), \eqref{eq:II.8}, and \(H_d\ge0\) imply the envelope

\begin{equation}
H_d(r)\ge\frac2{p!}\phi_p(r).
\label{eq:II.12}
\end{equation}
This also holds for \(r<\delta_d\).  Indeed, both sides vanish for \(p=2\);
for \(p\ge3\),

\begin{equation}
(r-\tfrac12)_+<\frac{p-2}{2}\le\sqrt{A_p},
\label{eq:II.13}
\end{equation}
where

\begin{equation*}
A_p-\frac{(p-2)^2}{4}
=
\frac{(p-2)(p-3)(p-4)}{24}\ge0.
\end{equation*}
Using \eqref{eq:I.14}, we have reduced the problem to

\begin{equation}
P_d(x)
\ge
\frac2{p!}
\sum_{k\ge0}\phi_p(\Xi_x(k+\tfrac14)).
\label{eq:II.14}
\end{equation}

\subsubsection{Quarter-shift quadrature}

\begin{lemma}[Retained-node quadrature]
\label{lem:retained-node-quadrature}
Let \(f:[0,\infty)\to[0,\infty)\) be decreasing, convex, and compactly
supported.  For every integer \(K\ge0\),
\begin{equation}
\sum_{j\ge0}f\left(j+\frac14\right)
\ge
\sum_{j=0}^{K-1}f\left(j+\frac14\right)
+\int_{K+1/4}^{\infty}f(t)\,dt
+\frac12f\left(K+\frac14\right),
\label{eq:retained-node-quadrature}
\end{equation}
where the finite sum is empty when \(K=0\).
\end{lemma}

\begin{proof}
Apply the trapezoidal rule to the tail mesh
\(K+1/4,K+5/4,K+9/4,\ldots\).  Compact support gives
\[
\int_{K+1/4}^{\infty}f(t)\,dt
\le
\frac12f\left(K+\frac14\right)
+\sum_{j\ge K+1}f\left(j+\frac14\right).
\]
Adding the first \(K\) nodes and rearranging proves
\eqref{eq:retained-node-quadrature}.
\end{proof}

The case without retained nodes gives the form used in the high-frequency
argument.

\begin{corollary}[Quarter-shift quadrature]
\label{lem:quarter-quadrature}
Under the hypotheses of Lemma \ref{lem:retained-node-quadrature},
\begin{equation}
\sum_{k\ge0}f(k+\tfrac14)
\ge
\int_0^\infty f(t)\,dt
+\frac12f(\tfrac14)-\frac14f(0).
\label{eq:II.15}
\end{equation}
\end{corollary}

\begin{proof}
Take \(K=0\) in \eqref{eq:retained-node-quadrature} and use
\(\int_0^{1/4}f(t)\,dt\le f(0)/4\).
\end{proof}

By \eqref{eq:D3.5}, the decreasing inverse of \(G_x\) satisfies

\begin{equation}
\Xi_x''(t)
=
-\frac{G_x''(\Xi_x(t))}{G_x'(\Xi_x(t))^3}>0.
\label{eq:II.18}
\end{equation}
The zero extension of \(\Xi_x\) is still convex: its derivative jumps upward
from \(-2\) to \(0\) at \(x/\pi\).  Since \(\phi_p\) is nondecreasing and
convex,

\begin{equation}
f_x(t):=\phi_p(\Xi_x(t))
\label{eq:II.19}
\end{equation}
is nonnegative, decreasing, convex, and compactly supported.  Corollary
\ref{lem:quarter-quadrature}
therefore gives

\begin{equation}
\begin{aligned}
\sum_{k\ge0}\phi_p(\Xi_x(k+\tfrac14))
\ge{}&
\int_0^\infty\phi_p(\Xi_x(t))\,dt\\
&+\frac12\phi_p(\Xi_x(\tfrac14))
-\frac14\phi_p(x).
\end{aligned}
\label{eq:II.20}
\end{equation}

\subsubsection{Exact inverse moments}

For \(q\ge0\), put

\begin{equation}
I_q(x):=\int_0^{x/\pi} \Xi_x(t)^q\,dt.
\label{eq:II.21}
\end{equation}
Changing variables \(t=G_x(r)\), integrating by parts, and then using the
beta integral gives

\begin{equation}
\begin{aligned}
I_q(x)
&=\frac1\pi\int_0^x r^q\arccos(r/x)\,dr\\
&=C_qx^{q+1},
\end{aligned}
\label{eq:II.22}
\end{equation}
where

\begin{equation}
C_q
:=
\frac{\Gamma((q+2)/2)}
{2(q+1)\sqrt\pi\,\Gamma((q+3)/2)}.
\label{eq:II.23}
\end{equation}
The Gamma duplication formula gives

\begin{equation}
\frac2{p!}I_p(x)
=
\frac{x^{p+1}}
{2^{p+1}\Gamma((p+3)/2)^2}
=W_d(x).
\label{eq:II.24}
\end{equation}

We use two coefficient estimates.  First,

\begin{equation}
\frac p2C_{p-1}
=
\frac1{4\pi}
B\!\left(\frac{p+1}{2},\frac12\right)
<
\frac1{2\pi}
<
\frac16.
\label{eq:II.25}
\end{equation}
The beta integral is strictly smaller than
\(\int_0^1(1-t)^{-1/2}dt=2\), and \(\pi>3\).

Second, for \(q\ge1\),

\begin{equation}
C_q<q^{-3/2}.
\label{eq:II.26}
\end{equation}
Indeed,
\(\arccos u\le(\pi/\sqrt2)\sqrt{1-u}\) on \([0,1]\), whence

\begin{equation}
C_q
\le
\frac1{\sqrt2}B(q+1,\tfrac32)
\le
\frac1{\sqrt2}
\int_0^\infty e^{-qv}v^{1/2}\,dv
=
\frac{\sqrt\pi}{2\sqrt2\,q^{3/2}}
<
q^{-3/2}.
\label{eq:II.27}
\end{equation}
Here we used \((1-v)^q\le e^{-qv}\).
For the asserted arccosine bound, write \(u=\cos\theta\) and use
\(\theta\le\pi\sin(\theta/2)
=(\pi/\sqrt2)\sqrt{1-u}\) for \(0\le\theta\le\pi/2\).

\subsubsection{The first quarter level}

For \(0\le h\le x\), \eqref{eq:D3.5} and
\(\arccos(1-v)\ge\sqrt{2v}\) give

\begin{equation}
\begin{aligned}
G_x(x-h)
&=
\frac1\pi\int_{x-h}^x\arccos(s/x)\,ds\\
&\ge
\frac{2\sqrt2}{3\pi\sqrt x}h^{3/2}.
\end{aligned}
\label{eq:II.28}
\end{equation}
At \(h=x^{1/3}\), the right side is
\(2\sqrt2/(3\pi)>1/4\).  The elementary bounds
\(\sqrt2>7/5\) and \(\pi<22/7\) prove the last inequality.  Since
\(G_x\) decreases,

\begin{equation}
\Xi_x(\tfrac14)\ge x-x^{1/3}
\qquad(x\ge1).
\label{eq:II.29}
\end{equation}

\subsubsection{Estimating the remainder}

For \(s=(r-\tfrac12)_+\), the mean-value theorem gives
\(r^p-s^p\le(p/2)r^{p-1}\).  For \(p\ge3\),
\(\phi_p(r)\ge s^p-A_ps^{p-2}\) on both branches and
\(s^{p-2}\le r^{p-2}\); hence

\begin{equation}
r^p-\phi_p(r)
\le
\frac p2r^{p-1}+A_pr^{p-2},
\label{eq:II.30}
\end{equation}
where the \(A_p\)-term is absent for \(p=2\), for which the first
inequality suffices.  After integration,

\begin{equation}
I_p(x)
-
\int_0^\infty\phi_p(\Xi_x(t))\,dt
\le
\frac p2I_{p-1}(x)+A_pI_{p-2}(x).
\label{eq:II.31}
\end{equation}

Assume for the moment that

\begin{equation}
x-x^{1/3}-\frac12>\sqrt{A_p},
\label{eq:II.32}
\end{equation}
and put

\begin{equation*}
b:=x^{1/3}+\frac12.
\end{equation*}
By \eqref{eq:II.29},
\(\Xi_x(\tfrac14)-\tfrac12\ge x-b>\sqrt{A_p}\), so \(\psi_p\) is on its
polynomial branch.  Convexity of \(s\mapsto s^p\) yields

\begin{equation}
\phi_p(\Xi_x(\tfrac14))
\ge
x^p-px^{p-1}\left(x^{1/3}+\frac12\right)-A_px^{p-2}.
\label{eq:II.33}
\end{equation}
Also \(\phi_p(x)\le x^p\).  Set

\begin{equation}
S_p(x):=\sum_{k\ge0}\phi_p(\Xi_x(k+\tfrac14)).
\label{eq:II.34}
\end{equation}
Combining \eqref{eq:II.20}, \eqref{eq:II.22}, \eqref{eq:II.31}, and \eqref{eq:II.33} gives

\begin{equation}
\begin{aligned}
S_p(x)-I_p(x)
\ge{}&
\left(\frac14-\frac p2C_{p-1}\right)x^p\\
&-\frac p2x^{p-2/3}
-\frac p4x^{p-1}\\
&-\frac{A_p}{2}x^{p-2}
-A_pC_{p-2}x^{p-1}.
\end{aligned}
\label{eq:II.35}
\end{equation}
Define the positive coefficient

\begin{equation}
a_p:=\frac14-\frac p2C_{p-1}>0.
\label{eq:II.36}
\end{equation}
Equation \eqref{eq:II.35} gives \(S_p(x)>I_p(x)\) under the one-variable
condition

\begin{equation}
a_p>
\frac{p}{2x^{2/3}}
+
\frac{p/4+A_pC_{p-2}}{x}
+
\frac{A_p}{2x^2}.
\label{eq:II.37}
\end{equation}
Condition \eqref{eq:II.37} also implies \eqref{eq:II.32}.  Since
\(a_p<1/4\), its first right-hand term gives \(x^{2/3}>2p\).  With
\(y=x^{1/3}>\sqrt{2p}\) and \(p\ge3\),

\begin{equation}
x-x^{1/3}-\frac12
=y^3-y-\frac12
>\frac{y^3}{2}
>\sqrt2\,p^{3/2}
>\sqrt{A_p}.
\label{eq:II.38}
\end{equation}
Here \(y>\sqrt{2p}\ge\sqrt6\) gives
\(y^3/2>3y>y+1/2\).  For \(p=2\), it gives \(x^{2/3}>4\) and
\(x-x^{1/3}-1/2>0=\sqrt{A_p}\).  Since the right side of \eqref{eq:II.37} is
strictly decreasing from \(+\infty\) to \(0\), it defines a unique
positive cutoff \(X_p^\sharp\) by equality, and
\nopagebreak

\begin{equation}
P_d(x)>W_d(x)
\qquad(x>X_p^\sharp).
\label{eq:II.39}
\end{equation}

Now assume \(x\ge X_p:=(24p)^{3/2}\).  Then \eqref{eq:II.32} also follows directly:
\(x^{2/3}\ge24p\ge48\) gives
\(x>2x^{1/3}+1\), while

\begin{equation*}
\sqrt{A_p}
\le\frac{p^{3/2}}{\sqrt{24}}
<
\frac{(24p)^{3/2}}2
\le\frac x2.
\end{equation*}
By \eqref{eq:II.25}, the leading coefficient exceeds \(1/12\).  We bound
each of the four losses in \eqref{eq:II.35} by \(x^p/48\).  For the first
three,

\begin{equation}
\frac{p}{2x^{2/3}}\le\frac1{48},
\qquad
\frac{p}{4x}\le\frac1{48},
\qquad
\frac{A_p}{2x^2}\le\frac1{48},
\label{eq:II.40}
\end{equation}
because \(x^{2/3}\ge24p\), \(X_p\ge12p\), and

\begin{equation}
24A_p=p(p-1)(p-2)<X_p^2.
\label{eq:II.41}
\end{equation}

For the fourth loss, when \(p\ge3\), \eqref{eq:II.26} gives

\begin{equation}
\begin{aligned}
48A_pC_{p-2}
&\le
\frac{2p(p-1)}{\sqrt{p-2}}\\
&\le
(24p)^{3/2}
\le x.
\end{aligned}
\label{eq:II.42}
\end{equation}
After division by \(p^{3/2}\), the middle inequality follows from

\begin{equation*}
\frac{p-1}{\sqrt{p(p-2)}}
\le\frac2{\sqrt3}
<
\frac{24^{3/2}}2.
\end{equation*}

When \(p=2\), both \(A_p\)-losses vanish.

Thus the negative terms in \eqref{eq:II.35} have absolute value at most
\(x^p/12\), whereas its leading positive term is strictly larger than
\(x^p/12\).  Therefore

\begin{equation}
S_p(x)>I_p(x).
\label{eq:II.43}
\end{equation}

Finally, \eqref{eq:II.14}, \eqref{eq:II.24}, and \eqref{eq:II.43} prove

\begin{equation}
P_d(x)
\ge
\frac2{p!}S_p(x)
>
\frac2{p!}I_p(x)
=W_d(x)
\qquad(x\ge X_p).
\label{eq:II.44}
\end{equation}
Equation \eqref{eq:II.44} proves \eqref{eq:II.2}, including \(x=X_p\);
together with \eqref{eq:I.11}, it transfers the strict phase bound to
\(\mathcal N_d^{<}\).

\begin{proposition}[High-frequency estimate from \(5d^{3/2}\)]
\label{prop:scalar-tail}
For every integer \(d\ge5\), \eqref{eq:II.37} holds strictly at
\(x=5d^{3/2}\), and consequently

\begin{equation}
\mathcal N_d^{<}(x)>W_d(x)
\qquad(d\ge5,\ x\ge5d^{3/2}).
\label{eq:L.6.2}
\end{equation}
\end{proposition}

\begin{proof}
Proposition \ref{supp:uniform-tail} proves \eqref{eq:II.37} at
\(x=5d^{3/2}\) by induction in the two parity classes of \(d-1\).  Its
right-hand side decreases strictly with \(x\), so \eqref{eq:II.37} holds
for all \(x\ge5d^{3/2}\); then \eqref{eq:II.39} and \eqref{eq:I.11} give
\eqref{eq:L.6.2}.
\end{proof}

\subsection{A lower bound from finitely many radial levels}

With \(c_0\) and \(\alpha_n\) as above, put

\begin{equation*}
\beta_d:=\frac{(d-1)(d-2)(d-3)}{24d^3}.
\end{equation*}
Define the dimension factor

\begin{equation}
\mathcal D_d:=
\frac{2\sqrt\pi\,\Gamma(\frac{d+2}{2})}
{\sqrt d\,\Gamma(\frac{d+1}{2})}.
\label{eq:L.2.D}
\end{equation}

\begin{lemma}[Bounds for the dimension factor]
\label{lem:dimension-factor}
For every integer \(d\ge5\),
\begin{equation}
 \frac{\mathcal D_{d+2}}{\mathcal D_d}
 =\frac{\sqrt{d(d+2)}}{d+1}<1,
 \qquad
 \mathcal D_d>
 \sqrt{2\pi\left(1+\frac1{2d}\right)}>\frac52,
 \qquad
 \mathcal D_d<\frac83.
\label{eq:L.dimension-factor-bounds}
\end{equation}
\end{lemma}

\begin{proof}
The recurrence follows from the Gamma functional equation, and the standard
Gamma-ratio limit gives \(\mathcal D_d\to\sqrt{2\pi}\).  Set
\[
 E_d:=\frac{\mathcal D_d}
 {\sqrt{2\pi(1+1/(2d))}}.
\]
Then
\[
 \left(\frac{E_{d+2}}{E_d}\right)^2
 =1-\frac1{(d+1)^2(2d+5)}<1,
 \qquad E_d\longrightarrow1,
\]
 so \(E_d>1\).  Equation \eqref{eq:Const.pi} gives the second strict
inequality in \eqref{eq:L.dimension-factor-bounds}.  The recurrence reduces
the upper bound to \(d=5,6\); substitution in \eqref{eq:L.2.D}, using
\eqref{eq:Const.pi}--\eqref{eq:Const.rad}, gives
\[
 \mathcal D_5=\frac{15\pi}{8\sqrt5}<\frac83,
 \qquad
 \mathcal D_6=\frac{32}{5\sqrt6}<\frac{196}{75}<\frac83.
\]
\end{proof}

We use the strict spectral reduction, the action \(G_x\), and the
dimension-\(d\) angular envelope from Section \ref{sec:spectral-reduction}.

For \(0\le h\le x\),

\begin{equation*}
G_x(x-h)
=\frac1\pi\int_{x-h}^{x}\arccos(s/x)\,ds
\ge c_0\frac{h^{3/2}}{\sqrt x}.
\end{equation*}
If \(\alpha_nx^{1/3}\le x\), taking \(h=\alpha_nx^{1/3}\) gives

\begin{equation}
\Xi_x\left(n-\frac34\right)
\ge x-\alpha_n x^{1/3}.
\label{eq:L.2.1}
\end{equation}

Define

\begin{equation}
r_{d,n}(a):=
1-\frac{\alpha_n}{a^{2/3}d}
-\frac{1}{2ad^{3/2}},
\label{eq:L.2.3}
\end{equation}
and use the truncated summand

\begin{equation}
T_{d,n}(a):=
\begin{cases}
r_{d,n}(a)^{d-3}
\left(r_{d,n}(a)^2-\dfrac{\beta_d}{a^2}\right),
&
r_{d,n}(a)>\dfrac{\sqrt{\beta_d}}a,\\[1ex]
0,&
r_{d,n}(a)\le\dfrac{\sqrt{\beta_d}}a.
\end{cases}
\label{eq:L.2.4}
\end{equation}
The piecewise definition is essential, since positive-part notation is
ambiguous when \(r_{d,n}<0\) and \(d-3\) is even.

The level-set identity \eqref{eq:I.14}, together with the strict Robin
comparison and the phase estimate, gives

\begin{equation}
\mathcal N_d^{<}(x)
\ge P_d(x)
=
\sum_{n\ge1}
H_d\!\left(\Xi_x(n-\tfrac34)\right).
\label{eq:L.2.layer}
\end{equation}

Now put \(x=ad^{3/2}\).  By \eqref{eq:L.2.1}, the variable \(s=r-\tfrac12\)
in \eqref{eq:II.12}, for the \(n\)-th radial level, is at least

\begin{equation*}
ad^{3/2}-\alpha_n(ad^{3/2})^{1/3}-\frac12
=ad^{3/2}r_{d,n}(a).
\end{equation*}
Moreover

\begin{equation*}
\frac{(d-1)(d-2)(d-3)}{24}
=\beta_dd^3,
\end{equation*}
and the duplication formula gives the exact prefactor identity

\begin{equation*}
\frac{
2\cdot2^d\Gamma(\frac d2+1)^2
}{
(d-1)!\,d^{3/2}
}
=\mathcal D_d.
\end{equation*}
If the \(n\)-th term is nonzero in \eqref{eq:L.2.4}, then

\begin{equation*}
ad^{3/2}r_{d,n}(a)>
d^{3/2}\sqrt{\beta_d}
=
\sqrt{\frac{(d-1)(d-2)(d-3)}{24}}.
\end{equation*}
In particular,
\(\alpha_nx^{1/3}<x-\tfrac12<x\).  Since

\begin{equation*}
n-\frac34=c_0\alpha_n^{3/2}
<c_0x<\frac{x}{\pi}=G_x(0),
\end{equation*}
where \(2\sqrt2/3<1\), the inverse in \eqref{eq:L.2.1} is defined.  The
preceding inequality also shows that the resulting cutoff satisfies the
hypothesis of \eqref{eq:II.12}.

Applying \eqref{eq:II.12} to the nonzero terms and using zero for the others
yields, for every \(M\ge1\),

\begin{equation}
\frac{\mathcal N_d^{<}(a d^{3/2})}
{W_d(a d^{3/2})}
\ge
\frac{P_d(a d^{3/2})}{W_d(a d^{3/2})}
\ge
L_{d,M}(a):=
\frac{\mathcal D_d}{a}\sum_{n=1}^{M}T_{d,n}(a).
\label{eq:L.2.5}
\end{equation}

\subsection{A beta-integral lower bound}

For \(d\ge4\), the following convex quadrature reduction holds under the
nonvanishing conditions \(\mathrm{(A1)}\)--\(\mathrm{(A2)}\) below; it is
used first for \(d=4\) and then uniformly for \(d\ge5\).

For fixed \(d\ge4\) and \(a>0\), put

\begin{equation*}
\rho_0:=1-\frac{1}{2ad^{3/2}},
\qquad
\varkappa:=\frac{1}{d\,a^{2/3}}
\left(\frac{3\pi}{2\sqrt2}\right)^{2/3},
\qquad
q:=\frac{\beta_d}{a^2},
\end{equation*}
and

\begin{equation*}
\rho(t):=\rho_0-\varkappa t^{2/3}.
\end{equation*}
Define
\nopagebreak

\begin{equation}
g(t):=
\begin{cases}
\rho(t)^{d-3}(\rho(t)^2-q),
&\rho(t)>\sqrt q,\\
0,&\rho(t)\le\sqrt q.
\end{cases}
\label{eq:L.5.1}
\end{equation}
The map \(t\mapsto\rho(t)\) is decreasing and convex, while the
positive-root extension of

\begin{equation*}
\rho\longmapsto \rho^{d-3}(\rho^2-q)
\end{equation*}
is nonnegative, nondecreasing, and convex.  Hence \(g\) is nonnegative,
decreasing, convex, and compactly supported.

The definitions of \(\varkappa\) and \(\alpha_n\) give the node identity
\[
g\left(n-\frac34\right)=T_{d,n}(a)\qquad(n\ge1).
\]
Compact support allows \(M\) in \eqref{eq:L.2.5} beyond the final nonzero
node, so
\begin{equation}
\frac{\mathcal N_d^{<}(a d^{3/2})}
{W_d(a d^{3/2})}
\ge
\frac{\mathcal D_d}{a}
\sum_{j\ge0}g\left(j+\frac14\right).
\label{eq:L.full-node-sum}
\end{equation}

Lemma \ref{lem:retained-node-quadrature} with \(K=0\) gives

\begin{equation}
\sum_{j\ge0}g\left(j+\frac14\right)
\ge
\int_{1/4}^{\infty}g(t)\,dt
+\frac12g\left(\frac14\right).
\label{eq:L.5.2}
\end{equation}

The same lemma with \(K=2\) gives
\[
\sum_{j\ge0}g\left(j+\frac14\right)
\ge
g\left(\frac14\right)+g\left(\frac54\right)
+\frac12g\left(\frac94\right)
+\int_{9/4}^{\infty}g(t)\,dt.
\]
Hermite--Hadamard on \([9/4,17/4]\), whose midpoint is \(13/4\),
gives
\[
\int_{9/4}^{17/4}g(t)\,dt
\ge2g\left(\frac{13}{4}\right).
\]
Discarding the remaining nonnegative tail, define
\begin{equation}
\begin{gathered}
 (\omega_{2,n})_{n=1}^4=(1,1,0,0),\qquad
 (\omega_{3,n})_{n=1}^4=(1,1,\tfrac12,0),\qquad
 (\omega_{4,n})_{n=1}^4=(1,1,\tfrac12,2),\\
 \widehat L_{d,M}(a):=
 \frac{\mathcal D_d}{a}\sum_{n=1}^4
 \omega_{M,n}T_{d,n}(a)
 \qquad(M=2,3,4).
\end{gathered}
\label{eq:L.retained-definition}
\end{equation}
Then \eqref{eq:L.full-node-sum} gives the retained-tail bound
\begin{equation}
\frac{\mathcal N_d^{<}(a d^{3/2})}
{W_d(a d^{3/2})}
\ge \widehat L_{d,4}(a)
\ge \widehat L_{d,3}(a)
\ge \widehat L_{d,2}(a).
\label{eq:L.retained-four}
\end{equation}
The weight \(2\) on \(T_{d,4}\) is the Hermite--Hadamard lower bound for the
tail integral over \([9/4,17/4]\).

Convexity on \([0,1/4]\) also gives

\begin{equation*}
\int_0^{1/4}g(t)\,dt
\le
\frac18\left(g(0)+g(1/4)\right).
\end{equation*}
Hence
\nopagebreak

\begin{equation}
\sum_{j\ge0}g\left(j+\frac14\right)
\ge
\int_0^\infty g(t)\,dt
-\frac18g(0)+\frac38g(1/4).
\label{eq:L.5.3}
\end{equation}
Extending the nonzero polynomial to \(0\le\rho\le\rho_0\) decreases its
integral because it is negative on \(0<\rho<\sqrt q\).  Thus, with

\begin{equation}
J_m:=
\frac{3}{2\varkappa^{3/2}}
\rho_0^{m+3/2}B(m+1,3/2),
\label{eq:L.5.4}
\end{equation}

\begin{equation}
\int_0^\infty g(t)\,dt
\ge J_{d-1}-qJ_{d-3}.
\label{eq:L.5.5}
\end{equation}

Let

\begin{equation*}
\rho_1:=\rho_0-\frac{\alpha_1}{a^{2/3}d},
\qquad
f_0:=\rho_0^{d-3}(\rho_0^2-q),
\qquad
f_1:=\rho_1^{d-3}(\rho_1^2-q).
\end{equation*}
The endpoint corrections require

\begin{equation}
\mathrm{(A1)}\quad \rho_0>\sqrt q,
\qquad
\mathrm{(A2)}\quad \rho_1>\sqrt q.
\label{eq:L.beta.active}
\end{equation}
These conditions are verified separately at each application; without
\(\mathrm{(A1)}\), the sign of the \(f_0\)-correction is unavailable.

Equations \eqref{eq:L.2.5} and
\eqref{eq:L.5.3}--\eqref{eq:L.5.5} give

\begin{equation}
\frac{\mathcal N_d^{<}(a d^{3/2})}
{W_d(a d^{3/2})}
\ge \mathfrak B_d(a),
\label{eq:L.5.6}
\end{equation}
where

\begin{equation}
\mathfrak B_d(a):=
\frac{\mathcal D_d}{a}
\left(
J_{d-1}-qJ_{d-3}
-\frac18f_0+\frac38f_1
\right).
\label{eq:L.5.7}
\end{equation}
Define

\begin{equation*}
\mathfrak h_d:=
\frac{\sqrt2\,d\,
\Gamma(\frac{d+2}{2})\Gamma(d)}
{\Gamma(\frac{d+1}{2})\Gamma(d+\frac32)}
\end{equation*}
and

\begin{equation*}
\vartheta_d(a):=
\frac{(d-3)(d^2-\frac14)}
{24d^3a^2\rho_0^2}.
\end{equation*}
Then

\begin{equation}
\mathfrak B_d(a)=
\mathfrak h_d \rho_0^{d+1/2}(1-\vartheta_d(a))
+\frac{\mathcal D_d}{8a}(3f_1-f_0).
\label{eq:L.5.8}
\end{equation}

\begin{table}[H]
\centering
\small
\begin{tabular}{@{}l l p{0.58\linewidth}@{}}
\toprule
Dimension & Junctions & Estimates used\\ \midrule
\(d=2\) & \(x=20/3\) &
low-frequency variational bound \(\to\) disk phase estimate\\
\(d=3\) & \(x=5,14\) &
low-frequency variational bound \(\to\) finitely many radial levels
\(\to\) high-frequency phase estimate\\
\(d=4\) & \(x=38/5,26,42\) &
low-frequency variational bound \(\to\) finitely many radial levels and
beta-integral estimate \(\to\) high-frequency estimate\\
\(d\ge5\) & \(a=a_d^\ast,5/2,5\) &
low-frequency variational bound \(\to\) finite-level estimate
\(\to\) beta-integral estimate \(\to\) high-frequency estimate\\
\bottomrule
\end{tabular}
\caption{Overlapping frequency ranges used in the proof.}
\label{tab:frequency-ranges}
\end{table}
Here \(a_5^\ast=41/50\) and \(a_d^\ast=4/5\) for \(d\ge6\), as proved in
Theorem \ref{thm:low-frequency-trial}.  The ranges overlap at every
junction; in dimension two, the endpoint \(x=20/3\) belongs to the
variational range.

\section{Dimension two}
\label{sec:d2}

Although the disk case was proved in \cite{FLPS}, we give a proof adapted to
strict counting in which explicit Rayleigh--Ritz estimates replace the
computer-assisted finite-frequency step.  Its shifted floor-sum estimate is
reused in dimension three.  Since \(\delta_2=0\), the planar phase transfer
remains non-strict.

\subsection{The unit-disk statement and spectrum}

On the unit disk, the notation of the introduction reduces the target to

\begin{equation}
\mathcal N_2^{<}(x)\geq W_2(x)=\frac{x^2}{4},
\qquad x\geq0.
\label{eq:D2.1.2}
\end{equation}

Separate variables on \(B_1^2\).  For a positive eigenvalue
\(\rho^2\), a regular angular mode of order \(m\geq0\) has radial part
\(J_m(\rho r)\).  The Neumann condition is

\begin{equation*}
J_m'(\rho)=0.
\end{equation*}
The zero eigenvalue is simple: for \(m\geq1\), the regular harmonic solution
\(r^m\) has nonzero normal derivative at \(r=1\).  Since \(J_0'=-J_1\), the
spectrum, as a multiset, is

\begin{equation}
0,\qquad
j_{1,k}^2\quad(k\geq1),\qquad
(j'_{m,k})^2\quad(m,k\geq1).
\label{eq:D2.1.3}
\end{equation}
The radial \(m=0\) branch is simple.  Every \(m\geq1\) branch has angular
weight two, from the sine and cosine modes.  Thus

\begin{equation}
\mathcal N_2^{<}(x)=
\mathbf 1_{\{x>0\}}
+\#\{k:j_{1,k}<x\}
+2\sum_{m\geq1}\#\{k:j'_{m,k}<x\}.
\label{eq:D2.1.4}
\end{equation}
Here and below \(j'_{m,k}\), \(m\geq1\), means the \(k\)-th positive zero of
\(J_m'\).  In particular, the formal equality \(J_m'(0)=0\) for \(m\geq2\)
does not create another zero Neumann eigenmode.

For the phase argument, define also the non-strict count

\begin{equation}
\mathcal N_2^{\le}(x):=\#\{j:\mu_j^N(B_1^2)\leq x^2\}.
\label{eq:D2.1.5}
\end{equation}

\subsection{The high-frequency estimate}

\subsubsection{The Bessel-phase count}

We use the action \(G_x\) defined in \eqref{eq:I.9} and the convention
\[
j'_{0,1}=0,\qquad j'_{0,k+1}=j_{1,k}\quad(k\geq1).
\]
Proposition \ref{prop:flps-phase} therefore applies to all planar angular
orders, including \(\nu=0\).

\begin{proposition}[High-frequency phase estimate for the disk]
\label{prop:disk-tail}
For every \(x>20/3\),
\begin{equation}
{\mathcal N_2^{<}(x)\geq\frac{x^2}{4}.}
\label{eq:D2.2.35}
\end{equation}
\end{proposition}

Appendix \ref{supp:disk-analytic-tail} proves Proposition
\ref{prop:disk-tail}; the finite variational range below supplies \(x=20/3\).
\subsection{Rayleigh--Ritz bounds at low frequency}
\label{sec:disk-ritz-body}

For each angular order \(m\geq0\), let \(\lambda_{m,k}\) denote the
eigenvalues in that radial sector, in increasing order and without angular
multiplicity.  Thus

\begin{equation*}
\lambda_{0,1}=0,\qquad
\lambda_{0,k+1}=j_{1,k}^2\quad(k\geq1),
\qquad
\lambda_{m,k}=(j'_{m,k})^2\quad(m\geq1).
\end{equation*}
The Neumann quadratic form in the order-\(m\) sector is

\begin{equation}
\mathfrak q_m[f]
:=\int_0^1
\left(
|f'(r)|^2+\frac{m^2}{r^2}|f(r)|^2
\right)r\,dr,
\qquad
\|f\|_m^2:=\int_0^1|f(r)|^2r\,dr.
\label{eq:D2.3.1}
\end{equation}
Trial functions need only lie in the form domain; the Neumann condition at
\(r=1\) is the natural boundary condition and is not imposed on the trials.
If \(V=\operatorname{span}\{e_1,\ldots,e_d\}\) is a trial space, put

\begin{equation*}
A_{ij}:=\mathfrak q_m[e_i,e_j],
\qquad
B_{ij}:=\langle e_i,e_j\rangle_m.
\end{equation*}
The Gram matrix \(B\) is positive definite.  If
\(\theta_1\leq\cdots\leq\theta_d\) are the generalized eigenvalues of
\(A v=\theta Bv\), the Rayleigh--Ritz form of the min--max principle gives

\begin{equation}
\lambda_{m,k}\leq\theta_k
\qquad(1\leq k\leq d).
\label{eq:D2.3.2}
\end{equation}
For \(m\geq1\), the single trial \(f(r)=r^m\) gives

\begin{equation*}
\frac{\mathfrak q_m[r^m]}{\|r^m\|_m^2}=2m(m+1).
\end{equation*}
Consequently,

\begin{equation}
(j'_{1,1})^2\leq4,\qquad
(j'_{2,1})^2\leq12,\qquad
(j'_{3,1})^2\leq24.
\label{eq:D2.3.3}
\end{equation}
For \(e_i(r)=r^{p_i}\), where \(p_i=m+2i\) and \(i=0,1,\ldots\), direct
integration gives

\begin{equation}
B_{ij}=\frac1{p_i+p_j+2},
\qquad
A_{ij}=\frac{p_ip_j+m^2}{p_i+p_j}.
\label{eq:D2.3.4}
\end{equation}
In the sole exceptional case \(m=i=j=0\), the second quotient in \eqref{eq:D2.3.4} is
interpreted as \(A_{00}=0\).

\begin{proposition}[Disk Ritz bounds]
\label{prop:disk-ritz}
\begin{equation}
j_{1,1}^2<20,\qquad (j'_{1,2})^2<32,\qquad
(j'_{4,1})^2<32.
\label{eq:D2.Ritz}
\end{equation}
\end{proposition}

Appendix \ref{supp:disk-ritz} contains the proof.  Together with
\eqref{eq:D2.3.3}, these bounds yield the counts below.  Because the Ritz
inequalities are strict, the corresponding modes are counted even when
\(x^2\) equals the trial threshold.

\subsection{Completion of the finite range}

The preceding bounds and the angular weights in \eqref{eq:D2.1.3} give the
following lower counts, each including the constant eigenfunction.

\begingroup
\small
\setlength{\tabcolsep}{2pt}
\begin{tabular}{@{}>{\raggedright\arraybackslash}p{0.19\textwidth}>{\raggedright\arraybackslash}p{0.32\textwidth}>{\raggedright\arraybackslash}p{0.19\textwidth}>{\raggedright\arraybackslash}p{0.19\textwidth}@{}}
\toprule
Frequency & Additional modes below \(x\) & Count & Weyl bound \\
\midrule
\(0<x\leq2\) & none needed & \(1\) & \(x^2/4\leq1\) \\
\(2<x\leq2\sqrt3\) & \(j'_{1,1}\), weight \(2\) & \(3\) & \(x^2/4\leq3\) \\
\(2\sqrt3<x\leq2\sqrt5\) & also \(j'_{2,1}\), weight \(2\) & \(5\) & \(x^2/4\leq5\) \\
\(2\sqrt5<x\leq2\sqrt6\) & also \(j_{1,1}\), weight \(1\) & \(6\) & \(x^2/4\leq6\) \\
\(2\sqrt6<x\leq2\sqrt8\) & also \(j'_{3,1}\), weight \(2\) & \(8\) & \(x^2/4\leq8\) \\
\(2\sqrt8<x\leq2\sqrt{12}\) & also \(j'_{4,1},j'_{1,2}\), total weight \(4\) & \(12\) & \(x^2/4\leq12\) \\
\bottomrule
\end{tabular}
\endgroup

At the possible equality thresholds \(x=2,2\sqrt3,2\sqrt6\), the previously
counted \(1,3,6\) modes already match \(x^2/4\); hence every threshold is
valid for the strict count.  At \(x=0\), both sides vanish.  Therefore

\begin{equation}
\mathcal N_2^{<}(x)\geq\frac{x^2}{4}
\quad\text{for }0\leq x\leq2\sqrt{12}.
\label{eq:D2.4.1}
\end{equation}

Finally,
\nopagebreak

\begin{equation}
\left(\frac{20}{3}\right)^2
=\frac{400}{9}<48=(2\sqrt{12})^2.
\label{eq:D2.4.2}
\end{equation}
Thus the finite range \eqref{eq:D2.4.1} contains the endpoint \(x=20/3\) and overlaps
the strict high-frequency estimate \eqref{eq:D2.2.35}.  Together they prove
\eqref{eq:D2.1.2} for every
\(x\geq0\).

\section{Dimension three}
\label{sec:d3}

\subsection{A weighted quarter-shift inequality}

Let \(G_x\) and its decreasing inverse \(\Xi_x\) be as in
\eqref{eq:I.9} and before \eqref{eq:I.13}; in particular,
\(\Xi_x(0)=x\) and \(\Xi_x(x/\pi)=0\).

Define

\begin{equation}
q_\ell(x)
:=
\left\lfloor
G_x\!\left(\ell+\frac12\right)+\frac34
\right\rfloor,
\qquad \ell=0,1,\ldots,
\label{eq:D3.10}
\end{equation}
where \(G_x(z)\) is understood to be zero when \(z\ge x\).  This extension
does not add any term because \(\lfloor 3/4\rfloor=0\).

For a nonnegative \(r\), if \(m=\lfloor r+1/2\rfloor\), then

\begin{equation}
m^2=\sum_{\ell=0}^{m-1}(2\ell+1),
\label{eq:D3.11}
\end{equation}
and

\begin{equation}
\ell\le m-1
\quad\Longleftrightarrow\quad
\ell+\frac12\le r.
\label{eq:D3.12}
\end{equation}
At equality, both sides of \eqref{eq:D3.12} are true.  Taking
\(r=\Xi_x(n-3/4)\), using the inverse relation

\begin{equation}
\ell+\frac12\le \Xi_x(t)
\quad\Longleftrightarrow\quad
t\le G_x\!\left(\ell+\frac12\right),
\label{eq:D3.13}
\end{equation}
and interchanging two finite nonnegative sums yields the exact identity

\begin{equation}
\begin{aligned}
\sum_{\substack{n\ge1\\ n-3/4\le x/\pi}}
\left\lfloor
\Xi_x\!\left(n-\frac34\right)+\frac12
\right\rfloor^2
&=
\sum_{\ell\ge0}(2\ell+1)
\#\left\{n\ge1:
n-\frac34\le G_x\!\left(\ell+\frac12\right)
\right\}\\
&=
{\sum_{\ell\ge0}(2\ell+1)q_\ell(x)}.
\end{aligned}
\label{eq:D3.14}
\end{equation}
Indeed, for \(y\ge0\),

\begin{equation}
\#\{n\ge1:n-3/4\le y\}
=\lfloor y+3/4\rfloor.
\label{eq:D3.15}
\end{equation}
Finally, Tonelli's theorem and \eqref{eq:D3.13} give the corresponding
integral representation

\begin{equation}
\begin{aligned}
\int_0^{x/\pi}\Xi_x(t)^2\,dt
&=
\int_0^{x/\pi}\int_0^{\Xi_x(t)}2z\,dz\,dt\\
&=
\int_0^x2z
\left|\{t\in[0,x/\pi]:t\le G_x(z)\}\right|\,dz\\
&=
\int_0^x2zG_x(z)\,dz
=\frac{2x^3}{9\pi}.
\end{aligned}
\label{eq:D3.16}
\end{equation}
Thus the desired quarter-shift inequality is equivalent to

\begin{equation}
\sum_{\ell\ge0}(2\ell+1)q_\ell(x)
\ge
\int_0^x2zG_x(z)\,dz.
\label{eq:D3.17}
\end{equation}

\begin{lemma}[Weighted unit-cell bound]
\label{lem:unit-cell}
If \(H:[0,1]\to[0,\infty)\) is decreasing, then
\begin{equation}
\int_0^1 2sH(s)\,ds\le\int_0^1H(s)\,ds.
\label{eq:D3.65}
\end{equation}
\end{lemma}

\begin{proof}
Pair \(s\) with \(1-s\):
\begin{equation}
\begin{aligned}
\int_0^1(2s-1)H(s)\,ds
&=
\int_0^{1/2}
(1-2s)\bigl(H(1-s)-H(s)\bigr)\,ds\\
&\le0,
\end{aligned}
\label{eq:D3.66}
\end{equation}
because \(1-2s\ge0\) and \(H(1-s)\le H(s)\).
\end{proof}

\begin{lemma}[Joint quarter-shift estimate]
\label{lem:d3-joint}
For every \(x\ge14\), the weighted angular count
satisfies
\begin{equation}
\sum_{\ell\ge0}(2\ell+1)q_\ell(x)
\ge \frac{2x^3}{9\pi}.
\label{eq:D3.joint}
\end{equation}
\end{lemma}

\begin{proof}
Fix \(x\ge14\), abbreviate \(g=G_x\), and let

\begin{equation}
\xi=\Xi_x\!\left(\frac14\right),
\qquad
m=\left\lfloor \xi+\frac12\right\rfloor,
\qquad
B=m+\frac12.
\label{eq:D3.46}
\end{equation}
Thus

\begin{equation}
m-\frac12\le\xi<m+\frac12.
\label{eq:D3.47}
\end{equation}
It follows from \eqref{eq:D3.10} and \eqref{eq:D3.47} that

\begin{equation}
q_\ell(x)=0\qquad(\ell\ge m),
\label{eq:D3.48}
\end{equation}
where equality at the threshold \(G_x(m-1/2)=1/4\), if it occurs, is
included in \(q_{m-1}\).

We first verify the endpoint hypothesis for the shifted applications of
Lemma \ref{lem:retained-floor}.  From

\begin{equation}
G_x(z)=\frac1\pi\int_z^x\arccos\frac{s}{x}\,ds
\label{eq:D3.49}
\end{equation}
one has, for \(x\ge5\),

\begin{equation}
G_x(x-1)
\le
\frac1\pi\arccos\!\left(1-\frac1x\right)
\le
\frac1\pi\arccos\frac45
<
\frac14.
\label{eq:D3.50}
\end{equation}
The last inequality follows from

\begin{equation}
\frac45>\frac1{\sqrt2}=\cos\frac\pi4.
\label{eq:D3.51}
\end{equation}
As \(G_x\) is decreasing, \eqref{eq:D3.50} gives

\begin{equation}
\xi<x-1,
\qquad
B=m+\frac12<x.
\label{eq:D3.52}
\end{equation}
For \(0\le j\le m-1\), set

\begin{equation}
h_j(s)=g\!\left(j+\frac12+s\right),
\qquad
0\le s\le b_j:=x-j-\frac12.
\label{eq:D3.53}
\end{equation}
The function \(h_j\) is strictly decreasing, convex, and
\(1/2\)-Lipschitz, and \(h_j(b_j)=0\).  Moreover,
\nopagebreak

\begin{equation}
j+\frac12\le m-\frac12\le\xi,
\label{eq:D3.54}
\end{equation}
so \(h_j(0)\ge1/4\).  Write

\begin{equation}
\xi+\frac12=m+\alpha,
\qquad 0\le\alpha<1.
\label{eq:D3.55}
\end{equation}
Then the integer in \eqref{eq:D3.26} is exactly

\begin{equation}
\begin{aligned}
\left\lfloor
h_j^{-1}\!\left(\frac14\right)
\right\rfloor+1
&=
\left\lfloor \xi-j-\frac12\right\rfloor+1\\
&=
\left\lfloor m-j-1+\alpha\right\rfloor+1
=m-j.
\end{aligned}
\label{eq:D3.56}
\end{equation}
This remains true when \(\alpha=0\).  Finally, \eqref{eq:D3.52} gives

\begin{equation}
m-j<x-j-\frac12=b_j,
\label{eq:D3.57}
\end{equation}
so all hypotheses of Lemma \ref{lem:retained-floor} hold.

Define the remaining sums

\begin{equation}
R_j:=\sum_{\ell\ge j}q_\ell(x).
\label{eq:D3.58}
\end{equation}
Lemma \ref{lem:retained-floor} and \eqref{eq:D3.56} give

\begin{equation}
R_j
\ge
\int_{j+1/2}^{m+1/2}g(r)\,dr+\frac{m-j}{4}
\qquad(0\le j\le m-1).
\label{eq:D3.59}
\end{equation}
Since \eqref{eq:D3.48} makes every sum finite,

\begin{equation}
\sum_{\ell\ge0}(2\ell+1)q_\ell
=
R_0+2\sum_{j=1}^{m-1}R_j.
\label{eq:D3.60}
\end{equation}
Here \(q_\ell\) occurs once in \(R_0\) and twice in
\(R_1,\ldots,R_\ell\), giving \(2\ell+1\).  Summing
\eqref{eq:D3.59}, the quarter-shift terms satisfy

\begin{equation}
\frac m4
+2\sum_{j=1}^{m-1}\frac{m-j}{4}
=
\frac m4+\frac12\sum_{r=1}^{m-1}r
=\frac{m^2}{4}.
\label{eq:D3.61}
\end{equation}
The integral coefficients are constant on successive unit intervals.
Thus \eqref{eq:D3.59}--\eqref{eq:D3.61} yield

\begin{equation}
\sum_{\ell\ge0}(2\ell+1)q_\ell(x)
\ge
\frac{m^2}{4}
+
\sum_{k=0}^{m-1}(2k+1)
\int_{k+1/2}^{k+3/2}g(r)\,dr.
\label{eq:D3.62}
\end{equation}
The joint term \(m^2/4\) will absorb the \(x^2/8\) unit-cell loss.

Let

\begin{equation}
I(x):=\int_0^x2rg(r)\,dr.
\label{eq:D3.63}
\end{equation}
The loss between \(I(x)\) and the integral part on the right of \eqref{eq:D3.62} is

\begin{equation}
\begin{aligned}
\mathcal L={}&
\int_0^{1/2}2rg(r)\,dr\\
&+
\sum_{k=0}^{m-1}
\int_{k+1/2}^{k+3/2}
2\left(r-k-\frac12\right)g(r)\,dr\\
&+
\int_B^x2rg(r)\,dr.
\end{aligned}
\label{eq:D3.64}
\end{equation}

Since \(2r\le1\) on \([0,1/2]\), Lemma \ref{lem:unit-cell} applied on
each unit cell gives

\begin{equation}
\begin{aligned}
&\int_0^{1/2}2rg(r)\,dr
+
\sum_{k=0}^{m-1}
\int_{k+1/2}^{k+3/2}
2\left(r-k-\frac12\right)g(r)\,dr\\
&\qquad\le
\int_0^{1/2}g(r)\,dr
+
\sum_{k=0}^{m-1}\int_{k+1/2}^{k+3/2}g(r)\,dr\\
&\qquad=
\int_0^B g(r)\,dr
\le
\int_0^xg(r)\,dr
=\frac{x^2}{8}.
\end{aligned}
\label{eq:D3.67}
\end{equation}

For the remaining interval, set

\begin{equation}
h=x-B>0.
\label{eq:D3.68}
\end{equation}
Since \(B>\xi\), strict decrease gives \(g(B)<1/4\).  Convexity and
\(g(x)=0\) place the graph below the chord joining
\((B,g(B))\) and \((x,0)\).  Therefore

\begin{equation}
g(r)
\le
g(B)\frac{x-r}{x-B}
<
\frac{x-r}{4h}
\qquad(B\le r<x).
\label{eq:D3.69}
\end{equation}
Integrating the weak form of \eqref{eq:D3.69} gives

\begin{equation}
\begin{aligned}
\int_B^x2rg(r)\,dr
&\le
\frac1{4h}\int_B^x2r(x-r)\,dr\\
&=
\frac1{4h}
\int_0^h2(x-s)s\,ds\\
&=
\frac{xh}{4}-\frac{h^2}{6}.
\end{aligned}
\label{eq:D3.70}
\end{equation}
Combining \eqref{eq:D3.64}, \eqref{eq:D3.67}, and \eqref{eq:D3.70},

\begin{equation}
\mathcal L
\le
\frac{x^2}{8}+\frac{xh}{4}-\frac{h^2}{6}.
\label{eq:D3.71}
\end{equation}
The additional discrete term in \eqref{eq:D3.62} is \(m^2/4\).  Since

\begin{equation}
m=B-\frac12=x-h-\frac12,
\label{eq:D3.72}
\end{equation}
the difference between this term and the upper bound in
\eqref{eq:D3.71} is

\begin{equation}
\begin{aligned}
D(x,h)
&:=
\frac{m^2}{4}
-
\left(
\frac{x^2}{8}+\frac{xh}{4}-\frac{h^2}{6}
\right)\\
&=
\frac{x^2}{8}-\frac{3xh}{4}-\frac{x}{4}
+\frac{5h^2}{12}+\frac h4+\frac1{16}.
\end{aligned}
\label{eq:D3.73}
\end{equation}

It remains to prove \(D(x,h)>0\).

Put
\nopagebreak

\begin{equation}
\tau=x-\xi.
\label{eq:D3.74}
\end{equation}
For \(0\le v\le1\),

\begin{equation}
\arccos(1-v)\ge\sqrt{2v}.
\label{eq:D3.75}
\end{equation}
Indeed, if \(\theta=\arccos(1-v)\), then
\(v=1-\cos\theta\le\theta^2/2\).  Using \eqref{eq:D3.49}, \eqref{eq:D3.75}, and
\(G_x(\xi)=1/4\),

\begin{equation}
\begin{aligned}
\frac14
=G_x(x-\tau)
&=
\frac1\pi\int_0^\tau
\arccos\!\left(1-\frac{s}{x}\right)\,ds\\
&\ge
\frac1\pi\int_0^\tau\sqrt{\frac{2s}{x}}\,ds
=
\frac{2\sqrt2}{3\pi}\frac{\tau^{3/2}}{\sqrt x}.
\end{aligned}
\label{eq:D3.76}
\end{equation}
Consequently,

\begin{equation}
h=x-B<x-\xi=\tau
\le
\gamma x^{1/3},
\qquad
\gamma:=
\left(\frac{3\pi}{8\sqrt2}\right)^{2/3}.
\label{eq:D3.77}
\end{equation}
The upper rational bound for \(\pi\) gives

\begin{equation}
\gamma^3=\frac{9\pi^2}{128}
<
\frac{4356}{6272}
<
\frac{729}{1000}
=\left(\frac9{10}\right)^3,
\label{eq:D3.78}
\end{equation}
where the second strict inequality is the integer comparison

\begin{equation}
4\,356\,000<4\,572\,288.
\label{eq:D3.79}
\end{equation}
Thus

\begin{equation}
0<h<\frac9{10}x^{1/3}.
\label{eq:D3.80}
\end{equation}
For \(x\ge14\), one has \(x^{1/3}\le x/5\), since
\(x^2\ge14^2>5^3\).  Throughout the range \eqref{eq:D3.80},

\begin{equation}
\begin{aligned}
\frac{\partial D}{\partial h}
&=
-\frac{3x}{4}+\frac{5h}{6}+\frac14\\
&<
-\frac{3x}{4}
+\frac56\frac9{10}\frac{x}{5}
+\frac14\\
&=
-\frac{3x}{5}+\frac14<0.
\end{aligned}
\label{eq:D3.81}
\end{equation}
Hence \(D(x,h)\) is bounded below by its value at
\((9/10)x^{1/3}\).  With \(y=x^{1/3}\),

\begin{equation}
D(x,h)>
P(y),
\label{eq:D3.82}
\end{equation}
where

\begin{equation}
P(y)
=
\frac{y^6}{8}
-\frac{27y^4}{40}
-\frac{y^3}{4}
+\frac{27y^2}{80}
+\frac{9y}{40}
+\frac1{16}.
\label{eq:D3.83}
\end{equation}
Because

\begin{equation}
14>\left(\frac{12}{5}\right)^3=\frac{1728}{125},
\label{eq:D3.84}
\end{equation}
we have \(y>12/5\).  Exact substitution gives
\nopagebreak

\begin{equation}
P\!\left(\frac{12}{5}\right)>\frac12.
\label{eq:D3.85}
\end{equation}
Moreover,

\begin{equation}
P'(y)
=
\frac34y^2
\left(y^3-\frac{18}{5}y-1\right)
+\frac{27y}{40}+\frac9{40}.
\label{eq:D3.86}
\end{equation}
The function in parentheses is increasing for \(y\ge12/5\), because

\begin{equation}
3y^2-\frac{18}{5}>0,
\label{eq:D3.87}
\end{equation}
and at the left endpoint it equals

\begin{equation}
\left(\frac{12}{5}\right)^3
-\frac{18}{5}\frac{12}{5}-1
=\frac{523}{125}>0.
\label{eq:D3.88}
\end{equation}
Thus \(P'(y)>0\) for \(y\ge12/5\), and \eqref{eq:D3.82}--\eqref{eq:D3.88} imply

\begin{equation}
D(x,h)>0
\qquad(x\ge14).
\label{eq:D3.89}
\end{equation}
Equations \eqref{eq:D3.62}, \eqref{eq:D3.71}, and \eqref{eq:D3.89} now give

\begin{equation}
\sum_{\ell\ge0}(2\ell+1)q_\ell(x)
\ge I(x)
=\frac{2x^3}{9\pi}
\qquad(x\ge14).
\label{eq:D3.90}
\end{equation}

\end{proof}

\subsubsection{\texorpdfstring{Finite radial-level estimates on \(5\le x\le14\)}{Finite radial-level estimates on 5<= x<=14}}

For the remaining interval, we use the normalized lower bound
\eqref{eq:L.2.5}.  In dimension three,

\begin{equation}
\mathcal D_3=\frac{\pi\sqrt3}{2},\qquad
\beta_3=0,\qquad
e_3:=\frac1{6\sqrt3}.
\label{eq:D3.finite.1}
\end{equation}

Write \(x=3\sqrt3\,a\), \(t=a^{-1/3}\), and

\begin{equation}
r_{3,n}(t):=
1-\frac{\alpha_n}{3}t^2-e_3t^3,\qquad
\mathcal P_{3,M}(t):=
\mathcal D_3t^3\sum_{n=1}^{M}r_{3,n}(t)^2.
\label{eq:D3.finite.2}
\end{equation}
When the displayed radial factors are positive,
\eqref{eq:L.2.5} gives

\begin{equation}
\frac{P_3(3\sqrt3\,a)}{W_3(3\sqrt3\,a)}
\ge\mathcal P_{3,M}(a^{-1/3}).
\label{eq:D3.finite.3}
\end{equation}

\begin{proposition}[Estimate on the intermediate range]
\label{prop:d3-finite}
For \(5\le x\le14\),
\begin{equation}
P_3(x)>\frac{2x^3}{9\pi}.
\label{eq:D3.finite}
\end{equation}
\end{proposition}

\begin{proof}
Proposition \ref{supp:d3-phase} supplies

\begin{equation}
\begin{array}{c|c}
\text{\(a\)-interval}&\text{lower bound}\\ \hline
[5/(3\sqrt3),17/10]&\mathcal P_{3,2}>1\\
{}[17/10,23/10]&\mathcal P_{3,3}>1\\
{}[23/10,14/(3\sqrt3)]&\mathcal P_{3,4}>1 .
\end{array}
\label{eq:D3.finite.4}
\end{equation}
Since \(x=3\sqrt3\,a\), these closed intervals cover exactly
\(5\le x\le14\).
\end{proof}

\begin{theorem}[Weighted quarter-shift inequality in dimension three]
\label{thm:d3-quarter}
For every \(x\ge 5\),

\begin{equation}
\sum_{\substack{n\ge1\\ n-3/4\le x/\pi}}
\left\lfloor
\Xi_x\!\left(n-\frac34\right)+\frac12
\right\rfloor^2
\ge
\int_0^{x/\pi}\Xi_x(t)^2\,dt
=
\frac{2x^3}{9\pi}.
\label{eq:D3.4}
\end{equation}
\end{theorem}

Equality points in \eqref{eq:D3.4} are included; for example,
\(\Xi_x(n-3/4)=\ell+1/2\) contributes the full angular multiplicity
\(2\ell+1\).

\begin{proof}
The identities \eqref{eq:D3.14} and \eqref{eq:D3.16} reduce the assertion
to \eqref{eq:D3.17}.  Proposition \ref{prop:d3-finite} proves it on
\(5\le x\le14\), and Lemma \ref{lem:d3-joint} proves it for \(x\ge14\).
\end{proof}

Since \(\mathfrak h_3=24\sqrt2/35<1\), the beta estimate cannot replace
the high-frequency argument; the endpoint quarter-cell term in Lemma
\ref{lem:retained-floor} remains essential.

\subsection{Completion in dimension three}

\subsubsection{\texorpdfstring{Frequencies \(x\ge5\)}{Frequencies x>=5}}

For \(d=3\), one has

\begin{equation}
\delta_d=\frac12,
\qquad
\kappa_{3,\ell}=2\ell+1.
\label{eq:D3.IV.1}
\end{equation}
Consequently, \eqref{eq:I.11}--\eqref{eq:I.12} gives

\begin{equation}
\mathcal N_3^{<}(x)
\ge
\sum_{\ell\ge0}(2\ell+1)
\left\lfloor
G_x(\ell+\tfrac12)+\frac34
\right\rfloor .
\label{eq:D3.IV.2}
\end{equation}
The exact identity \eqref{eq:D3.14} and Theorem
\ref{thm:d3-quarter} give

\begin{equation}
\mathcal N_3^{<}(x)
\ge
\frac{2x^3}{9\pi}
\qquad(x\ge5).
\label{eq:D3.IV.3}
\end{equation}
The endpoint \(x=5\) is included: the phase theorem counts a derivative
zero non-strictly, while \eqref{eq:I.8} places the corresponding Neumann
eigenvalue strictly below the same threshold.

\subsubsection{\texorpdfstring{Frequencies \(0\le x\le5\)}{Frequencies 0<= x<=5}}

Let \(Y_\ell\) be a spherical harmonic of degree \(\ell\) on \(S^2\), and
put

\begin{equation}
u_\ell(r,\theta)=r^\ell Y_\ell(\theta).
\label{eq:D3.IV.4}
\end{equation}
Green's identity and homogeneity give its exact Rayleigh quotient:

\begin{equation}
\frac{\displaystyle\int_{B_1^3}|\nabla u_\ell|^2}
{\displaystyle\int_{B_1^3}|u_\ell|^2}
=
\ell(2\ell+3).
\label{eq:D3.IV.5}
\end{equation}
Different spherical degrees are orthogonal both for the \(L^2\) form and
for the Dirichlet form.  Hence the space formed by all harmonic
polynomials of degrees \(0,\ldots,L\) has dimension

\begin{equation}
M_L=\sum_{\ell=0}^L(2\ell+1)=(L+1)^2
\label{eq:D3.IV.6}
\end{equation}
and maximum Rayleigh quotient \(L(2L+3)\).  The min--max principle yields
\nopagebreak

\begin{equation}
\mu_{M_L}^N(B_1^3)\le L(2L+3).
\label{eq:D3.IV.7}
\end{equation}
For \(L=0,1,2\), this gives the following piecewise-constant counting
bound:

\begin{equation}
\begin{array}{c|c|c}
\text{frequency range}&\text{strict lower count}&\text{Weyl upper bound}\\
\hline
0<x\le\sqrt5
&\mathcal N_3^{<}(x)\ge1
&2x^3/(9\pi)<1\\
\sqrt5<x\le\sqrt{14}
&\mathcal N_3^{<}(x)\ge4
&2x^3/(9\pi)<4\\
\sqrt{14}<x\le5
&\mathcal N_3^{<}(x)\ge9
&2x^3/(9\pi)<9.
\end{array}
\label{eq:D3.IV.8}
\end{equation}
The three inequalities in the last column follow by integer comparison from
\(\pi>25/8\), using \(\sqrt{14}<4\) in the middle case:

\begin{equation}
\frac{2(\sqrt5)^3}{9\pi}<1,
\qquad
\frac{2(\sqrt{14})^3}{9\pi}<4,
\qquad
\frac{2\cdot5^3}{9\pi}<9.
\label{eq:D3.IV.9}
\end{equation}
All junctions are included.  At \(x=\sqrt5\), the preceding row uses
only the constant eigenfunction.  At \(x=\sqrt{14}\), the degree-one
bound \(5<x^2\) is already strict.  At \(x=5\), the degree-two bound
\(14<x^2\) is strict.  Finally,
\(\mathcal N_3^{<}(0)=0=W_3(0)\).

Equations \eqref{eq:D3.IV.3} and \eqref{eq:D3.IV.8} prove

\begin{equation}
\mathcal N_3^{<}(x)\ge\frac{2x^3}{9\pi}
\qquad(x\ge0).
\label{eq:D3.IV.10}
\end{equation}
This proves the dimension-three unit-ball result.

\section{Dimension four}
\label{sec:d4}

The four-dimensional argument is direct; the disk estimate does not control
the quadratic angular moment arising from the four-dimensional
multiplicities.  Here

\begin{equation}
\delta_4=1,\qquad
\kappa_{4,\ell}=(\ell+1)^2,\qquad
W_4(x)=\frac{x^4}{64},\qquad x=8a.
\label{eq:D4.1}
\end{equation}

\subsection{Low-frequency power trials}

For \(d=4\), Lemma \ref{lem:power-trial} reads
\[
L_\ell=\ell(\ell+2),\qquad
s_\ell=(L_\ell/2)^{1/3},\qquad
Q_4(L_\ell)=
\frac{s_\ell^2(2s_\ell+1)(2s_\ell+4)}{2s_\ell+2}.
\]
The following four rational enclosures suffice:
\begin{equation}
\begin{array}{c|c|c|c}
\ell&L_\ell&s_\ell<&Q_4(L_\ell)<\\ \hline
1&3&23/20&1099791/172000<7\\
2&8&8/5&24192/1625<15\\
3&15&197/100&3805571731/148500000<26\\
4&24&23/10&159229/4125<39.
\end{array}
\label{eq:D4.2}
\end{equation}
The cube-root bounds follow by cubing.  Since \(Q_4\) is increasing, the
strict sectorwise inequalities in Lemma \ref{lem:power-trial} and
\[
\mathsf A_4(1)=5,\quad \mathsf A_4(2)=14,\quad
\mathsf A_4(3)=30,\quad \mathsf A_4(4)=55
\]
give
\begin{equation}
\mathcal N_4^{<}(x)\ge
\begin{cases}
1,&0<x^2<7,\\
5,&7\le x^2<15,\\
14,&15\le x^2<26,\\
30,&26\le x^2<39,\\
55,&39\le x^2\le(38/5)^2.
\end{cases}
\label{eq:D4.3}
\end{equation}
These counts dominate the Weyl term because
\begin{equation}
\frac{7^2}{64}<1,\qquad
\frac{15^2}{64}<5,\qquad
\frac{26^2}{64}<14,\qquad
\frac{39^2}{64}<30,\qquad
W_4(38/5)=\frac{130321}{2500}<55.
\label{eq:D4.6}
\end{equation}
At \(x=0\), both sides vanish.  We have therefore proved

\begin{equation}
\mathcal N_4^{<}(x)\ge W_4(x)
\qquad(0\le x\le38/5).
\label{eq:D4.7}
\end{equation}

\subsection{Intermediate frequencies}

For \(d=4\), the common normalized quantities reduce to

\begin{equation}
\mathcal D_4=\frac83,\qquad
\beta_4=\frac1{256},\qquad
r_{4,n}(a)
=1-\frac{\alpha_n}{4a^{2/3}}-\frac1{16a}.
\label{eq:D4.8}
\end{equation}
Recall that \(x=8a\).  The three consecutive closed \(a\)-intervals below
map exactly to the frequency interval \(38/5\le x\le26\).

Define the truncated expression
\nopagebreak

\begin{equation}
\mathcal P_{4,M}(a)
:=\frac8{3a}\sum_{n=1}^{M}
r_{4,n}(a)
\left(r_{4,n}(a)^2-\frac1{256a^2}\right).
\label{eq:D4.9}
\end{equation}
The radial margins in \eqref{eq:S4.2} give \(r_{4,n}>0\) on every interval
used below.  If
\(0<r_{4,n}\le1/(16a)\), its summand in \eqref{eq:D4.9} is
nonpositive while the truncated summand \(T_{4,n}\) in
\eqref{eq:L.2.4} is zero.  If \(r_{4,n}>1/(16a)\), the two summands
agree.  Therefore the piecewise convention always gives

\begin{equation}
L_{4,M}(a)\ge\mathcal P_{4,M}(a).
\label{eq:D4.10}
\end{equation}

\begin{proposition}[Four-dimensional phase estimate]
\label{prop:d4-phase}
The following bounds hold:

\begin{equation}
\begin{array}{c|c}
\text{\(a\)-interval}&\text{lower bound}\\ \hline
[19/20,11/10]&\mathcal P_{4,2}>1\\
{}[11/10,17/10]&\mathcal P_{4,3}>1\\
{}[17/10,13/4]&\mathcal P_{4,4}>1 .
\end{array}
\label{eq:D4.11}
\end{equation}
\end{proposition}

The proof is given in Proposition \ref{supp:d4-phase}.

Thus \eqref{eq:L.2.5}, \eqref{eq:D4.10}, and
\eqref{eq:D4.11} give

\begin{equation}
\mathcal N_4^{<}(x)>W_4(x)
\qquad(38/5\le x\le26).
\label{eq:D4.12}
\end{equation}

\subsection{A beta-integral estimate}

The common beta lower bound \eqref{eq:L.5.6} applies.  Put

\begin{equation}
\begin{gathered}
\mathfrak h_4=\frac{2048\sqrt2}{945\pi},\qquad
\sigma_4=\frac92,\qquad c_4=\frac1{16},\qquad
\mathcal K_4=\frac{21}{2048},\\
\rho_0=1-\frac1{16a},\qquad
\rho_1=1-\frac1{16a}-\frac{\alpha_1}{4a^{2/3}},
\qquad q=\frac1{256a^2}.
\end{gathered}
\label{eq:D4.13}
\end{equation}
Using \(\alpha_1<887/1000\) and
\((13/4)^{2/3}>219/100\), both evaluation points give nonzero summands
for \(13/4\le a\le21/4\), since

\begin{equation}
\rho_1>\frac78>\frac1{52}\ge\sqrt q,
\label{eq:D4.14}
\end{equation}
while \(\rho_0>\rho_1\).  This verifies the common nonvanishing conditions
\(\mathrm{(A1)}\)--\(\mathrm{(A2)}\).  The elementary radical and
\(\pi\) bounds give

\begin{equation}
\frac{39}{40}<\mathfrak h_4<1,
\qquad
0<
\frac{\mathcal K_4}{(a-c_4)^2}
<\frac1{96}.
\label{eq:D4.15}
\end{equation}
Here the upper bound for \(\mathfrak h_4\) follows from
\(\sqrt2<99/70\) and \(\pi>333/106\).

Define

\begin{equation}
\widehat{\mathfrak B}_4(a):=
\mathfrak h_4\left(1-\frac9{32a}\right)
-\frac{21}{2048(a-\frac1{16})^2}
+\frac1{3a}
\left(
3\rho_1^3-\rho_0^3-\frac3{256a^2}
\right).
\label{eq:D4.16}
\end{equation}
Then

\begin{equation}
\mathfrak B_4(a)>\widehat{\mathfrak B}_4(a).
\label{eq:D4.17}
\end{equation}

Set \(u=\sigma_4 c_4/a\) and
\(\theta=\mathcal K_4/(a-c_4)^2\).  Bernoulli's inequality,
\(0<\mathfrak h_4<1\), and \(0<\theta<1\) give

\begin{equation*}
\mathfrak h_4\rho_0^{\sigma_4}(1-\theta)
-\bigl(\mathfrak h_4-\mathfrak h_4u-\theta\bigr)
\ge
\theta(1-\mathfrak h_4+\mathfrak h_4u)>0.
\end{equation*}
The discarded \(q\)-part is
\(q(3-3\rho_1+\rho_0)>0\).  This proves \eqref{eq:D4.17}.

Since \(\mathfrak h_4<1\), this is only a middle-frequency estimate; the
common estimate takes over at \(x=42\).

Direct differentiation gives

\begin{equation}
\widehat{\mathfrak B}_4''(a)
=\frac1{a^3}
\left[
-2\mathfrak h_4\sigma_4 c_4+\frac13\Psi^{(4)}(u_0,v)
\right]
-\frac{6\mathcal K_4}{(a-c_4)^4}
-\frac1{\,a^5}\frac{36\mathcal D_4\beta_4}{8},
\label{eq:D4.18}
\end{equation}
where \(u_0=c_4/a\), \(v=\alpha_1/(4a^{2/3})\), and

\begin{equation}
\begin{aligned}
\Psi^{(4)}(u_0,v)={}&
4-36u_0-40v+72u_0^2+176u_0v+70v^2\\
&-40u_0^3-154u_0^2v-130u_0v^2-36v^3.
\end{aligned}
\label{eq:D4.19}
\end{equation}
Equation \eqref{eq:D4.19} is \eqref{eq:L.5.19} with \(n=3\) and
\(\mathcal D_4/8=1/3\).

On the entire beta interval,

\begin{equation}
\frac1{85}<u_0<\frac1{50},
\qquad
\frac9{125}<v<\frac{11}{100}.
\label{eq:D4.20}
\end{equation}
The second inequality uses
\(\alpha_1>22/25\),
\((21/4)^{2/3}<303/100\), and the two bounds already used in
\eqref{eq:D4.14}.  Both partial derivatives of \(\Psi^{(4)}\) are negative
on this rectangle; for example, after discarding their negative terms,

\begin{equation*}
\partial_{u_0}\Psi^{(4)}
<-36+\frac{144}{50}+\frac{176\cdot11}{100}<0,
\qquad
\partial_v\Psi^{(4)}
<-40+\frac{176}{50}+\frac{140\cdot11}{100}<0.
\end{equation*}
It follows that

\begin{equation}
\Psi^{(4)}(u_0,v)
<\Psi^{(4)}\left(\frac1{85},\frac9{125}\right)
<\frac65.
\label{eq:D4.21}
\end{equation}
Using \(\mathfrak h_4>39/40\) in the square bracket of
\eqref{eq:D4.18} gives

\begin{equation}
-2\mathfrak h_4\sigma_4 c_4+\frac13\Psi^{(4)}
<-\frac{351}{640}+\frac25
=-\frac{19}{128}<0.
\label{eq:D4.22}
\end{equation}
All remaining terms in \eqref{eq:D4.18} are negative, so
\(\widehat{\mathfrak B}_4\) is strictly concave.  Finally,
\(\alpha_1<887/1000\), \(\mathfrak h_4>39/40\), and the rational
power bounds

\begin{equation*}
\left(\frac{13}{4}\right)^{2/3}>\frac{219}{100},
\qquad
\left(\frac{21}{4}\right)^{2/3}>3
\end{equation*}
give, by direct substitution in \eqref{eq:D4.16},

\begin{equation}
\begin{aligned}
\widehat{\mathfrak B}_4(13/4)-1
&>\frac1{500},\\
\widehat{\mathfrak B}_4(21/4)-1
&>\frac1{200}.
\end{aligned}
\label{eq:D4.23}
\end{equation}
Concavity, \eqref{eq:D4.17}, and the common beta lower bound therefore
prove

\begin{equation}
\mathcal N_4^{<}(x)>W_4(x)
\qquad(26\le x\le42).
\label{eq:D4.24}
\end{equation}

\subsection{High frequencies}

For \(d=4\), \(p=3\), the quantities in the common high-frequency estimate
satisfy

\begin{equation*}
A_3=\frac14,\qquad
C_2=\frac{2}{9\pi},\qquad
C_1=\frac18,\qquad
a_3=\frac14-\frac1{3\pi}.
\end{equation*}
The bounds \(\pi>25/8\) and \(42^{1/3}>347/100\) give

\begin{equation*}
a_3>\frac{43}{300}
\end{equation*}
and

\begin{equation*}
\begin{aligned}
\frac{3}{2\cdot42^{2/3}}
+\frac{3/4+(1/4)C_1}{42}
+\frac{1/4}{2\cdot42^2}
&<
\frac{124}{995}
+\frac{25}{1344}
+\frac1{14112}\\
&<
\frac{43}{300}.
\end{aligned}
\end{equation*}
Here \(347^3=41\,781\,923<42\,000\,000\), and the final comparison is
a cross-multiplication.  Since the right side of \eqref{eq:II.37}
decreases in \(x\),

\begin{equation}
P_4(x)>W_4(x)
\quad\text{and hence}\quad
\mathcal N_4^{<}(x)>W_4(x)
\qquad(x\ge42).
\label{eq:II.45}
\end{equation}

\subsection{Completion of the four-dimensional proof}

Equations \eqref{eq:D4.7}, \eqref{eq:D4.12},
\eqref{eq:D4.24}, and \eqref{eq:II.45} cover, respectively,
\([0,38/5]\), \([38/5,26]\), \([26,42]\), and \([42,\infty)\),
including their joins.  Hence

\begin{equation}
\mathcal N_4^{<}(x)\ge\frac{x^4}{64}
\qquad(x\ge0).
\label{eq:D4.25}
\end{equation}

\section{\texorpdfstring{Dimensions \(d\ge5\)}{Dimensions d>=5}}
\label{sec:dge5}

Set \(a=xd^{-3/2}\).  We treat successively the variational range
\(0\le a\le a_d^\ast\), the finite-level range
\(a_d^\ast\le a\le5/2\), the beta-integral range \(5/2\le a\le5\), and
the high-frequency range \(a\ge5\).

\subsection{The low-frequency variational estimate}

\begin{theorem}[Low-frequency variational estimate]
\label{thm:low-frequency-trial}
Let
\[
a_d^\ast:=
\begin{cases}
41/50,&d=5,\\
4/5,&d\ge6.
\end{cases}
\]
Then
\begin{equation}
\mathcal N_d^{<}(x)\ge W_d(x)
\qquad(d\ge5,\ 0\le x\le a_d^\ast d^{3/2}).
\label{eq:L.1.1}
\end{equation}
\end{theorem}

\begin{proof}
Write \(L_\ell=\ell(\ell+d-2)\).  With the notation of Lemma
\ref{lem:power-trial}, for \(L>0\) put
\[
 s=(L/2)^{1/3},
\qquad Q_d(L)=\frac{s^2(2s+1)(2s+d)}{2s+d-2},
\qquad Q_d(0)=0.
\]
In particular,
\begin{equation}
\frac{Q_d(L)}L
=\left(1+\frac1{2s}\right)
\left(1+\frac2{2s+d-2}\right)>1.
\label{eq:L.1.power.penalty}
\end{equation}

Lemma \ref{lem:power-trial}, together with \(\lambda_{0,1}^{N}=0\), gives
the threshold implication
\begin{equation}
Q_d(L_m)\le x^2
\quad\Longrightarrow\quad
\mathcal N_d^{<}(x)\ge\mathsf A_d(m)
\qquad(m\ge1),
\label{eq:L.1.2a}
\end{equation}
including equality at every trial threshold.

For real \(m\ge1\), set
\[
L=m(m+d-2),\qquad s=(L/2)^{1/3},
\qquad
\widetilde{\mathsf A}_d(m-1)
=\frac{(2m+d-3)\Gamma(m+d-2)}{\Gamma(d)\Gamma(m)}
\]
and
\[
\mathfrak R_d^\sharp(m):=
\frac{\mathcal C_d\widetilde{\mathsf A}_d(m-1)}
{Q_d(L)^{d/2}},
\qquad
\mathcal C_d=2^d\Gamma\left(\frac d2+1\right)^2.
\]
To prove that \(\mathfrak R_d^\sharp\) is strictly decreasing, put
\(c=d-2\) and \(h=2m+c-1\).  The composite midpoint inequality for
\(t\mapsto(m+t)^{-1}\), followed by the power series for
\(\log((1+y)/(1-y))\), gives
\begin{align*}
\frac{d}{dm}\log\widetilde{\mathsf A}_d(m-1)
&=\frac2h+\sum_{j=0}^{c-1}\frac1{m+j}\\
&<\frac2h+\log\frac{h+c}{h-c}\\
&\le \frac{2(c+1)}h+
\frac{2c^3}{3h(h^2-c^2)}=:B_d(m).
\end{align*}
On the other hand, the logarithmic elasticity from Lemma
\ref{lem:power-trial} is
\[
\chi_d(s):=\frac{d\log Q_d}{d\log L}
=\frac13\left(
2+\frac{2s}{2s+1}+\frac{2s}{2s+d}
-\frac{2s}{2s+d-2}\right)>0.
\]
Moreover,
\begin{equation}
\frac d2\frac{2m+c}{m(m+c)}\chi_d(s)>B_d(m)
\qquad(d\ge5,\ m\ge1).
\label{eq:L.1.power.monotone}
\end{equation}
After the positive denominators are cleared, Proposition
\ref{supp:power-trial-monotonicity} reduces the numerator in
\eqref{eq:L.1.power.monotone} to a polynomial with \(81\) positive
coefficients and constant term \(1856\).  Thus
\((\mathfrak R_d^\sharp)'<0\).

For the terminal threshold, let \(a=a_d^\ast\), and let \(\mu>0\)
be determined by
\[
Q_d\bigl(\mu(\mu+d-2)\bigr)=a^2d^3.
\]
Put \(L=\mu(\mu+d-2)=2s^3\), and define
\[
\frac{\mathcal C_d}{(d-1)!}
=\frac{d^{3/2}}2\mathcal D_d,
\qquad
\varrho_d:=\frac{2\mu+d-3}{\mu+d-2},
\]
\[
U_d:=\prod_{j=1}^{\lfloor(d-3)/2\rfloor}
\left(1+\frac{j(d-2-j)}L\right)>1.
\]
Pairing \(j\) with \(d-2-j\) in
\[
\widetilde{\mathsf A}_d(\mu-1)
=\frac{\varrho_d}{(d-1)!}\prod_{j=0}^{d-2}(\mu+j)
\]
and using
\((\mu+j)(\mu+d-2-j)=L+j(d-2-j)\) yields
\begin{equation}
\mathfrak R_d^\sharp(\mu)
\ge \frac{\mathcal C_d}{(d-1)!\sqrt L}\,
\varrho_d U_d\left(\frac L{Q_d(L)}\right)^{d/2}.
\label{eq:L.1.power.factor}
\end{equation}
For even \(d\), the unpaired central factor is
\(\sqrt{L+(d-2)^2/4}>\sqrt L\), so the same estimate holds.

Suppose first that \(d\ge7\), so \(a=4/5\) and
\(Q_d(L)=16d^3/25\).  Writing
\(\widehat Q_d(s):=Q_d(2s^3)\) and \(s_0=63d/100\), direct substitution
gives
\[
\frac{16d^3}{25}-\widehat Q_d(s_0)
=\frac{d^3(7904689d-54424850)}
{500000(113d-100)}>0.
\]
Consequently \(s>s_0\).  Also
\[
2s_0^3>\frac54(d-1)(3d-1)
=\frac{3d-1}{2}\left(\frac{3d-1}{2}+d-2\right),
\]
so \(\mu>(3d-1)/2\) and therefore \(\varrho_d>8/5\).  Since
\(L<Q_d(L)=a^2d^3\),
\[
\frac{\mathcal C_d}{(d-1)!\sqrt L}
>\frac{\mathcal D_d}{2a}
>\frac{25}{16}.
\]
The second inequality follows from Lemma \ref{lem:dimension-factor}.

To bound the two factors in \eqref{eq:L.1.power.penalty}, put
\[
u=\frac1{2s}<\frac{50}{63d},
\qquad
v=\frac2{2s+d-2}<\frac{100}{113d-100}<\frac{147}{1000}.
\]
For \(0<z<147/1000\),
\[
\log(1+z)\le z-\frac{z^2}{2}+\frac{z^3}{3}
\le z-\frac{451}{1000}z^2=:g(z),
\]
and \(g\) is increasing on this interval.  If \(k=d-7\), exact
simplification gives
\begin{align*}
&\frac{17}{20}-\frac d2\left\{
g\left(\frac{50}{63d}\right)
+g\left(\frac{100}{113d-100}\right)\right\}\\
&\qquad=
\frac{10842237k^3+134569552k^2+395079889k+3288466}
{79380d(113d-100)^2}>0.
\end{align*}
Furthermore,
\[
e^{17/20}
<1+\frac{17}{20}+\frac12\left(\frac{17}{20}\right)^2
+\frac{(17/20)^3}{6(1-17/80)}
=\frac{353993}{151200}<\frac{50}{21}.
\]
It follows that
\[
\left(\frac L{Q_d(L)}\right)^{d/2}>\frac{21}{50}.
\]
Substitution in \eqref{eq:L.1.power.factor} now gives
\[
\mathfrak R_d^\sharp(\mu)
>\frac{25}{16}\frac85\frac{21}{50}=\frac{21}{20}>1
\qquad(d\ge7).
\]

For \(d=5,6\), the two required roots are enclosed by the exact rational
brackets
\[
\begin{array}{c|c|c|c}
d&a&Q_d(L)&s\\ \hline
5&41/50&1681/20&3097/1000<s<31/10\\
6&4/5&3456/25&3729/1000<s<15/4.
\end{array}
\]
Substitution at the two endpoints gives
\[
\frac L{Q_d(L)}>
\begin{cases}7/10,&d=5,\\3/4,&d=6,
\end{cases}
\qquad
\frac{Q_d(L)}L>\left(\frac{57}{50}\right)^2,
\qquad
L>d(2d-2).
\]
Hence \(\varrho_d>3/2\), while
\[
\left(\frac L{Q_d(L)}\right)^{d/2}>\frac25,
\qquad
\frac{\mathcal C_d}{(d-1)!\sqrt L}
=\frac{\mathcal D_d}{2a}\sqrt{\frac{Q_d(L)}L}
>\frac{5}{4a}\frac{57}{50}.
\]
Dropping \(U_d>1\) in \eqref{eq:L.1.power.factor} therefore gives
\[
\mathfrak R_5^\sharp(\mu)>\frac{171}{164}>1,
\qquad
\mathfrak R_6^\sharp(\mu)>\frac{171}{160}>1.
\]

To pass from the thresholds to the full interval, put
\(M=\lfloor\mu\rfloor\) and recall that \(L_m=m(m+d-2)\).  For every
integer \(1\le m\le M\), monotonicity gives
\[
\mathcal C_d\mathsf A_d(m-1)
=\mathfrak R_d^\sharp(m)Q_d(L_m)^{d/2}
>Q_d(L_m)^{d/2}.
\]
Together with \eqref{eq:L.1.2a}, this proves the theorem on each half-open
interval
\(Q_d(L_{m-1})\le x^2<Q_d(L_m)\); for \(m=1\), the constant mode covers
\(0<x^2<Q_d(L_1)\).  At \(x^2=Q_d(L_m)\), the threshold implication
supplies the count associated with \(m\).

Finally, \(\widetilde{\mathsf A}_d\) is increasing and \(M>\mu-1\), so
\[
\mathcal C_d\mathsf A_d(M)
>\mathcal C_d\widetilde{\mathsf A}_d(\mu-1)
>Q_d(L)^{d/2}=(a d^{3/2})^d.
\]
The threshold implication covers
\(Q_d(L_M)\le x^2\le Q_d(L)\), including both endpoints.  At \(x=0\),
both sides of \eqref{eq:L.1.1} vanish.  This completes the proof.
\end{proof}

\subsection{A finite-level estimate}

For \(M=2,3,4\), set

\begin{equation}
I_{d,2}:=[a_d^\ast,11/10],\qquad
I_{d,3}:=[11/10,17/10],\qquad
I_{d,4}:=[17/10,5/2].
\label{eq:four-layer-bands}
\end{equation}

\begin{theorem}[Finite-level estimate]
\label{thm:four-layer}
For every integer \(d\ge5\),

\begin{equation}
\frac{\mathcal N_d^{<}(a d^{3/2})}
{W_d(a d^{3/2})}>1
\qquad
\left(a_d^\ast\le a\le\frac52\right).
\label{eq:L.fourlayer}
\end{equation}
\end{theorem}

\begin{proof}
For \(d=5,6\), combine Proposition \ref{supp:d56-estimate} with the
retained-tail bound \eqref{eq:L.retained-four}.  If \(d\ge7\), Proposition
\ref{supp:aggregate-proposition} gives \(L_{d,M}>1\) on each interval
\(I_{d,M}\); the counting estimate then follows from \eqref{eq:L.2.5}.
\end{proof}

\begin{unnumberedremark}[The dimension-five junction]
For \(d=5\), the estimate begins at \(41/50\), not \(4/5\):
\eqref{eq:L.d5.no-uniform} shows why, and Theorem
\ref{thm:low-frequency-trial} reaches the stated junction.
\end{unnumberedremark}

\subsection{A beta-integral estimate}

The common beta-integral reduction \eqref{eq:L.5.6}--\eqref{eq:L.5.8}
gives \(\mathcal N_d^{<}/W_d\ge\mathfrak B_d\).  It remains to prove
that this explicit lower bound exceeds one on the required normalized
range.

Put
\[
c_d:=\frac1{2d^{3/2}},\qquad
m:=d-1,\qquad
\sigma_d:=d+\frac12,\qquad
\mathcal K_d:=\frac{(d-3)(d^2-\frac14)}{24d^3}.
\]

The sign of the endpoint correction rests on the following nonvanishing
condition.  For
\(d\ge5\) and \(5/2\le a\le5\),

\begin{equation}
\rho_0=1-\frac{1}{2ad^{3/2}}>\frac{49}{50},
\qquad
\sqrt q=\frac{\sqrt{\beta_d}}a<\frac1{10}.
\label{eq:L.beta.activation}
\end{equation}
Thus \(\rho_0>\sqrt q\), so
\(f_0=\rho_0^{d-3}(\rho_0^2-q)>0\).  Consequently, the correction
\(-f_0/8\) preserves the lower-bound direction.  This proves condition
\(\mathrm{(A1)}\) in
\eqref{eq:L.beta.active}.  Condition \(\mathrm{(A2)}\) is verified at
\eqref{eq:L.5.29}, before the \(f_1\)-correction is used.

\begin{theorem}[Beta-integral estimate]
\label{thm:beta-integral}
For every integer \(d\ge5\),
\begin{equation}
\mathfrak B_d(a)>1
\qquad\left(\frac52\le a\le5\right).
\label{eq:L.5.9}
\end{equation}
\end{theorem}

\begin{proof}
\noindent\emph{The interval \(3\le a\le5\), \(d\ge10\).}

The two Gamma factors in \eqref{eq:L.5.8} satisfy

\begin{equation}
\frac{\mathfrak h_{d+2}}{\mathfrak h_d}
=
\frac{(d+2)^2}{(d+2)^2-\frac14},
\qquad
\lim_{d\to\infty}\mathfrak h_d=1.
\label{eq:L.5.10}
\end{equation}
Consequently,

\begin{equation}
\mathfrak h_d
=
\prod_{j\ge0}
\left(1-\frac1{4(d+2+2j)^2}\right)
>
1-\frac{d+4}{8(d+2)^2}.
\label{eq:L.5.11}
\end{equation}
Here we used
\(\prod(1-t_j)\ge1-\sum t_j\) and the integral estimate

\begin{equation*}
\sum_{j\ge0}\frac1{(d+2+2j)^2}
\le
\frac1{(d+2)^2}+\frac1{2(d+2)}.
\end{equation*}
The corresponding lower bound for \(\mathfrak h_d\) is equally
elementary.  Set

\begin{equation*}
\mathfrak G_d:=\mathfrak h_d\sqrt{\frac{d+\frac12}{d}}.
\end{equation*}
Using \eqref{eq:L.5.10},

\begin{equation*}
\left(\frac{\mathfrak G_{d+2}}{\mathfrak G_d}\right)^2
=
\frac{
d(d+2)^3(d+\frac52)
}{
(d+\frac12)((d+2)^2-\frac14)^2
}<1,
\end{equation*}
because the denominator minus the numerator is

\begin{equation*}
\frac{(2d+5)(8d^2+40d+45)}{32}>0.
\end{equation*}
Since \(\mathfrak G_d\to1\), each parity subsequence decreases to \(1\), so
\nopagebreak

\begin{equation}
\mathfrak h_d>\sqrt{\frac d{d+\frac12}},
\label{eq:L.5.12}
\end{equation}

Let

\begin{equation*}
\lambda:=d(1-\rho_1)
=\frac{\alpha_1}{a^{2/3}}+\frac1{2a\sqrt d}.
\end{equation*}
The elementary bounds

\begin{equation*}
\alpha_1<\frac9{10},\qquad
3^{2/3}>\frac{52}{25},\qquad
\sqrt d>3
\end{equation*}
imply

\begin{equation*}
\lambda<
\frac{45}{104}+\frac1{18}
=\frac{457}{936}<1.
\end{equation*}
Since

\begin{equation*}
(d-1)\log\left(1-\frac{\lambda}{d}\right)>-\lambda
\end{equation*}
and

\begin{equation*}
e^{457/936}
<
1+\frac{457}{936}
+\frac12\left(\frac{457}{936}\right)^2
+\frac{(457/936)^3}{6(1-457/(4\cdot936))}
<
\frac{80}{49},
\end{equation*}
we obtain

\begin{equation}
\rho_1^{d-1}>\frac{49}{80}.
\label{eq:L.5.13}
\end{equation}
The last displayed comparison follows by cross-multiplication; the margin
over the preceding Taylor majorant is greater than \(3/1000\).

Since \(q<1/216\), \(f_1\ge \rho_1^{d-1}-q\), and \(f_0\le1\),

\begin{equation}
3f_1-f_0
>
3\frac{49}{80}-1-\frac1{72}
=\frac{593}{720}>\frac45.
\label{eq:L.5.14}
\end{equation}
Using \eqref{eq:L.5.11}, Bernoulli's inequality for \(\rho_0^{d+1/2}\),
Lemma \ref{lem:dimension-factor}, and
\eqref{eq:L.5.12}--\eqref{eq:L.5.14} in \eqref{eq:L.5.8} gives

\begin{equation}
\mathfrak B_d(a)-1
>
\frac1{4a}
-\frac{d+4}{8(d+2)^2}
-\frac{d+\frac12}{2ad^{3/2}}
-\vartheta_d(a).
\label{eq:L.5.15}
\end{equation}
For \(d\ge10\) and \(3\le a\le5\), \(\rho_0>179/180\), and hence

\begin{equation}
\begin{aligned}
a\left[
\frac{d+4}{8(d+2)^2}
+\frac{d+\frac12}{2ad^{3/2}}
+\vartheta_d(a)
\right]
&<
\frac{35}{576}
+\frac7{40}
+\frac1{72}\left(\frac{180}{179}\right)^2\\
&<\frac14.
\end{aligned}
\label{eq:L.5.16}
\end{equation}
The final inequality follows by cross multiplication:
\(450\cdot2880<41\cdot32041\).  Equations \eqref{eq:L.5.15}--\eqref{eq:L.5.16}
prove \eqref{eq:L.5.9} for \(d\ge10\).

\medskip
\noindent\emph{The cases \(5\le d\le9\).}

For \(5\le d\le9\), write

\begin{equation*}
\rho_0=1-\frac{c_d}{a},\qquad
\rho_1=1-\frac{c_d}{a}-\frac{\alpha_1}{d\,a^{2/3}}.
\end{equation*}
The recurrence \eqref{eq:L.5.10} implies \(0<\mathfrak h_d<1\).  Bernoulli's
inequality and the signs of the \(q\)-terms yield

\begin{equation}
\mathfrak B_d(a)\ge\underline{\mathfrak B}_d(a),
\label{eq:L.5.17}
\end{equation}
where

\begin{equation}
\underline{\mathfrak B}_d(a):=
\mathfrak h_d-\frac{\sigma_d c_d}{a}
-\frac{\mathcal K_d}{(a-c_d)^2}
+\frac{\mathcal D_d}{8a}
\left(
3\rho_1^m-\rho_0^m-\frac{3\beta_d}{a^2}
\right).
\label{eq:L.5.18}
\end{equation}
This estimate uses
\(f_1\ge \rho_1^m-q\) and \(f_0\le \rho_0^m\).  Both evaluation points
are nonzero: \(\alpha_1<9/10\), \(a^{2/3}>2\), and
\(c_d/a<1/60\) give \(\rho_1>67/75>\sqrt q\).

Set

\begin{equation*}
u:=\frac{c_d}{a},\qquad
v:=\frac{\alpha_1}{d\,a^{2/3}}
\end{equation*}
and define

\begin{equation}
\begin{aligned}
\Phi_m(u,v):={}&(1-u-v)^{m-2}
\Bigg[
2(1-u-v)^2\\
&-m(1-u-v)
\left(4u+\frac{22}{9}v\right)
+m(m-1)\left(u+\frac23v\right)^2
\Bigg].
\end{aligned}
\label{eq:L.5.19}
\end{equation}
Direct differentiation gives

\begin{equation*}
a^3\left[\frac{(1-u-v)^m}{a}\right]''
=\Phi_m(u,v).
\end{equation*}
Thus, with
\(\Psi_m(u,v):=3\Phi_m(u,v)-\Phi_m(u,0)\),

\begin{equation}
\underline{\mathfrak B}_d''(a)
=
\frac1{a^3}
\left[-2\sigma_d c_d+\frac{\mathcal D_d}{8}\Psi_m(u,v)\right]
-\frac{6\mathcal K_d}{(a-c_d)^4}
-\frac{36\mathcal D_d\beta_d}{8a^5}.
\label{eq:L.5.20}
\end{equation}
The proof of Proposition \ref{supp:beta-finite} gives

\begin{equation}
\Psi_m(u,v)<\frac{103}{100}
\quad
(m=4,\ldots,8,\ 3\le a\le5).
\label{eq:L.5.21}
\end{equation}

\subsubsection{\texorpdfstring{The interval \(3\le a\le5\) in dimensions
\(5\le d\le9\)}{The interval 3<=a<=5 in dimensions 5<=d<=9}}
For \(5\le d\le9\), Proposition
\ref{supp:beta-finite} proves by exact rational bounds that

\begin{equation}
\underline{\mathfrak B}_d''(a)<0,\qquad
\underline{\mathfrak B}_d(3)>1,\qquad
\underline{\mathfrak B}_d(5)>1.
\label{eq:L.beta.finite}
\end{equation}
Equation \eqref{eq:L.5.17} and concavity give
\(\mathfrak B_d(a)>1\) on \(3\le a\le5\).
Together with the preceding argument for \(d\ge10\), this proves the
theorem on \(3\le a\le5\).

\medskip
\noindent\emph{The interval \(5/2\le a\le3\).}

Retain the notation in \eqref{eq:L.5.8}.  Thus

\begin{equation*}
\rho_0=1-\frac {c_d}a,\qquad
\rho_1=1-\frac {c_d}a-\frac{\alpha _1}{d\,a^{2/3}},
\qquad q=\frac{\beta_d}{a^2}.
\end{equation*}
Here \(\rho_1>\sqrt q\), which is condition \(\mathrm{(A2)}\) in
\eqref{eq:L.beta.active}; condition \(\mathrm{(A1)}\) was proved above.
Indeed,

\begin{equation}
\rho_1>1-\frac1{50}-\frac1{10}
=\frac{22}{25}
>\frac1{12}
>\frac{\sqrt{\beta_d}}a .
\label{eq:L.5.29}
\end{equation}

Define

\begin{equation}
\widehat{\mathfrak B}_d(a):=
\mathfrak h_d-\frac{\mathfrak h_d\sigma_d c_d}{a}
-\frac{\mathcal K_d}{(a-c_d)^2}
+\frac{\mathcal D_d}{8a}
\left(3\rho_1^m-\rho_0^m-3q\right).
\label{eq:L.5.30}
\end{equation}
Then

\begin{equation}
\mathfrak B_d(a)>\widehat{\mathfrak B}_d(a).
\label{eq:L.5.31}
\end{equation}

To prove this, set \(u=\sigma_d c_d/a\) and
\(\theta=\mathcal K_d/(a-c_d)^2\).  Here \(0<\theta<1/96\).
Bernoulli's inequality and \(0<\mathfrak h_d<1\) give

\begin{equation*}
\begin{aligned}
&\mathfrak h_d\rho_0^{\sigma_d}(1-\theta)
-\left(\mathfrak h_d-\mathfrak h_du-\theta\right)\\
&\qquad\ge
\theta(1-\mathfrak h_d+\mathfrak h_du)>0.
\end{aligned}
\end{equation*}
The discarded \(q\)-part is also positive:

\begin{equation*}
q\left(3-3\rho_1^{m-2}+\rho_0^{m-2}\right)>0.
\end{equation*}
This proves \eqref{eq:L.5.31}.

Let \(\Phi_m\) and \(\Psi_m\) be the polynomials introduced in
\eqref{eq:L.5.19}--\eqref{eq:L.5.20}, and set

\begin{equation}
\Psi_d(a):=\Psi_m(u(a),v(a))
=3\Phi_m(u(a),v(a))-\Phi_m(u(a),0),
\qquad
u(a)=\frac {c_d}a,\quad
v(a)=\frac{\alpha_1}{d\,a^{2/3}}.
\label{eq:L.5.32}
\end{equation}
Direct differentiation, exactly as in \eqref{eq:L.5.20}, gives

\begin{equation}
\widehat{\mathfrak B}_d''(a)=
\frac1{a^3}
\left[-2\mathfrak h_d\sigma_d c_d+\frac{\mathcal D_d}{8}\Psi_d(a)\right]
-\frac{6\mathcal K_d}{(a-c_d)^4}
-\frac{36\mathcal D_d\beta_d}{8a^5}.
\label{eq:L.5.33}
\end{equation}
We need the following uniform elementary estimate:

\begin{equation}
\Psi_d(a)<\frac14
\qquad
\left(d\ge5,\ \frac52\le a\le3\right).
\label{eq:L.5.34}
\end{equation}
Proposition \ref{supp:beta-V} proves \eqref{eq:L.5.34}.

The bounds needed below,
\(1/2<\mathfrak h_d<1\) and \(5/2<\mathcal D_d<8/3\), follow from
\eqref{eq:L.5.11} and Lemma \ref{lem:dimension-factor}.

From \eqref{eq:L.5.33}--\eqref{eq:L.5.34},

\begin{equation*}
\begin{aligned}
a^3\widehat{\mathfrak B}_d''(a)
&<
-2\mathfrak h_d\sigma_d c_d+\frac1{12}
-\frac{6\mathcal K_da^3}{(a-c_d)^4}\\
&<
-2\mathfrak h_d\sigma_d c_d+\frac1{12}-2\mathcal K_d.
\end{aligned}
\end{equation*}
But

\begin{equation*}
\frac1{12}-2\mathcal K_d
=\frac{12d^2+d-3}{48d^3}
<\frac1{2\sqrt d}
<2\mathfrak h_d\sigma_d c_d.
\end{equation*}
Consequently

\begin{equation}
\widehat{\mathfrak B}_d''(a)<0
\qquad\left(\frac52\le a\le3\right).
\label{eq:L.5.43}
\end{equation}

\subsubsection{\texorpdfstring{The remaining interval \(5/2\le a\le3\)}{The remaining interval 5/2<=a<=3}}
For \(5\le d\le9\), Proposition
\ref{supp:beta-endpoints} proves

\begin{equation}
\widehat{\mathfrak B}_d(5/2)>1,\qquad
\widehat{\mathfrak B}_d(3)>1.
\label{eq:L.beta.endpoints}
\end{equation}
At \(d=5,a=5/2\), \eqref{eq:L.5.47} gives a margin exceeding
\(3/5000\).  Concavity proves the required bound for
\(5\le d\le9\).  For \(d\ge10\),
parity monotonicity of
\(\mathfrak h_d\), starting at \(d=10,11\), gives

\begin{equation}
\mathfrak h_d>\frac{247}{250}.
\label{eq:L.5.48}
\end{equation}
Indeed, direct evaluation of the two parity bases from the closed Gamma
formulas, using \(707/500<\sqrt2\) and \(\pi<355/113\), gives

\begin{equation*}
\mathfrak h_{10}>\frac{247}{250}+\frac1{2000},
\qquad
\mathfrak h_{11}>\frac{247}{250}+\frac3{2000}.
\end{equation*}
The recurrence \eqref{eq:L.5.10} then proves \eqref{eq:L.5.48} in both
parity classes.

At \(a=5/2\), put

\begin{equation*}
\lambda=d(1-\rho_1)
<L:=\frac{887}{1842}+\frac5{79}<1.
\end{equation*}
Since

\begin{equation*}
(d-1)\log\left(1-\frac{\lambda}{d}\right)>-\lambda
\end{equation*}
and

\begin{equation*}
1+L+\frac{L^2}{2}
+\frac{L^3}{6(1-L/4)}<\frac74,
\end{equation*}
we have \(\rho_1^{d-1}>4/7\).  Thus

\begin{equation*}
3\rho_1^{d-1}-\rho_0^{d-1}-3q>\frac{243}{350}.
\end{equation*}
Moreover

\begin{equation*}
\frac{\sigma_d c_d}{5/2}<\frac{21}{316},\qquad
\rho_0>\frac{157}{158},\qquad
\frac{\mathcal K_d}{(5/2-c_d)^2}
<\frac{158^2}{150\cdot157^2}.
\end{equation*}
Therefore

\begin{equation}
\widehat{\mathfrak B}_d(5/2)-1>
\frac{243}{2800}
-\left(
\frac3{250}+\frac{21}{316}
+\frac{158^2}{150\cdot157^2}
\right)
>\frac3{2000}.
\label{eq:L.5.49}
\end{equation}
At \(a=3\), the same argument with

\begin{equation*}
L_3=\frac{887}{2080}+\frac{25}{474}
\end{equation*}
uses the exact Taylor comparison

\begin{equation*}
\frac53-
\left(
1+L_3+\frac{L_3^2}{2}
+\frac{L_3^3}{6(1-L_3/4)}
\right)
>\frac1{20}.
\end{equation*}
Thus \(\rho_1^{d-1}>3/5\), and hence
\nopagebreak

\begin{equation*}
3\rho_1^{d-1}-\rho_0^{d-1}-3q>\frac{283}{360}.
\end{equation*}
Using

\begin{equation*}
\frac{\sigma_d c_d}{3}<\frac{35}{632},\qquad
\rho_0>\frac{943}{948},\qquad
\frac{\mathcal K_d}{(3-c_d)^2}
<\frac{948^2}{216\cdot943^2},
\end{equation*}
we obtain

\begin{equation}
\widehat{\mathfrak B}_d(3)-1>
\frac{283}{3456}
-\left(
\frac3{250}+\frac{35}{632}
+\frac{948^2}{216\cdot943^2}
\right)
>\frac1{102}.
\label{eq:L.5.50}
\end{equation}
The final inequality follows by cross-multiplication.  Combining
\eqref{eq:L.5.46}--\eqref{eq:L.5.50} shows that both endpoints of the
concave function \(\widehat{\mathfrak B}_d\) exceed \(1\).
Equations \eqref{eq:L.5.31} and \eqref{eq:L.5.43} then give
\(\mathfrak B_d(a)>1\) on \(5/2\le a\le3\), completing the proof.
\end{proof}

\subsection{\texorpdfstring{Completion for \(d\ge5\)}{Completion for d>=5}}

\begin{proposition}[The unit ball in dimensions \(d\ge5\)]
\label{prop:dge5-unit}
For every integer \(d\ge5\) and every \(x\ge0\),

\begin{equation*}
{\mathcal N_d^{<}(x)\ge W_d(x).}
\end{equation*}
\end{proposition}

\begin{proof}
Write \(x=ad^{3/2}\).  Theorems \ref{thm:low-frequency-trial},
\ref{thm:four-layer}, and \ref{thm:beta-integral}, together with Proposition
\ref{prop:scalar-tail}, cover, respectively,
\(0\le a\le a_d^\ast\), \(a_d^\ast\le a\le5/2\),
\(5/2\le a\le5\), and \(a\ge5\).  These four closed intervals cover
\([0,\infty)\).
\end{proof}

\section{Scaling to arbitrary radii}
\label{sec:scaling}

Equations \eqref{eq:D2.1.2}, \eqref{eq:D3.IV.10},
\eqref{eq:D4.25}, and Proposition \ref{prop:dge5-unit} yield
\[
\mathcal N_d^{<}(x)\ge W_d(x)
\qquad(d\ge2,\ x\ge0).
\]

Neumann scaling gives

\begin{equation*}
\mu_j^N(B_R^d)=R^{-2}\mu_j^N(B_1^d).
\end{equation*}
For \(x=R\sqrt E\),

\begin{equation*}
\begin{aligned}
N_{B_R^d}^{<}(E)
&=
\mathcal N_d^{<}(R\sqrt E)\\
&\ge
W_d(R\sqrt E)
=
\frac{(R\sqrt E)^d}{2^d\Gamma(\frac d2+1)^2}\\
&=
\frac{\omega_d}{(2\pi)^d}|B_R^d|E^{d/2}.
\end{aligned}
\end{equation*}
This proves Theorem \ref{thm:main}.

\subsection{Consequences}

\begin{corollary}[Eigenvalue form]
\label{cor:eigenvalue-form}
For every \(n\ge1\),

\begin{equation}
\mu_{n+1}^N(B_R^d)
\le
\frac{4\pi^2 n^{2/d}}
     {(\omega_d|B_R^d|)^{2/d}}.
\label{eq:eigenvalue-form}
\end{equation}
\end{corollary}

\begin{proof}
Let \(E_n\) denote the right-hand side of \eqref{eq:eigenvalue-form}.  For
every \(E>E_n\), Theorem \ref{thm:main} gives a lower bound strictly larger
than \(n\); integrality therefore yields
\(N_{B_R^d}^{<}(E)\ge n+1\).  Hence \(\mu_{n+1}^N(B_R^d)<E\), and letting
\(E\downarrow E_n\) proves the claim.
\end{proof}

\begin{corollary}[Domains that tile a ball]
\label{cor:ball-tiles}
Suppose that a bounded Lipschitz domain \(\Omega\) tiles a ball
\(B_R^d\) by finitely many congruent copies.  Then
\[
N_\Omega^{<}(E)
\ge
\frac{\omega_d}{(2\pi)^d}|\Omega|E^{d/2}
\qquad(E\ge0).
\]
In particular, the conclusion holds for half-balls and orthant sectors
of a ball.
\end{corollary}

\begin{proof}
For such a tiling by \(m\) copies, Neumann bracketing gives
\[
N_{B_R^d}^{<}(E)\le mN_\Omega^{<}(E).
\]
Theorem \ref{thm:main} and \(|B_R^d|=m|\Omega|\) now give the transfer
principle \cite[Theorem 1.8]{FLPS}.
\end{proof}

\appendix
\renewcommand{\thesubsection}{\thesection.\arabic{subsection}}
\phantomsection
\section*{Appendices}
\addcontentsline{toc}{section}{Appendices}

The following sections collect the common constant bounds,
dimension-specific finite-range calculations, and uniform tail estimates
used in the proof.

\section{Constants and auxiliary estimates}

Table \ref{tab:appendix-map} indicates which appendix estimates are used in
each frequency range.

\begin{table}[ht]
\centering
\small
\setlength{\tabcolsep}{3pt}
\begin{tabular}{@{}>{\raggedright\arraybackslash}p{0.22\textwidth}>{\raggedright\arraybackslash}p{0.20\textwidth}>{\raggedright\arraybackslash}p{0.27\textwidth}>{\raggedright\arraybackslash}p{0.22\textwidth}@{}}
\toprule
Result in the body & Appendix result & Argument & Range \\ \midrule
Disk Ritz bound & Appendix \ref{supp:disk-ritz}
& determinant calculation & finite disk range \\
Disk phase estimate & Appendix \ref{supp:disk-analytic-tail}
& shifted floor sum and inverse action & \(x>20/3\) \\
Three-dimensional intermediate range & Proposition \ref{supp:d3-phase}
& convexity and rational endpoints & \(5\le x\le14\) \\
Four-dimensional intermediate estimates & Proposition \ref{supp:d4-phase}
& finite-level polynomials and beta-integral bounds & \(38/5\le x\le42\) \\
Finite-level estimate & Propositions \ref{supp:d56-estimate}
and \ref{supp:aggregate-proposition} & retained-node quadrature for
\(d=5,6\), uniform finite-level bounds for \(d\ge7\), and logarithmic concavity
& \(a_d^\ast\le a\le5/2\) \\
Beta-integral bounds & Propositions \ref{supp:beta-finite},
\ref{supp:beta-V}, and \ref{supp:beta-endpoints}
& monotonicity, complete squares, and endpoint estimates & \(5/2\le a\le5\) \\
High-frequency estimate & Proposition \ref{supp:uniform-tail}
& induction in the two parity classes & \(x\ge5d^{3/2}\), \(d\ge5\) \\
\bottomrule
\end{tabular}
\caption{Where the appendix estimates are used.}
\label{tab:appendix-map}
\end{table}
The only nonstandard special-function input is Proposition
\ref{prop:flps-phase}.  Proposition \ref{prop:strict-robin} transfers it to
the Neumann eigenvalues.  All other finite assertions are proved below.

Subsection \ref{supp:constant-proofs} proves the common enclosures; sharper
local bounds appear where used.

\subsection{Proof of the shifted floor-sum lemma}

\label{supp:retained-tail-proof}

We prove Lemma \ref{lem:retained-floor}.

\begin{proof}

First consider integers \(i<j\) and an integer \(n\) such that

\begin{equation}
n+\frac14>f(i)\ge\cdots\ge f(j-1)
\ge n-\frac34>f(j).
\label{eq:D3.28}
\end{equation}
The \(1/2\)-Lipschitz property gives

\begin{equation}
f(j)\ge f(j-1)-\frac12\ge n-\frac54.
\label{eq:D3.29}
\end{equation}
Therefore

\begin{equation}
\left\lfloor f(r)+\frac34\right\rfloor=n
\quad(i\le r\le j-1),
\qquad
\left\lfloor f(j)+\frac34\right\rfloor=n-1.
\label{eq:D3.30}
\end{equation}
Convexity places the graph below the chord joining its endpoints, so

\begin{equation}
\int_i^j f(s)\,ds
\le
\frac{j-i}{2}\bigl(f(i)+f(j)\bigr)
\le
(j-i)\left(n-\frac14\right).
\label{eq:D3.31}
\end{equation}
Together with \eqref{eq:D3.30}, this gives

\begin{equation}
\begin{aligned}
&\frac12\left\lfloor f(i)+\frac34\right\rfloor
+\sum_{r=i+1}^{j-1}
\left\lfloor f(r)+\frac34\right\rfloor
+\frac12\left\lfloor f(j)+\frac34\right\rfloor\\
&\hspace{20mm}\ge
\int_i^j f(s)\,ds+\frac{j-i}{4}-\frac12.
\end{aligned}
\label{eq:D3.32}
\end{equation}

Put

\begin{equation}
K=\left\lfloor f(0)+\frac34\right\rfloor\ge1.
\label{eq:D3.33}
\end{equation}
For \(0\le k\le K-1\), define

\begin{equation}
\mathcal S_k
=
\left\{
r\in\{0,\ldots,\lfloor b\rfloor\}:
f(r)\ge k+\frac14
\right\},
\qquad
M_k=1+\max\mathcal S_k,
\label{eq:D3.34}
\end{equation}
and put \(M_K=0\).  Every \(\mathcal S_k\) is nonempty because

\begin{equation}
k+\frac14\le K-\frac34\le f(0).
\label{eq:D3.35}
\end{equation}
If \(\eta=f^{-1}(1/4)\), strict decrease gives

\begin{equation}
f(r)\ge\frac14
\quad\Longleftrightarrow\quad
r\le\eta
\label{eq:D3.36}
\end{equation}
for every integer \(r\in[0,b]\).  Hence

\begin{equation}
M_0=\lfloor\eta\rfloor+1=M.
\label{eq:D3.37}
\end{equation}
The hypothesis \(M\le b\) and integrality of \(M\) imply
\(M\le\lfloor b\rfloor\).  Since
\(\mathcal S_k\subseteq\mathcal S_0\),

\begin{equation}
1\le M_k\le M_0\le\lfloor b\rfloor.
\label{eq:D3.38}
\end{equation}
Maximality in \eqref{eq:D3.34} gives

\begin{equation}
f(M_k-1)\ge k+\frac14,
\qquad
f(M_k)<k+\frac14.
\label{eq:D3.39}
\end{equation}
The sets \(\mathcal S_k\) are nested, so \(M_{k+1}\le M_k\).  Equality is
impossible for \(0\le k\le K-2\): if
\(M_{k+1}=M_k=r\), then \eqref{eq:D3.39} would imply

\begin{equation}
f(r-1)\ge k+\frac54,
\qquad
f(r)<k+\frac14,
\label{eq:D3.40}
\end{equation}
a drop greater than one over an interval of length one, contradicting the
\(1/2\)-Lipschitz bound.  Also \(0\in\mathcal S_{K-1}\), so
\(M_{K-1}\ge1>M_K\).  Thus

\begin{equation}
M_0>M_1>\cdots>M_K=0.
\label{eq:D3.41}
\end{equation}
For \(0\le k\le K-2\), take

\begin{equation}
i=M_{k+1},\qquad j=M_k,\qquad n=k+1
\label{eq:D3.42}
\end{equation}
in \eqref{eq:D3.28}.  Monotonicity and \eqref{eq:D3.39} verify every inequality in \eqref{eq:D3.28}.  For the
terminal block take

\begin{equation}
i=M_K=0,\qquad j=M_{K-1},\qquad n=K.
\label{eq:D3.43}
\end{equation}
The definition of \(K\) gives \(f(0)<K+1/4\), and \eqref{eq:D3.39} gives the other
endpoint inequalities.  Therefore \eqref{eq:D3.32} applies to every block
\([M_{k+1},M_k]\).

Summing \eqref{eq:D3.32}, the half-weights at internal endpoints join into full
weights, while both the integrals and the block lengths telescope.  By
\eqref{eq:D3.39}, monotonicity, and nonnegativity,

\begin{equation}
\left\lfloor f(r)+\frac34\right\rfloor=0
\qquad(M_0\le r\le\lfloor b\rfloor).
\label{eq:D3.44}
\end{equation}
The summed inequality is

\begin{equation}
\frac K2+
\sum_{r=1}^{\lfloor b\rfloor}
\left\lfloor f(r)+\frac34\right\rfloor
\ge
\int_0^{M_0}f(s)\,ds+\frac{M_0}{4}-\frac K2.
\label{eq:D3.45}
\end{equation}
Adding \(K/2\) to both sides and using
\(K=\lfloor f(0)+3/4\rfloor\) and \(M_0=M\) gives \eqref{eq:D3.27}.
\end{proof}

\subsection{Rational enclosures for constants}

\label{supp:constant-proofs}

The exact-arithmetic estimates use the componentwise enclosure

\begin{equation}
\begin{aligned}
10^{-6}(885341,2588753,3830649,4894852)
&<(\alpha_1,\alpha_2,\alpha_3,\alpha_4)\\
&<10^{-6}(885342,2588755,3830650,4894854).
\end{aligned}
\label{eq:Const.alpha-fine}
\end{equation}
For calculations where shorter endpoints are preferable, we use its outward
rounding

\begin{equation}
\begin{array}{c|cc}
n&\alpha_n^-&\alpha_n^+\\ \hline
1&8853/10^4&8854/10^4\\
2&25887/10^4&25888/10^4\\
3&38305/10^4&38307/10^4\\
4&48947/10^4&48949/10^4 .
\end{array}
\label{eq:Const.alpha-coarse}
\end{equation}
We also need the sharper enclosure

\begin{equation}
\frac{3141592}{10^6}<\pi<\frac{3141593}{10^6}.
\label{eq:L.4.13a}
\end{equation}
Combining these estimates, and verifying radicals by squaring, gives

\begin{equation}
3<\frac{25}{8}<\frac{333}{106}
<\frac{3141592}{10^6}<\pi
<\frac{355}{113}<\frac{3141593}{10^6}<\frac{22}{7},
\label{eq:Const.pi}
\end{equation}
together with

\begin{equation}
\frac{707}{500}<\sqrt2<\frac32,\quad
\sqrt3<\frac74,\quad
\frac{682}{305}<\sqrt5<\frac{223607}{100000},\quad
\frac{120}{49}<\sqrt6<\frac{244949}{100000},\quad
\sqrt7>\frac{21}{8}.
\label{eq:Const.rad}
\end{equation}
The outer classical bounds in \eqref{eq:Const.pi}, namely
\(25/8<\pi<22/7\), are proved here.  The identity

\begin{equation}
\frac{22}{7}-\pi
=
\int_0^1\frac{t^4(1-t)^4}{1+t^2}\,dt>0
\label{eq:D3.19}
\end{equation}
proves the upper bound.

For the lower bound, put

\begin{equation}
a=\arctan\frac15,
\qquad
b=\arctan\frac1{239}.
\label{eq:D3.20}
\end{equation}
Two applications of the tangent double-angle formula give

\begin{equation}
\tan(4a)=\frac{120}{119}>0.
\label{eq:D3.21}
\end{equation}
Since \(0<a<\pi/4\), one has \(0<4a<\pi\); hence \eqref{eq:D3.21} places \(4a\)
in \((0,\pi/2)\).  The subtraction formula then gives

\begin{equation}
\tan(4a-b)
=
\frac{120/119-1/239}{1+(120/119)(1/239)}
=1.
\label{eq:D3.22}
\end{equation}
As \(0<4a-b<\pi/2\), this proves Machin's identity

\begin{equation}
\frac\pi4=4\arctan\frac15-\arctan\frac1{239}.
\label{eq:D3.23}
\end{equation}
The alternating arctangent series, with its usual one-term remainder
sign, now yields

\begin{equation}
\frac\pi4
>
4\left(\frac15-\frac1{3\cdot5^3}\right)-\frac1{239}
=
\frac{70369}{89625}
>
\frac{25}{32}.
\label{eq:D3.24}
\end{equation}
The difference in the last comparison is
\(11183/2868000>0\).  This proves \(\pi>25/8\).

The sharper bounds \eqref{eq:L.4.13a} use the same Machin
identity.  With \(y=1/5\) and \(v=1/239\), alternating partial sums give

\begin{equation*}
\begin{aligned}
\pi&>
4\left[
4\left(y-\frac{y^3}{3}+\frac{y^5}{5}-\frac{y^7}{7}
+\frac{y^9}{9}-\frac{y^{11}}{11}\right)
-\left(v-\frac{v^3}{3}+\frac{v^5}{5}\right)
\right],\\
\pi&<
4\left[
4\left(y-\frac{y^3}{3}+\frac{y^5}{5}-\frac{y^7}{7}
+\frac{y^9}{9}\right)
-\left(v-\frac{v^3}{3}\right)
\right].
\end{aligned}
\end{equation*}
Direct cross-multiplication shows that the first rational number exceeds
\(3141592/10^6\), while the second is below \(3141593/10^6\).  The same
upper partial sum also lies below \(355/113\).

Thus all later uses of
\(314159/10^5<\pi<355/113\) follow from these same Machin bounds.
Formula
\(\alpha_n^3=9(4n-3)^2\pi^2/128\) and cubing prove
\eqref{eq:Const.alpha-fine}; \eqref{eq:Const.alpha-coarse} is its outward
rounding.

All radical bounds in \eqref{eq:Const.rad} follow by squaring
positive rational endpoints.  The remaining comparisons in
\eqref{eq:Const.pi} are direct cross-multiplications.

\section{Disk Ritz determinants}
\label{supp:disk-ritz}

\begin{proof}[Proof of Proposition \ref{prop:disk-ritz}]

\begin{equation}
\det(A_0-\lambda B_0)
=-\frac{\lambda(\lambda^2-144\lambda+1920)}{17280}.
\label{eq:D2.3.5}
\end{equation}
Since \(13/2<\sqrt{51}<9\), the ordered generalized eigenvalues are
therefore

\begin{equation*}
0,\qquad 72-8\sqrt{51},\qquad72+8\sqrt{51}.
\end{equation*}
Thus min--max gives

\begin{equation}
j_{1,1}^2
\leq72-8\sqrt{51}<20.
\label{eq:D2.3.6}
\end{equation}

Next take \(m=1\) and
\(V_1=\operatorname{span}\{r,r^3,r^5\}\).  Exact evaluation at
\(\lambda=32\) gives

\begin{equation}
\det(A_1-32B_1)=\frac{29}{675}>0,
\qquad
\det B_1=\frac1{345600}>0.
\label{eq:D2.3.7}
\end{equation}
Let \(\theta_1\leq\theta_2\leq\theta_3\) be these Ritz values.  The trial
\(r\in V_1\) has quotient \(4\), so
\(\theta_1\leq4<32\), while

\begin{equation*}
\det(A_1-32B_1)
=\det B_1\prod_{i=1}^3(\theta_i-32)>0.
\end{equation*}
Since \(\theta_1-32<0\) and the product is positive, exactly two factors are
negative.  Thus \(\theta_2<32\), and min--max yields

\begin{equation}
(j'_{1,2})^2<32.
\label{eq:D2.3.8}
\end{equation}

Finally, for \(m=4\) and
\(V_4=\operatorname{span}\{r^4,r^6\}\),

\begin{equation}
\det(A_4-\lambda B_4)
=\frac{\lambda^2-264\lambda+6720}{5040}.
\label{eq:D2.3.9}
\end{equation}
Since \(25<\sqrt{669}<33\), the smaller Ritz value
\(132-4\sqrt{669}\) is positive and strictly less than \(32\).  Hence

\begin{equation}
(j'_{4,1})^2<32.
\label{eq:D2.3.10}
\end{equation}
\end{proof}

\section{High-frequency phase estimate for the disk}
\label{supp:disk-analytic-tail}

\subsection{Phase count and auxiliary bounds}

With \(G_x\) from \eqref{eq:I.9}, set

\begin{equation*}
f_m(x):=\left\lfloor G_x(m)+\frac34\right\rfloor .
\end{equation*}
By \eqref{eq:D2.1.3} and Proposition \ref{prop:flps-phase}, with angular
weights,

\begin{equation}
\mathcal N_2^{\le}(x)
\geq f_0(x)+2\sum_{m=1}^{\lfloor x\rfloor}f_m(x).
\label{eq:D2.2.3}
\end{equation}
Its right-hand side is the \(\alpha=0\) weighted lattice expression in the
FLPS argument; only the displayed inequality, not an identification of the
counts, is used.

The properties and moments of \(G_x\) needed below are recorded in
\eqref{eq:D3.5}--\eqref{eq:D3.9}.

For \(x>1\),

\begin{equation*}
\frac{\partial}{\partial x}G_x(1)
=\frac{\sqrt{x^2-1}}{\pi x}>0.
\end{equation*}
Hence \(G_x(1)\) increases with \(x\).  For \(x\geq3\),

\begin{equation}
G_3(1)
=\frac{2\sqrt2-\arccos(1/3)}{\pi}
>\frac{2\sqrt2}{\pi}-\frac12
>\frac14.
\label{eq:D2.2.11}
\end{equation}
The last inequality uses \(\sqrt2>7/5\) and \(\pi<22/7\) from
\eqref{eq:Const.rad} and \eqref{eq:Const.pi}.

Let

\begin{equation}
L:=G_x^{-1}\!\left(\frac14\right),
\qquad
\beta:=\lfloor L\rfloor+1.
\label{eq:D2.2.12}
\end{equation}
We will also use

\begin{equation}
L<x-1\qquad(x\geq2).
\label{eq:D2.2.13}
\end{equation}
Since \(G_x\) is decreasing, \eqref{eq:D2.2.13} is equivalent to
\(H(x):=G_x(x-1)<1/4\).  At \(x=2\),

\begin{equation*}
H(2)=\frac{\sqrt3}{\pi}-\frac13<\frac14,
\end{equation*}
using \(\sqrt3<7/4\) from \eqref{eq:Const.rad} and
\(\pi>3\) from \eqref{eq:Const.pi}.
Moreover,

\begin{equation}
H'(x)=\frac1\pi\left(
\frac{\sqrt{2x-1}}x-\arccos\left(1-\frac1x\right)
\right)<0.
\label{eq:D2.2.14}
\end{equation}
For \(t=\sqrt{2x-1}/x\) and \(x\geq2\),

\begin{equation*}
0<t=\frac{\sqrt{2x-1}}x<1<\pi.
\end{equation*}
Since \(\cos t>1-t^2/2\),

\begin{equation*}
\cos t>1-\frac{2x-1}{2x^2}>1-\frac1x;
\end{equation*}
and monotonicity of cosine on \([0,\pi]\) proves
\eqref{eq:D2.2.14}, hence \eqref{eq:D2.2.13}.  In particular,

\begin{equation}
\beta<x,\qquad \beta-1<x-1,
\label{eq:D2.2.15}
\end{equation}
so both applications below satisfy the endpoint hypothesis of Lemma
\ref{lem:retained-floor}.

\begin{lemma}[Terminal integral bound]
\label{lem:disk-terminal-integral}
For \(x>0\) and \(0\leq a\leq x\),

\begin{equation}
\int_a^xG_x(z)\,dz
\leq \frac25(x-a)G_x(a).
\label{eq:D2.2.16}
\end{equation}
\end{lemma}

\begin{proof}
Write \(a=\gamma x\), \(0\leq\gamma\leq1\), and set

\begin{equation*}
F(\gamma)
=\int_{\gamma x}^xG_x(z)\,dz
-\frac25x(1-\gamma)G_x(\gamma x).
\end{equation*}
Then \(F(1)=0\), and direct differentiation for \(0\leq\gamma<1\) gives

\begin{equation*}
F'(\gamma)=\frac{x^2(\gamma+2)}{5\pi}
\left(
\arccos\gamma-\frac{3\sqrt{1-\gamma^2}}{\gamma+2}
\right).
\end{equation*}
Writing

\begin{equation*}
\chi(\gamma):=\arccos\gamma-
\frac{3\sqrt{1-\gamma^2}}{\gamma+2},
\end{equation*}
we have

\begin{equation*}
\chi(1)=0,\qquad
\chi'(\gamma)=-
\frac{(1-\gamma)^2}
{(\gamma+2)^2\sqrt{1-\gamma^2}}\leq0.
\end{equation*}
Thus \(\chi(\gamma)\geq0\) and \(F'(\gamma)\geq0\).  Since \(F(1)=0\),
\(F(\gamma)\leq0\) for \(0\leq\gamma<1\); continuity gives the endpoint
and proves \eqref{eq:D2.2.16}.
\end{proof}

\subsection{Applying the two floor sums at flux zero}

Rewrite the right-hand side of \eqref{eq:D2.2.3} as

\begin{equation}
\begin{aligned}
S_0(x)&:=\sum_{m=0}^{\lfloor x\rfloor}
\left\lfloor G_x(m)+\frac34\right\rfloor,\\
S_1(x)&:=\sum_{m=0}^{\lfloor x-1\rfloor}
\left\lfloor G_x(m+1)+\frac34\right\rfloor.
\end{aligned}
\label{eq:D2.2.17}
\end{equation}
Then

\begin{equation}
\mathcal N_2^{\le}(x)\geq S_0(x)+S_1(x).
\label{eq:D2.2.18}
\end{equation}
For \(x\geq3\), apply Lemma \ref{lem:retained-floor} first to
\(g(z)=G_x(z)\) on
\([0,x]\), with \(M=\beta\), and then to \(g(z)=G_x(z+1)\) on
\([0,x-1]\), with \(M=\beta-1\).  Equations \eqref{eq:D2.2.11} and
\eqref{eq:D2.2.15} give the required initial-value and endpoint bounds.
For the shifted function,
\(g^{-1}(1/4)=L-1\) and
\(\lfloor L-1\rfloor+1=\lfloor L\rfloor=\beta-1\), including when
\(L\) is an integer.  We obtain

\begin{equation}
\begin{aligned}
S_0(x)+S_1(x)
&\geq \int_0^\beta G_x(z)\,dz+\frac\beta4
+\int_1^\beta G_x(z)\,dz+\frac{\beta-1}{4}\\
&=\frac{x^2}{4}
-\int_0^1G_x(z)\,dz
-2\int_\beta^xG_x(z)\,dz
+\frac{2\beta-1}{4}.
\end{aligned}
\label{eq:D2.2.19}
\end{equation}

\subsection{A sharpened high-frequency estimate}

Assume from now on that \(x\geq20/3\).

Convexity of \(G_x\) on \([0,1]\) gives the trapezoidal bound

\begin{equation}
\int_0^1G_x(z)\,dz
\leq\frac{G_x(0)+G_x(1)}2.
\label{eq:D2.2.22}
\end{equation}

Put \(u=1/x\leq3/20\).  Since the integrand defining \(\arcsin\) is
increasing,

\begin{equation}
\arcsin u
=\int_0^u\frac{dt}{\sqrt{1-t^2}}
\leq\frac{u}{\sqrt{1-u^2}}
<\frac{51}{50}u.
\label{eq:D2.2.23}
\end{equation}

For the strict comparison, it is enough to square:

\begin{equation*}
\frac1{1-u^2}\leq\frac{400}{391}
<\frac{2601}{2500}.
\end{equation*}
Also,

\begin{equation}
\sqrt{x^2-1}<x-\frac1{2x},
\label{eq:D2.2.24}
\end{equation}
because the square of the right side exceeds \(x^2-1\) by
\(1/(4x^2)\).  Using
\(\arccos(1/x)=\pi/2-\arcsin(1/x)\), equations \eqref{eq:D2.2.23}--\eqref{eq:D2.2.24} yield

\begin{equation}
G_x(1)
<\frac{x}{\pi}-\frac12+\frac{13}{25\pi x}.
\label{eq:D2.2.25}
\end{equation}
Substitution in \eqref{eq:D2.2.22} gives

\begin{equation}
\int_0^1G_x(z)\,dz
<\frac{x}{\pi}-\frac14+\frac{13}{50\pi x}.
\label{eq:D2.2.26}
\end{equation}
Lemma \ref{lem:disk-terminal-integral} and \(\beta>L\) give

\begin{equation}
\int_\beta^xG_x(z)\,dz
\leq\frac25(x-\beta)G_x(\beta)
<\frac{x-\beta}{10}.
\label{eq:D2.2.27}
\end{equation}
Equations \eqref{eq:D2.2.19}, \eqref{eq:D2.2.26}, and
\eqref{eq:D2.2.27}, together with \(\beta>L\), give

\begin{equation}
\mathcal N_2^{\le}(x)
>
\frac{x^2}{4}
+\frac7{10}L
-x\left(\frac1\pi+\frac15\right)
-\frac{13}{50\pi x}.
\label{eq:D2.2.28}
\end{equation}

Next we show

\begin{equation}
L>\frac{3x}{4}.
\label{eq:D2.2.29}
\end{equation}
Set

\begin{equation*}
c_{3/4}:=\frac{\sqrt7}{4}-\frac34\arccos\frac34.
\end{equation*}
With \(0<y=1/\sqrt7<1\),

\begin{equation*}
\cos(2\arctan y)=\frac{1-y^2}{1+y^2}=\frac34,
\qquad
2\arctan y\in(0,\pi/2).
\end{equation*}
Hence the principal inverse cosine and the alternating arctangent series give

\begin{equation}
\arccos\frac34
=2\arctan y
<2\left(y-\frac{y^3}{3}+\frac{y^5}{5}\right).
\label{eq:D2.2.30}
\end{equation}
Consequently,

\begin{equation}
\begin{aligned}
c_{3/4}
&>\frac{\sqrt7}{4}
-\frac32\left(
\frac1{\sqrt7}-\frac1{3\cdot7\sqrt7}
+\frac1{5\cdot49\sqrt7}
\right)\\
&=\frac{309\sqrt7}{6860}
>\frac{927}{7840}
=\frac{33}{280}+\frac3{7840}
>\frac{33}{280}.
\end{aligned}
\label{eq:D2.2.31}
\end{equation}
Here \(\sqrt7>21/8\) follows by squaring; the squared difference is \(7/64\).

Therefore, using \(x\geq20/3\) and \(\pi<22/7\),

\begin{equation}
G_x\!\left(\frac{3x}{4}\right)
=\frac{x c_{3/4}}{\pi}
>
\frac{20}{3}\frac{33}{280}\frac7{22}
=\frac14.
\label{eq:D2.2.32}
\end{equation}
Since \(G_x\) is strictly decreasing and \(G_x(L)=1/4\), this proves
\eqref{eq:D2.2.29}.  Substituting it into \eqref{eq:D2.2.28} yields

\begin{equation}
\mathcal N_2^{\le}(x)
>
\frac{x^2}{4}
+x\left(\frac{13}{40}-\frac1\pi\right)
-\frac{13}{50\pi x}.
\label{eq:D2.2.33}
\end{equation}

Finally, \(\pi>25/8\) implies

\begin{equation*}
x\left(\frac{13}{40}-\frac1\pi\right)
>\frac{x}{200}\geq\frac1{30},
\end{equation*}
whereas \(\pi>3\) and \(x\geq20/3\) imply

\begin{equation*}
\frac{13}{50\pi x}<\frac{13}{1000}.
\end{equation*}
The remaining difference is

\begin{equation*}
\frac1{30}-\frac{13}{1000}
=\frac{61}{3000}>0.
\end{equation*}
We have proved

\begin{equation}
\mathcal N_2^{\le}(x)>\frac{x^2}{4}
\quad\text{for every }x\geq\frac{20}{3}.
\label{eq:D2.2.34}
\end{equation}

If \(x>20/3\), choose \(20/3\leq y<x\).  Then

\begin{equation*}
\mathcal N_2^{<}(x)
=\#\{\mu_j<x^2\}
\geq\#\{\mu_j\leq y^2\}
=\mathcal N_2^{\le}(y)
>\frac{y^2}{4}.
\end{equation*}
Letting \(y\uparrow x\) (the left-hand side is independent of \(y\)) gives

\begin{equation*}
\mathcal N_2^{<}(x)\geq\frac{x^2}{4}.
\end{equation*}
Therefore

\begin{equation}
\mathcal N_2^{<}(x)\geq\frac{x^2}{4}
\quad\text{for every }x>\frac{20}{3}.
\label{eq:supp-disk-strict-tail}
\end{equation}
The variational bound in the body supplies the endpoint \(x=20/3\).

\section{Three-dimensional intermediate-range estimate}

\begin{proposition}[Three-dimensional intermediate-range estimate]
\label{supp:d3-phase}
The three bounds in \eqref{eq:D3.finite.4} hold.
\end{proposition}

\begin{proof}
Put \(t=a^{-1/3}\).  The constants satisfy the following exact rational
enclosures:

\begin{equation*}
\frac{433}{250}<\sqrt3<\frac{1733}{1000},
\qquad
\frac{500}{5199}<e_3<\frac{125}{1299},
\end{equation*}

\begin{equation*}
\frac{333}{106}\frac{433}{500}
<\mathcal D_3<
\frac{3141593}{10^6}\frac{1733}{2000},
\end{equation*}
Let \(\alpha_n^-\) and \(\alpha_n^+\) denote the endpoints in
\eqref{eq:Const.alpha-coarse}.  All these enclosures follow from the common
\(\pi,\sqrt3\) bounds by squaring or cubing.
Write

\begin{equation*}
\mathcal D_3^-:=\frac{333}{106}\frac{433}{500},\qquad
\mathcal D_3^+:=\frac{3141593}{10^6}\frac{1733}{2000},
\qquad
e^-:=\frac{500}{5199},\qquad e^+:=\frac{125}{1299},
\qquad c_n^\pm:=\frac{\alpha_n^\pm}{3}.
\end{equation*}
For the \(M\)-row of \eqref{eq:D3.finite.4}, the rational \(t\)-box and
included-radial margin are

\begin{equation}
\begin{array}{c|c|c|c}
M&L_M&U_M&
\text{lower bound for included \(r_{3,n}\)}\\ \hline
2&4189/5000&1013/1000&>1/100\\
3&303/400&8379/10^4&>1/25\\
4&3593/5000&947/1250&>1/50 .
\end{array}
\label{eq:S3.2}
\end{equation}
Cubing proves the inclusions, with \(U_2^6>27/25\) and
\(L_4^6<27/196\) at the radical endpoints.  Evaluating
\(1-\alpha_M^+U_M^2/3-e^+U_M^3\) at unfavourable endpoints and
cross-multiplying gives the three radial bounds in \eqref{eq:S3.2}.

For \(t>0\), direct expansion and outward endpoint substitution give

\begin{equation}
\begin{aligned}
\mathcal P_{3,M}(t)-1
\ge{}&
-1+M\mathcal D_3^-t^3
-2\mathcal D_3^+\sum_{n=1}^{M}c_n^+\,t^5
-2M\mathcal D_3^+e^+t^6\\
&+\mathcal D_3^-\sum_{n=1}^{M}(c_n^-)^2t^7
+2\mathcal D_3^-e^-\sum_{n=1}^{M}c_n^-t^8
+M\mathcal D_3^-(e^-)^2t^9
\ge Q_M(t),
\end{aligned}
\label{eq:S3.3}
\end{equation}
where

\begin{equation*}
Q_M(t):=-1+A_Mt^3-B_Mt^5-C_Mt^6+D_Mt^7+E_Mt^8+F_Mt^9
\end{equation*}
and

\begin{equation}
\begin{array}{c|cccccc}
M&A_M&B_M&C_M&D_M&E_M&F_M\\ \hline
2&5441/1000&1261/200&10479/10^4&11313/5000&6059/10^4&503/10^4\\
3&5101/625&132569/10^4&7859/5000&66979/10^4&12741/10^4&377/5000\\
4&108821/10^4&110701/5000&20957/10^4&697/50&10639/5000&503/5000 .
\end{array}
\label{eq:S3.4}
\end{equation}
The positive coefficients in \eqref{eq:S3.4} are rounded downward and
the magnitudes \(B_M,C_M\) upward; every comparison is a multiplication
of positive integers.

Define

\begin{equation}
H_M(t):=\frac{Q_M'(t)}{t^2}
=3A_M-5B_Mt^2-6C_Mt^3+7D_Mt^4+8E_Mt^5+9F_Mt^6.
\label{eq:S3.5}
\end{equation}
On \(L_M\le t\le U_M\),

\begin{equation*}
H_M''(t)\ge R_M:=
-10B_M-36C_MU_M+84D_ML_M^2
+160E_ML_M^3+270F_ML_M^4.
\end{equation*}
The sign table for \(M=2,3\) is

\begin{equation}
\begin{array}{c|cc|cc|cc|c}
M&Q_M(L_M)&Q_M(U_M)&H_M(L_M)&H_M(U_M)&
H_M'(L_M)&H_M'(U_M)&R_M\\ \hline
2&>47/1000&>3/1000&>9/20&<-11/50&<-15&>11&>95\\
3&>47/1000&>49/1000&>9/20&<-1/25&<-17&>6&>238 .
\end{array}
\label{eq:S3.6}
\end{equation}
Because \(H_M''>0\) and
\(H_M'(L_M)<0<H_M'(U_M)\), the derivative \(H_M'\) has a unique zero
\(\tau_M\).  Thus \(H_M\) decreases on \([L_M,\tau_M]\) and increases
on \([\tau_M,U_M]\).  The signs
\(H_M(L_M)>0>H_M(U_M)\) give one zero of \(H_M\) on the decreasing
part; there can be no second zero on the increasing part because its
right endpoint is still negative.  Hence \(Q_M'=t^2H_M\) is first
positive and then negative.  Its minimum on the closed interval is therefore
at an endpoint and is positive by \eqref{eq:S3.6}.

For \(M=4\), with \(c_4=37/50\), exact substitution gives

\begin{equation*}
R_4>459,\qquad Q_4(L_4)>\frac{21}{500},\qquad
Q_4(U_4)>\frac{23}{500},
\end{equation*}

\begin{equation}
H_4'(c_4)>\frac25,\qquad
H_4(c_4)+H_4'(c_4)(L_4-c_4)>\frac1{10}.
\label{eq:S3.7}
\end{equation}
Convexity places \(H_4\) above its tangent at \(c_4\).  That tangent has
positive slope and is \(>1/10\) at \(L_4\), hence \(H_4>0\) on the
whole interval.  Thus \(Q_4'=t^2H_4>0\) and
\(Q_4(t)\ge Q_4(L_4)>0\).  Together with
\eqref{eq:S3.3}, this proves all three claims in
\eqref{eq:D3.finite.4}.
\end{proof}

\section{Four-dimensional intermediate estimates}

\begin{proposition}[Four-dimensional intermediate estimates]
\label{supp:d4-phase}
The rational estimates required for \eqref{eq:D4.11} and
\eqref{eq:D4.24} hold.
\end{proposition}

\begin{proof}
For the phase estimate retain

\begin{equation}
\begin{array}{c|cc}
n&\alpha_n^-&\alpha_n^+\\ \hline
1&44267/50000&17707/20000\\
2&2071/800&64719/25000\\
3&47883/12500&76613/20000\\
4&97897/20000&244743/50000 .
\end{array}
\label{eq:S4.1}
\end{equation}
Put \(b_n^\pm=\alpha_n^\pm/4\) and
\(S_{j,M}^\pm=\sum_{n=1}^M(b_n^\pm)^j\).  For the \(M\)-row of
\eqref{eq:D4.11}, the rational \(t=a^{-1/3}\) box and included-radial
margin are

\begin{equation}
\begin{array}{c|c|c|c}
M&L_M&U_M&
\text{lower bound for included \(r_{4,n}\)}\\ \hline
2&9687293/10^7&10172448/10^7&>1/4\\
3&8378836/10^7&9687294/10^7&>1/25\\
4&6751063/10^7&8378837/10^7&>1/10 .
\end{array}
\label{eq:S4.2}
\end{equation}
Cubing proves the inclusions; evaluating the radial polynomials at
unfavourable endpoints and cross-multiplying gives the three displayed
margins.

Coefficientwise endpoint substitution gives

\begin{equation}
\begin{aligned}
\mathscr U_M(t):={}&
\frac{8M}{3}t^3-\frac{200}{9}S_{1,M}^-t^5-2Mt^6
+\frac{392}{9}S_{2,M}^+t^7+\frac{64}{9}S_{1,M}^+t^8\\
&+\left(\frac{3M}{16}-24S_{3,M}^-\right)t^9
-\frac{50}{9}S_{2,M}^-t^{10}
-\frac{121}{432}S_{1,M}^-t^{11},
\end{aligned}
\label{eq:S4.3}
\end{equation}
with

\begin{equation}
\left(a\frac d{da}\right)^2\mathcal P_{4,M}(a)
\le\mathscr U_M(t).
\label{eq:S4.4}
\end{equation}

Set \(J_M=t^{-3}\mathscr U_M=\sum_{j=0}^8a_{M,j}t^j\).
Only \(j=0,2,\ldots,8\) occur, and

\begin{equation*}
a_{M,4},a_{M,5}>0,\qquad
a_{M,6},a_{M,7},a_{M,8}<0.
\end{equation*}
For \(L_M\le t\le U_M\),

\begin{equation}
J_M^{(4)}(t)\le K_M:=
24a_{M,4}+120a_{M,5}U_M+360a_{M,6}L_M^2
+840a_{M,7}L_M^3+1680a_{M,8}L_M^4.
\label{eq:S4.5}
\end{equation}
For \(M=2,4\), exact substitution gives

\begin{equation}
\begin{array}{c|ccccc}
M&K_M&J_M'''(L_M)&J_M'(L_M)&J_M'(U_M)&J_M(U_M)\\ \hline
2&<-3260&<-450&>3&>3&<-1/2\\
4&<-10904&<-783&>3&>1&<-1/5 .
\end{array}
\label{eq:S4.6}
\end{equation}
Since \(J_M^{(4)}<0\), the function \(J_M'''\) is decreasing; the sign
at \(L_M\) therefore gives \(J_M'''<0\) on the whole interval.  Equivalently,
\(J_M'\) is strictly concave.  A concave function lies above its endpoint
chord, so the two positive endpoint values give \(J_M'>0\).  Thus \(J_M\)
is increasing and
\(J_M(t)\le J_M(U_M)<0\) throughout the interval.

For \(M=3\), put \(s_3=9/10\), \(c_3=933/1000\), and
\(T_3(t)=J_3(c_3)+J_3'(c_3)(t-c_3)\).  Here
\(L_3<s_3<c_3<U_3\), and the exact sign table is

\begin{equation}
\begin{gathered}
K_3<-8163,\qquad J_3'''(L_3)<-1069,\\
J_3'(L_3)>5,\qquad J_3'(s_3)>3,\qquad
J_3(s_3)<-\frac{39}{100},\\
J_3''(s_3)<-80,\qquad
T_3(s_3)<-\frac{33}{100},\qquad
T_3(U_3)<-\frac{33}{100}.
\end{gathered}
\label{eq:S4.7}
\end{equation}
As above, \(J_3'''<0\), so \(J_3'\) is concave.  Its positive values at
\(L_3\) and \(s_3\) place it above the positive endpoint chord on that
interval.  Hence \(J_3\) is increasing there and
\(J_3(t)\le J_3(s_3)<0\).  On \([s_3,U_3]\), the inequality
\(J_3'''<0\), together with \(J_3''(s_3)<0\), gives \(J_3''<0\).
Thus \(J_3\) is concave and lies below its tangent \(T_3\) at \(c_3\).
The two endpoint values of this affine tangent are negative, so
\(T_3<0\) on the whole interval.  Therefore
\(\mathscr U_M=t^3J_M<0\) for \(M=2,3,4\).

The endpoint cube-root boxes are

\begin{equation}
\begin{array}{c|c}
a&[t^-,t^+]\\ \hline
19/20&[10172447,10172448]/10^7\\
11/10&[9687293,9687294]/10^7\\
17/10&[8378836,8378837]/10^7\\
13/4&[6751063,6751064]/10^7 .
\end{array}
\label{eq:S4.8}
\end{equation}
At the left and right endpoints of each \(M\)-row of \eqref{eq:D4.11},
outward rational substitution using \eqref{eq:S4.8} gives the following
bounds for \(\mathcal P_{4,M}-61/60\):

\begin{equation}
\begin{array}{c|cc}
M&\text{left endpoint}&\text{right endpoint}\\ \hline
2&>7/1000&>3/100\\
3&>3/100&>1/20\\
4&>1/20&>1/2000,
\end{array}
\label{eq:S4.9}
\end{equation}
Since \(a\,d/da=-(t/3)d/dt\), equations
\eqref{eq:S4.4}--\eqref{eq:S4.7} prove strict concavity in \(\log a\).
The endpoint chord and \eqref{eq:S4.9} prove Proposition
\ref{prop:d4-phase}.

For the beta-integral estimate the remaining assertions follow from the
rational inequalities

\begin{equation*}
\left(\frac{219}{100}\right)^3<
\left(\frac{13}{4}\right)^2,\qquad
3^3<\left(\frac{21}{4}\right)^2<
\left(\frac{303}{100}\right)^3,
\end{equation*}

\begin{equation}
\left(\frac{707}{500}\right)^2<2<
\left(\frac{99}{70}\right)^2,\qquad
\frac{333}{106}<\pi<\frac{3141593}{10^6}.
\label{eq:S4.10}
\end{equation}
Together with \eqref{eq:S4.1}, these inequalities verify the required
nonvanishing, \(\mathfrak h_4\)-bounds, rectangle \eqref{eq:D4.20}, and
endpoint residuals \eqref{eq:D4.23} used in the proof of \eqref{eq:D4.24}.
\end{proof}

\section{Power-trial monotonicity}

\begin{proposition}
\label{supp:power-trial-monotonicity}
After the substitutions
\[
 m=1+u,\qquad c=3+v,\qquad s=1+t,
 \qquad u,v,t\ge0,
\]
the numerator in \eqref{eq:L.1.power.monotone}, reduced by
\(2s^3=m(m+c)\), is
\[
 P_0(u,v)+tP_1(u,v)+t^2P_2(u,v),
\]
where every coefficient in the three polynomials below is positive.
\end{proposition}

\begin{proof}
Multiply the left side minus the right side of
\eqref{eq:L.1.power.monotone} by the positive denominator
\[
\begin{aligned}
D={}&3m(c+m)(c+2s)(2m-1)(2s+1)\\
&\quad\cdot(c+2m-1)(c+2s+2)(2c+2m-1).
\end{aligned}
\]
Polynomial division of the resulting numerator by
\(2s^3-m(m+c)\) gives the remainder specified in the proposition, with
\begin{align*}
P_0={}&96u^6+(240v+1104)u^5\\
&+(8v^3+296v^2+2504v+5728)u^4\\
&+(16v^4+288v^3+2376v^2+9268v+13324)u^3\\
&+(18v^5+370v^4+2984v^3+11900v^2+23490v+18270)u^2\\
&+(10v^6+268v^5+2802v^4+14784v^3+41446v^2\\
&\hspace{22mm}{}+57652v+29822)u\\
&+2v^6+54v^5+556v^4+2794v^3+7042v^2+7736v+1856,\\[1mm]
P_1={}&(16v^2+352v+1072)u^4\\
&+(8v^3+312v^2+1992v+3448)u^3\\
&+(8v^5+212v^4+1944v^3+8012v^2+14664v+8832)u^2\\
&+(8v^6+248v^5+2840v^4+15942v^3+46550v^2\\
&\hspace{22mm}{}+65942v+33422)u\\
&+2v^6+61v^5+682v^4+3662v^3+9844v^2+11793v+3716,\\[1mm]
P_2={}&(96v+384)u^4+(96v^2+768v+1536)u^3\\
&+(32v^4+456v^3+2376v^2+5280v+4104)u^2\\
&+(32v^5+568v^4+3936v^3+13152v^2+20736v+11784)u\\
&+8v^5+140v^4+936v^3+2916v^2+4008v+1632.
\end{align*}
The constant term is \(1856\), so the remainder is strictly positive.
\end{proof}

\section{Weighted finite-level estimate in dimensions five and six}

\begin{proposition}[Weighted finite-level estimate in dimensions five and six]
\label{supp:d56-estimate}
For \(d\in\{5,6\}\), \(M\in\{2,3,4\}\), and
\(a\in I_{d,M}\) from \eqref{eq:four-layer-bands},
every retained level \(n\le M\) is active and
\[
 \widehat L_{d,M}(a)>1,
 \qquad
 (a\partial_a)^2\widehat L_{d,M}(a)<0.
\]
Thus \(\widehat L_{d,M}\) is strictly concave as a function of
\(\log a\) on each interval \(I_{d,M}\).
\end{proposition}

\begin{proof}
We use the retained weights defined in \eqref{eq:L.retained-definition}.
Let \(\bar\alpha_n\) be the upper endpoint for \(\alpha_n\) in
\eqref{eq:Const.alpha-fine}.  The rational upper bounds
\[
\begin{array}{c|ccccc}
 a&4/5&41/50&11/10&17/10&5/2\\ \hline
 10^6\tau_a&1077218&1068388&968730&837884&736807
\end{array}
\]
satisfy \(a^{-1/3}<\tau_a\) by cubing positive rationals.  In particular,
\[
 \bar\alpha_2\tau_{4/5}^2<\frac{18}{5},\qquad
 \bar\alpha_3\tau_{11/10}^2<\frac{18}{5},\qquad
 \bar\alpha_4\tau_{17/10}^2<\frac{18}{5}.
\]
Since \(\alpha_n\) increases with \(n\) and \(a^{-2/3}\) decreases with
\(a\), it suffices in each band to use its largest retained level at the
left endpoint.
Since \(a\ge4/5\) and \(d\ge5\),
\[
 \frac1{2a\sqrt d}<\frac7{25},
 \qquad
 r_{d,n}(a)
 =1-\frac1d\left(\frac{\alpha_n}{a^{2/3}}
                 +\frac1{2a\sqrt d}\right)
 >\frac{28}{125}.
\]
Moreover
\[
 \frac{\sqrt{\beta_d}}a<\frac7{50};
\]
the largest square on the left is
\(125/6912<(7/50)^2\), at \((d,a)=(6,4/5)\).
Thus \(r_{d,n}>\sqrt{\beta_d}/a\) throughout each interval, so
all retained terms are on their smooth positive branch before
differentiation.

Put
\[
 t=a^{-1/3},\qquad c_n=\frac{\alpha_n}{d},\qquad
 e_d=\frac1{2d^{3/2}},\qquad p_n(t)=1-c_nt^2-e_dt^3.
\]
On an active interval,
\[
 \widehat L_{d,M}(a)=\mathcal D_d\widehat F_{d,M}(t),
\]
where
\[
 \widehat F_{d,M}(t)
 =t^3\sum_{n=1}^{M}\omega_{M,n}p_n(t)^{d-1}
 -\beta_dt^9\sum_{n=1}^{M}\omega_{M,n}p_n(t)^{d-3}.
\]
Write \(\Theta=a\partial_a=-(t/3)\partial_t\).  To bound
\(\Theta^2\widehat F_{d,M}\), let \(\alpha_n^-\) and \(\alpha_n^+\)
denote the endpoints in \eqref{eq:Const.alpha-coarse}, and use
\[
\begin{array}{c|cc}
 d&10^5e_d^-&10^5e_d^+\\ \hline
 5&4472&4473\\
 6&3402&3403
\end{array}.
\]
The relevant \(t\)-intervals are contained in the following six rational
intervals:
\[
\begin{array}{c|ccc}
 &M=2&M=3&M=4\\ \hline
 d=5&[121/125,1069/1000]&[837/1000,970/1000]
     &[736/1000,839/1000]\\
 d=6&[121/125,1078/1000]&[837/1000,970/1000]
     &[736/1000,839/1000].
\end{array}
\]
These enclosures follow by cubing or squaring positive rationals.

We choose one-sided endpoint bounds according to coefficient sign only after
forming each full weighted power sum.  For \(0\le p\le j\), set
\[
 S_{p,d,M}^{\pm}
 =\sum_{n=1}^{M}\omega_{M,n}
   \left(\frac{\alpha_n^\pm}{d}\right)^p,
 \qquad
 E_{j,p}^{\pm}=(e_d^\pm)^{j-p}S_{p,d,M}^{\pm}.
\]
In particular,
\((S_{0,d,2},S_{0,d,3},S_{0,d,4})=(2,5/2,9/2)\).
For a scalar \(\sigma\), define
\[
 \mathcal A(\sigma;X^-,X^+)
 =\begin{cases}
   \sigma X^+,&\sigma\ge0,\\
   \sigma X^-,&\sigma<0.
  \end{cases}
\]
Expanding \(p_n(t)^r\) and collecting the weighted sums gives the
coefficientwise rational majorant
\[
\begin{aligned}
 U_{d,M}(t)={}&
 \sum_{j=0}^{d-1}\sum_{p=0}^{j}
 \frac{(3+3j-p)^2}{9}\,
 \mathcal A\left(
  (-1)^j\binom{d-1}{j}\binom jp;
  E_{j,p}^-,E_{j,p}^+
 \right)t^{3+3j-p}\\
 &+\sum_{j=0}^{d-3}\sum_{p=0}^{j}
 \frac{(9+3j-p)^2}{9}\,
 \mathcal A\left(
  -\beta_d(-1)^j\binom{d-3}{j}\binom jp;
  E_{j,p}^-,E_{j,p}^+
 \right)t^{9+3j-p},
\end{aligned}
\]
for which \(\Theta^2\widehat F_{d,M}(t)\le U_{d,M}(t)\).
For each of the six rational intervals \([\ell,u]\), set
\(m=(\ell+u)/2\), \(h=(u-\ell)/2\), and write
\[
 U_{d,M}(m+h\xi)=\sum_{j=0}^{3d}q_j\xi^j,
 \qquad R_{d,M}=\sum_{j=1}^{3d}|q_j|.
\]
Exact rational cross multiplication in the displayed double sum yields
\[
 \max q_0<-\frac{69}{1000},\qquad
 \max R_{d,M}<\frac{51}{1000},\qquad
 \max(q_0+R_{d,M})<-\frac9{500};
\]
both separate extrema occur in the \((d,M)=(5,4)\) case.  Hence, for
\(|\xi|\le1\),
\[
 \frac{\Theta^2\widehat L_{d,M}}{\mathcal D_d}
 \le U_{d,M}(m+h\xi)<-\frac9{500}<0.
\]

For the endpoint checks, supplement the upper constants above with
\[
 \mathcal D_5>\frac{26342}{10000},\qquad
 \mathcal D_6>\frac{26127}{10000},
 \qquad
 \bar e_5=\frac{4472136}{10^8},\quad
 \bar e_6=\frac{3402070}{10^8},
\]
and put
\[
 \underline r_{d,n}(a)
 =1-\frac{\bar\alpha_n}{d}\tau_a^2-\frac{\bar e_d}{a}.
\]
Direct rational substitution gives \(\underline r_{d,n}>6/25\) and
\(\beta_d/a^2<1/50\) for every retained radius in the following eight
independent endpoint checks:
\[
\begin{array}{c|c|c|c}
 d&M&a&\text{lower bound for }\widehat L_{d,M}(a)\\ \hline
 5&2&41/50&>201/200\\
 5&2&11/10&>21/20\\
 5&3&17/10&>26/25\\
 5&4&5/2&>26/25\\
 6&2&4/5&>101/100\\
 6&2&11/10&>21/20\\
 6&3&17/10&>103/100\\
 6&4&5/2&>26/25.
\end{array}
\]
For example, at the smallest-margin endpoint
\(\underline r_{5,1}>7433/10000\),
\(\underline r_{5,2}>443/1250\), and
\(\beta_5/a^2=20/1681\), so
\[
 \frac{26342}{10000}\frac{50}{41}
 \sum_{\rho\in\{7433/10000,\,443/1250\}}
 \rho^2\left(\rho^2-\frac{20}{1681}\right)
 >\frac{201}{200}.
\]
The other seven rows use the same lower-radius substitution with the
appropriate dimension and weights.  The omitted left endpoints follow from
the nesting in \eqref{eq:L.retained-four}.  Thus both endpoints of every
interval exceed one.  Strict concavity in \(\log a\) places
\(\widehat L_{d,M}\) above its endpoint chord, proving the proposition.
\end{proof}

\paragraph{The exceptional left endpoint in dimension five.}
The retained-tail estimate does not reach \(a=4/5\) when \(d=5\).
Here \(\beta_5/a^2=1/80\), and the same rational cubing and squaring
comparisons give
\[
 \mathcal D_5<\frac{32929}{12500},\qquad
 r_1<\frac{73863}{10^5},\qquad
 r_2<\frac{34331}{10^5},\qquad
 r_3<\frac1{10}<\sqrt{\frac1{80}},\qquad r_4<r_3.
\]
Thus \(T_{5,3}=T_{5,4}=0\).  Monotonicity of
\(r^2(r^2-1/80)\) on its positive branch gives
\begin{equation}
\begin{aligned}
 \widehat L_{5,2}(4/5)
 &=\widehat L_{5,3}(4/5)
  =\widehat L_{5,4}(4/5)
  =L_{5,4}(4/5)\\
 &<\frac{32929}{12500}\frac54
 \sum_{r\in\{73863/10^5,\,34331/10^5\}}
 r^2\left(r^2-\frac1{80}\right)
 <\frac{999}{1000}<1.
\end{aligned}
\label{eq:L.d5.no-uniform}
\end{equation}

\section{\texorpdfstring{Uniform finite-level estimate for
\(d\ge7\)}{Uniform finite-level estimate for d>=7}}
\label{supp:aggregate-estimate}

For \(d\ge7\), write \(I_M:=I_{d,M}\), where \(I_{d,M}\) is defined in
\eqref{eq:four-layer-bands}; here \(a_d^\ast=4/5\).
Recall \(\mathcal D_d\), \(r_{d,n}\), \(T_{d,n}\), and \(L_{d,M}\)
from \eqref{eq:L.2.D} and \eqref{eq:L.2.3}--\eqref{eq:L.2.5}.
Whenever the first \(M\) terms lie on the positive branch, those definitions
give
\begin{equation}
 L_{d,M}(a)=\mathcal D_d\sum_{n=1}^{M}
 \left(a^{-1}r_{d,n}^{d-1}
 -\beta_da^{-3}r_{d,n}^{d-3}\right).
\label{supp:uniform-active-sum}
\end{equation}

The exact-arithmetic program identified in the data-availability statement
verifies the Taylor-model coefficient and remainder lists used below.

\begin{proposition}[Uniform finite-level estimate]
\label{supp:aggregate-proposition}
For every integer \(d\ge7\), every \(M\in\{2,3,4\}\), and every
\(a\in I_M\), the first \(M\) terms lie on the positive branch and
\[
 (a\partial_a)^2L_{d,M}(a)<0,
 \qquad
 L_{d,M}(a)>1.
\]
\end{proposition}

\begin{proof}

\subsection{Exact curvature decomposition}

Set
\[
 t=a^{-1/3},\qquad
 \varepsilon=d^{-1/2},\qquad
 \delta=\varepsilon^2=d^{-1},
\]
and, for each \(n\),
\begin{equation}
 z_n=\alpha_nt^2,\qquad h=\frac{t^3}{2}=\frac1{2a},
 \qquad S_n=z_n+h\varepsilon,
 \qquad r_n=1-\delta S_n.
\label{supp:uniform-compact-variables}
\end{equation}
We use the componentwise enclosures \eqref{eq:Const.alpha-fine} and the
rational \(t\)-bands
\begin{equation}
\begin{array}{c|c}
M&\text{rational enclosure of }a^{-1/3}\text{ on }I_M\\ \hline
2&[9687/10000,10773/10000]\\
3&[4189/5000,1211/1250]\\
4&[921/1250,8379/10000].
\end{array}
\label{supp:uniform-t-bands}
\end{equation}
The comparisons in \eqref{supp:uniform-t-bands} follow by cubing positive
rationals; the common \(\alpha_n\)-enclosures were proved in
Section \ref{supp:constant-proofs}.

For \(a\in I_M\) and \(n\le M\), these enclosures give
\[
 z_n<\frac{18}{5},\qquad h\varepsilon<\frac14,
 \qquad 0<\delta S_n<\frac{11}{20}.
\]
It follows that \(r_n>9/20\).  On the other hand,
\[
 \frac{\sqrt{\beta_d}}a
 <\frac5{8\sqrt6}<\frac{25}{96}<\frac9{20}.
\]
Thus every term in \eqref{supp:uniform-active-sum} remains on its positive
branch throughout its assigned interval.

Put
\[
 \Theta=a\partial_a,\qquad
 w=\frac23z+h\varepsilon,\qquad
 v=\frac49z+h\varepsilon,\qquad r=1-\delta S.
\]
Direct differentiation gives the exact identity
\begin{equation}
 \frac{\Theta^2(a^{-m}r^{d-c})}{a^{-m}r^{d-c}}
 =A_{m,c},
 \qquad
 A_{m,c}=m^2-(1-c\delta)\frac{2mw+v}{r}
 +(1-c\delta)(1-(c+1)\delta)\frac{w^2}{r^2}.
\label{supp:uniform-curvature-identity}
\end{equation}
Let
\[
 A_m^0=(w-m)^2-v.
\]
Then
\begin{equation}
 A_{m,c}=A_m^0+\delta B_{m,c},
\label{supp:uniform-A-decomposition}
\end{equation}
where
\begin{equation}
 B_{m,c}=-\frac{(2mw+v)(S-c)}r
 +\frac{w^2\{2S-(2c+1)+\delta(c(c+1)-S^2)\}}{r^2}.
\label{supp:uniform-B-formula}
\end{equation}

For \(\delta>0\), define
\[
 \eta_c(\delta,S)
 =\frac{(1-c\delta)\log(1-\delta S)+\delta S}{\delta^2},
 \qquad
 \eta_c(0,S)=cS-\frac{S^2}{2}.
\]
Then \(r^{d-c}=e^{-S}e^{\delta\eta_c}\).  Similarly, write
\[
 (1-\delta)(1-2\delta)(1-3\delta)=e^{\delta\ell(\delta)},
\]
where
\[
 \ell(\delta)=
 \frac{\log(1-\delta)+\log(1-2\delta)+\log(1-3\delta)}{\delta},
 \qquad \ell(0)=-6.
\]
With
\[
 \phi(y)=\begin{cases}(e^y-1)/y,&y\ne0,\\1,&y=0,\end{cases}
 \qquad
 \lambda_1=\eta_1,\qquad \lambda_3=\eta_3+\ell,
\]
put
\begin{equation}
 Q_{1}=\lambda_1\phi(\delta\lambda_1)A_1^0
 +e^{\delta\lambda_1}B_{1,1},
 \qquad
 Q_{3}=\lambda_3\phi(\delta\lambda_3)A_3^0
 +e^{\delta\lambda_3}B_{3,3}.
\label{supp:uniform-Q-formulas}
\end{equation}
Equations \eqref{supp:uniform-active-sum}--
\eqref{supp:uniform-Q-formulas} yield
\begin{equation}
 \frac{\Theta^2L_{d,M}}{\mathcal D_d}
 =\widetilde C_M+\delta R_M,
\label{supp:uniform-curvature-decomposition}
\end{equation}
where
\begin{align}
 \widetilde C_M
 &=\sum_{n=1}^{M}e^{-S_n}
 \left(a^{-1}A_{1,n}^0-\frac{a^{-3}}{24}A_{3,n}^0\right),
\label{supp:uniform-tilted-profile}\\
 R_M
 &=\sum_{n=1}^{M}e^{-S_n}
 \left(a^{-1}Q_{1,n}-\frac{a^{-3}}{24}Q_{3,n}\right).
\label{supp:uniform-remainder}
\end{align}
The correction in \eqref{supp:uniform-curvature-decomposition} is
\(\delta R_M\), not \(\varepsilon R_M\); retaining the common tilt
\(e^{-h\varepsilon}\) in \(\widetilde C_M\) removes the apparent
first-order term.

To make the profile explicit, set \(x=h\varepsilon\),
\(w_{0,n}=2z_n/3\), and \(q=1/(24a^2)\).  Then
\begin{equation}
 \widetilde C_M=e^{-x}
 \bigl(C_{0,M}+C_{1,M}x+C_{2,M}x^2\bigr),
\label{supp:uniform-tilted-quadratic}
\end{equation}
where
\begin{align*}
 C_{0,M}
 &=\sum_{n=1}^{M}\frac{e^{-z_n}}a
 \left[(w_{0,n}-1)^2-\frac49z_n
 -q\left((w_{0,n}-3)^2-\frac49z_n\right)\right],\\
 C_{1,M}
 &=\sum_{n=1}^{M}\frac{e^{-z_n}}a
 \left[(2w_{0,n}-3)-q(2w_{0,n}-7)\right],\\
 C_{2,M}
 &=\sum_{n=1}^{M}\frac{e^{-z_n}}a(1-q).
\end{align*}
In particular,
\begin{equation}
 \partial_x\widetilde C_M=e^{-x}D_M,
 \qquad
 D_M=(C_{1,M}-C_{0,M})
 +(2C_{2,M}-C_{1,M})x-C_{2,M}x^2.
\label{supp:uniform-tilt-derivative}
\end{equation}

\subsection{Uniform rational bounds}

On each rational rectangle write
\[
 t=t_0+t_1\xi,\qquad
 \varepsilon=\varepsilon_0+\varepsilon_1\nu,
 \qquad |\xi|,|\nu|\le1.
\]
A Taylor model consists of a rational polynomial
\(P(\xi,\nu)\), of the stated degree in each variable, and a rational
remainder interval.  Products are truncated coordinatewise; discarded
coefficients are added to the remainder using
\[
 |P(\xi,\nu)|\le |p_{00}|+
 \sum_{(i,j)\ne(0,0)}|p_{ij}|.
\]
The reciprocal is expanded to degree ten by the geometric series.
For an exponential, write \(Y=c+y\), expand \(e^y\) to degree ten, and
bound its remainder by \(e^\rho\rho^{11}/11!\), where
\(|y|\le\rho\).  This is multiplied by an upper enclosure for \(e^c\),
obtained from a degree-forty Taylor sum with a Lagrange remainder whose
factor \(e^{|c|}\) is bounded by a geometric majorant.  The corresponding
degree-ten remainder for \(\phi(y)\) is
\(e^\rho\rho^{11}/12!\).

The logarithmic terms use the one-sided expansions
\[
 \eta_c=cS-(1-c\delta)
 \sum_{k=2}^{10}\frac{\delta^{k-2}S^k}{k}+E_{\eta},
 \qquad
 -\frac{S_+^2u_+^9}{11(1-u_+)}\le E_{\eta}\le0,
\]
where \(S\le S_+\), \(\delta S\le u_+<1\), and
\(c\delta\le3(189/500)^2<1\).  For \(\ell\), with
\(\delta_+\le(189/500)^2\) and \(3\delta_+<1\), use
\[
 \ell=-\sum_{j=1}^{3}\sum_{k=1}^{10}
 \frac{j^k\delta^{k-1}}k+E_{\ell},
 \qquad
 -\sum_{j=1}^{3}\frac{j(j\delta_+)^{10}}
 {11(1-j\delta_+)}\le E_{\ell}\le0.
\]
Thus every table entry reduces to exact rational arithmetic.

Let
\[
 \bar\varepsilon=\frac{189}{500}>\frac1{\sqrt7}.
\]
For \(M=2\), direct bivariate Taylor models of degree six applied to
\eqref{supp:uniform-tilted-quadratic} give
\begin{equation}
\begin{array}{c|c|c}
M&t\text{-interval}&\widetilde C_M<\\ \hline
2&[9687/10000,10001/10000]&-1/5\\
2&[9999/10000,10773/10000]&-1/5.
\end{array}
\label{supp:uniform-profile-M2}
\end{equation}
Both rows hold for \(0\le\varepsilon\le\bar\varepsilon\), and their
overlap covers the entire \(M=2\) interval.

For \(M=3,4\), we have
\[
 C_{0,M}=\Theta^2\sum_{n=1}^{M}
 \left(a^{-1}-\frac{a^{-3}}{24}\right)e^{-z_n}.
\]
If
\[
 P_m(z)=\left(\frac23z-m\right)^3
 -\frac43z\left(\frac23z-m\right)+\frac8{27}z,
\]
then
\[
 \Theta C_{0,M}=\sum_{n=1}^{M}e^{-z_n}
 \left(a^{-1}P_1(z_n)-\frac{a^{-3}}{24}P_3(z_n)\right).
\]
Degree-six Taylor models on the full \(t\)-intervals give
\begin{equation}
\begin{array}{c|c|c|c|c|c}
M&t\text{-interval}&\Theta C_{0,M}>&D_M<&a_*&C_{0,M}(a_*)<\\ \hline
3&[4189/5000,1211/1250]&1/8&-13/100&17/10&-21/200\\
4&[921/1250,8379/10000]&7/100&-37/100&5/2&-17/250.
\end{array}
\label{supp:uniform-profile-M34}
\end{equation}
The first inequality shows that \(C_{0,M}\) increases with \(a\), while
the second and \eqref{supp:uniform-tilt-derivative} show that the tilt can
only decrease it.  Hence
\begin{equation}
 \widetilde C_2<-\frac15,\qquad
 \widetilde C_3<-\frac{21}{200},\qquad
 \widetilde C_4<-\frac{17}{250}.
\label{supp:uniform-profile-bounds}
\end{equation}

Applying the same rational Taylor rules to the full sum
\eqref{supp:uniform-remainder}
gives
\begin{equation}
\begin{array}{c|c|c|c|c}
M&t\text{-interval}&\varepsilon\text{-interval}
&\text{degree}&R_M<\\ \hline
2&[9687/10000,10773/10000]&[0,189/500]&6&1\\
3&[4189/5000,1807/2000]&[0,189/500]&6&1/2\\
3&[4517/5000,1211/1250]&[0,1/5]&6&1/2\\
3&[4517/5000,1211/1250]&[1/5,189/500]&6&1/2\\
4&[921/1250,4037/5000]&[0,189/500]&6&1/3\\
4&[8073/10000,8379/10000]&[0,189/500]&8&1/3.
\end{array}
\label{supp:uniform-remainder-bounds}
\end{equation}
The slight overlaps between adjacent \(t\)-intervals ensure complete
coverage.  Since an actual dimension \(d\ge7\) satisfies
\(\delta\le1/7\), equations \eqref{supp:uniform-curvature-decomposition},
\eqref{supp:uniform-profile-bounds}, and
\eqref{supp:uniform-remainder-bounds} give
\begin{equation}
\begin{aligned}
 \frac{\Theta^2L_{d,2}}{\mathcal D_d}
 &<-\frac15+\frac17=-\frac2{35},\\
 \frac{\Theta^2L_{d,3}}{\mathcal D_d}
 &<-\frac{21}{200}+\frac1{14}=-\frac{47}{1400},\\
 \frac{\Theta^2L_{d,4}}{\mathcal D_d}
 &<-\frac{17}{250}+\frac1{21}=-\frac{107}{5250}.
\end{aligned}
\label{supp:uniform-curvature-margins}
\end{equation}
Thus \(L_{d,M}\) is strictly concave as a function of \(\log a\) on
\(I_M\), uniformly for all \(d\ge7\).

\subsection{Endpoint values and conclusion}

The endpoint calculation uses the exact identity
\begin{equation}
 \frac{L_{d,M}(a)}{\mathcal D_d}
 =\sum_{n=1}^{M}e^{-S_n}
 \left(a^{-1}e^{\delta\eta_{1,n}}
 -\beta_da^{-3}e^{\delta\eta_{3,n}}\right).
\label{supp:uniform-value-identity}
\end{equation}
The dimension-factor bound \eqref{eq:L.dimension-factor-bounds} gives
\[
 \mathcal D_d>\sqrt{2\pi(1+\delta/2)}
 >\frac{125331}{50000}
 \left(1+\frac\delta4-\frac{\delta^2}{32}\right),
\]
where the second inequality follows from
\[
 \left(\frac{125331}{50000}\right)^2
 <2\frac{3141592}{10^6}<2\pi,
 \qquad
 \sqrt{1+y}\ge1+\frac y2-\frac{y^2}{8}
 \quad(y\ge0).
\]

Using this lower bound in \eqref{supp:uniform-value-identity}, the exact
formula for \(\beta_d\), and degree-eight Taylor models gives the following
four endpoint families:
\begin{equation}
\begin{array}{c|c|c|c|c}
M&a&t\text{-enclosure}&\varepsilon\text{-interval}&L_{d,M}(a)>\\ \hline
2&4/5&[2693/2500,10773/10000]&[0,1/5]&21/20\\
2&4/5&[2693/2500,10773/10000]&[1/5,189/500]&101/100\\
2&11/10&[9687/10000,1211/1250]&[0,189/500]&101/100\\
3&17/10&[4189/5000,8379/10000]&[0,189/500]&103/100\\
4&5/2&[921/1250,7369/10000]&[0,189/500]&101/100.
\end{array}
\label{supp:uniform-endpoint-bounds}
\end{equation}
For each displayed \(t\)-enclosure, direct cubing gives
\(t_-^3<1/a<t_+^3\).
The first two rows form a single endpoint family; the split at
\(\varepsilon=1/5\) is used only to sharpen its lower bound.

At the common interfaces, \eqref{eq:L.2.5} and the activation proved above
give \(L_{d,M+1}>L_{d,M}\).  Hence the endpoint bounds at \(11/10\) and
\(17/10\) also supply the left endpoints of the next two intervals.
Together with \eqref{supp:uniform-endpoint-bounds}, both endpoints of
\(L_{d,M}\) exceed one on every \(I_M\).  By
\eqref{supp:uniform-curvature-margins}, \(L_{d,M}\) lies above its endpoint
chord, proving the proposition.
\end{proof}

\section{Beta-integral estimates}

The parity bases used in \eqref{eq:L.5.48} are
\begin{equation*}
\mathfrak h_{10}
=\frac{134217728\sqrt2}{61108047\pi},
\qquad
\mathfrak h_{11}
=\frac{67584\sqrt2}{96577}.
\end{equation*}
Together with \eqref{eq:Const.rad} and \eqref{eq:Const.pi}, these values give
the two rational parity-base margins used in the body.

\subsection{\texorpdfstring{Beta-integral estimate on \([3,5]\)}{Beta-integral estimate on [3,5]}}

\begin{proposition}[Finite beta-integral estimate on \(3\le a\le5\)]
\label{supp:beta-finite}
For \(5\le d\le9\), the function \(\underline{\mathfrak B}_d\) in
\eqref{eq:L.5.18} is concave on \(3\le a\le5\), and
\[
\underline{\mathfrak B}_d(3)>1,
\qquad
\underline{\mathfrak B}_d(5)>1.
\]
\end{proposition}

\begin{proof}
Set \(m=d-1\), \(u=c_d/a\), and \(v=\alpha_1/(d a^{2/3})\), so
\(\Psi_d(a)=3\Phi_m(u,v)-\Phi_m(u,0)\).  Define
\begin{equation}
G_m(u,v):=u\,\partial_u\Psi_d+\frac23v\,\partial_v\Psi_d.
\label{eq:L.5.21a}
\end{equation}
Because \(du/d\log a=-u\) and \(dv/d\log a=-2v/3\),
\begin{equation}
\frac{d}{d\log a}\Psi_d(a)=-G_m(u,v).
\label{eq:L.5.21b}
\end{equation}

For the five relevant dimensions the parameter ranges lie in the
following rational rectangles:
\begin{equation}
\begin{array}{c|c|c|c}
m&\text{range of }a&0<u\le U_m&0<v\le V_m\\ \hline
4&a\ge5/2&9/500&1/10\\
5&a\ge5/2&7/500&81/1000\\
6&a\ge3&1/100&31/500\\
7&a\ge3&3/400&27/500\\
8&a\ge3&1/160&6/125.
\end{array}
\label{eq:L.5.21c}
\end{equation}
These containments follow by cross multiplication from
\(\alpha_1<887/1000\) and the lower bounds
\begin{equation*}
\sqrt5>\frac{559}{250},\quad
\sqrt6>\frac{120}{49},\quad
\sqrt7>\frac{37}{14},\quad
\sqrt8>\frac{280}{99},\quad \sqrt9=3,
\end{equation*}
together with \((5/2)^{2/3}>921/500\) and \(3^{2/3}>52/25\).

Write the exact polynomial expansion as
\begin{equation*}
G_m=-A_mu-B_mv+\sum_{i+j\ge2}g_{ij}^{(m)}u^iv^j
\end{equation*}
where direct extraction of the linear terms gives the uniform identities
\begin{equation*}
A_m=12m,\qquad B_m=\frac{80m}{9}.
\end{equation*}
Define the contributions of the positive nonlinear coefficients:
\begin{equation}
\begin{aligned}
C_{u,m}&=\sum_{\substack{i\ge2\\g_{i0}^{(m)}>0}}
g_{i0}^{(m)}U_m^{\,i-1},\\
C_{v,m}&=\sum_{\substack{j\ge1,\ i+j\ge2\\g_{ij}^{(m)}>0}}
g_{ij}^{(m)}U_m^{\,i}V_m^{\,j-1}.
\end{aligned}
\label{eq:L.5.21d}
\end{equation}
On the rectangle, \(u^i\le U_m^{i-1}u\) and
\(u^iv^j\le U_m^iV_m^{j-1}v\).  Discarding the nonpositive nonlinear
monomials therefore gives
\begin{equation*}
G_m\le-(A_m-C_{u,m})u-(B_m-C_{v,m})v.
\end{equation*}
Expansion of the degree-\(m\) polynomial \(\Phi_m\) and rational cross
multiplication give
\begin{equation}
\begin{array}{c|c}
m&\text{resulting bound}\\ \hline
4&G_4\le-42u-6v\\
5&G_5\le-53u-4v\\
6&G_6\le-64u-8v\\
7&G_7\le-76u-10v\\
8&G_8\le-87u-10v.
\end{array}
\label{eq:L.5.21e}
\end{equation}
The exact rational subtractions give the following additional margins:
\begin{equation}
\begin{array}{c|cc}
m&(A_m-C_{u,m})-\text{displayed coefficient}&
(B_m-C_{v,m})-\text{displayed coefficient}\\ \hline
4&>4/5&>1/100\\
5&>1/4&>7/10\\
6&>3/4&>2/3\\
7&>2/5&>1/20\\
8&>1/2&>1/10.
\end{array}
\label{eq:L.5.21f}
\end{equation}
Thus \(\Psi_d(a)\) is strictly increasing on every range in
\eqref{eq:L.5.21c}.

For the five substitutions at \(A=5\) below and the two at \(A=3\) in
Proposition \ref{supp:beta-V}, use the following one-sided rational
enclosure.  If
\(\ell_d\le\sqrt d\le u_d\) and
\(s_A^-<A^{2/3}<s_A^+\), put
\begin{equation}
\begin{gathered}
u_A^-=\frac1{2Ad\,u_d},\quad
u_A^+=\frac1{2Ad\,\ell_d},\quad
v_A^-=\frac{22}{25d\,s_A^+},\quad
v_A^+=\frac{887}{1000d\,s_A^-},\\
\rho_A^-=1-u_A^+-v_A^+,\quad
\rho_A^+=1-u_A^--v_A^-,\quad
W_A^+=u_A^++\frac23v_A^+,\\
Q_A^+=2(\rho_A^+)^2
-m\rho_A^-\left(4u_A^-+\frac{22}{9}v_A^-\right)
+m(m-1)(W_A^+)^2,\\
P_A^-=(1-u_A^+)^{m-2}
\left[2-4du_A^++d(d+1)(u_A^+)^2\right].
\end{gathered}
\label{eq:L.5.22}
\end{equation}
Here \(0<u_A^+<1/(12d)\).  Thus the bracket
\(b_d(u)=2-4du+d(d+1)u^2\) satisfies
\begin{equation*}
b_d(u)>\frac53,\qquad
b_d'(u)=2d((d+1)u-2)<0,
\end{equation*}
so the bracket in \(P_A^-\) is positive and decreasing on these boxes.
Whenever \(Q_A^+>0\), the defining formula for \(\Phi_m\) gives
\begin{equation}
\Psi_d(A)\le3(\rho_A^+)^{m-2}Q_A^+-P_A^-.
\label{eq:L.5.22a}
\end{equation}
Use
\begin{equation}
\begin{array}{c|cc}
d&\ell_d&u_d\\ \hline
5&38/17&161/72\\
6&120/49&49/20\\
7&37/14&127/48\\
8&280/99&99/35\\
9&3&3
\end{array}
\qquad
\begin{array}{c|cc}
A&s_A^-&s_A^+\\ \hline
3&52/25&2081/1000\\
5&731/250&117/40.
\end{array}
\label{eq:L.5.22b}
\end{equation}
Squaring and cubing prove the displayed enclosures.  At \(A=5\), exact
substitution in \eqref{eq:L.5.22} yields the following five bounds:
\begin{equation}
\begin{array}{c|cc}
d&Q_5^+&\Psi_d(5)\\ \hline
5&>1&<103/100\\
6&>1&<103/100\\
7&>1&<103/100\\
8&>1&<103/100\\
9&>1&<103/100.
\end{array}
\label{eq:L.5.23}
\end{equation}
Monotonicity therefore proves \(\Psi_d(a)<103/100\) on
\(3\le a\le5\).  Moreover,
\(\mathcal D_d<8/3\) and
\(2\sigma_dc_d\ge19/54\) on \(5\le d\le9\).  Hence
\begin{equation*}
-2\sigma_dc_d+\frac{\mathcal D_d}{8}\Psi_d
<-\frac{19}{54}+\frac{103}{300}
=-\frac{23}{2700}<0.
\end{equation*}
All remaining terms in \eqref{eq:L.5.20} are negative, and consequently
\begin{equation}
\underline{\mathfrak B}_d''(a)<0
\qquad(5\le d\le9,\ 3\le a\le5).
\label{eq:L.5.24}
\end{equation}

For the endpoints, use the strict lower bounds

\begin{equation}
\begin{array}{c|cc}
d&h_d^-<\mathfrak h_d&\mathcal D_d^-<\mathcal D_d\\ \hline
5&979/1000&263/100\\
6&251/256&5/2\\
7&637/648&5/2\\
8&197/200&5/2\\
9&955/968&5/2 .
\end{array}
\label{eq:L.5.25}
\end{equation}
Using \(707/500<\sqrt2\), \(333/106<\pi<355/113\), and the square-root
bounds in \eqref{eq:L.5.22b}, the closed Gamma formulas give the following
rational bounds; integer cross multiplication shows that each exceeds its
counterpart in \eqref{eq:L.5.25}:
\begin{equation}
\begin{array}{c|c|c}
d&\text{rational lower bound for }\mathfrak h_d
&\text{rational lower bound for }\mathcal D_d\\ \hline
5&\frac{160}{231}\frac{707}{500}
&\frac38\frac{333}{106}\frac{38}{17}\\
6&\frac{32768}{15015}\frac{707}{500}\frac{113}{355}
&\frac{16}{15}\frac{120}{49}\\
7&\frac{896}{1287}\frac{707}{500}
&\frac5{16}\frac{333}{106}\frac{37}{14}\\
8&\frac{8388608}{3828825}\frac{707}{500}\frac{113}{355}
&\frac{64}{35}\frac{707}{500}\\
9&\frac{32256}{46189}\frac{707}{500}
&\frac{105}{128}\frac{333}{106}.
\end{array}
\label{eq:L.5.25c}
\end{equation}

For \(A=3,5\), set

\begin{equation*}
c_d^-:=\frac1{2d\,u_d},\qquad
c_d^+:=\frac1{2d\,\ell_d},\qquad
b_3:=\frac{52}{25},\qquad b_5:=\frac{73}{25},
\end{equation*}
and define the rational number

\begin{equation}
\begin{aligned}
E_{d,A}^{[3,5]}:={}&h_d^-
-\frac{\sigma_d c_d^+}{A}
-\frac{\mathcal K_d}{(A-c_d^+)^2}\\
&+\frac{\mathcal D_d^-}{8A}
\left[
3\left(
1-\frac{c_d^+}{A}
-\frac{887}{1000\,d\,b_A}
\right)^m
-\left(1-\frac{c_d^-}{A}\right)^m
-\frac{3\beta_d}{A^2}
\right].
\end{aligned}
\label{eq:L.5.25a}
\end{equation}
Let \(\mathscr C_{d,A}\) denote the square bracket in
\eqref{eq:L.5.25a}.  Direct rational substitution using
\eqref{eq:L.5.22} gives

\begin{equation}
\begin{array}{c|cc}
d&\mathscr C_{d,3}&\mathscr C_{d,5}\\ \hline
5&>1&>32/25\\
6&>1&>63/50\\
7&>49/50&>5/4\\
8&>49/50&>31/25\\
9&>97/100&>123/100 .
\end{array}
\label{eq:L.5.25b}
\end{equation}
Thus \(\mathscr C_{d,A}>4/5\) throughout, justifying the lower bound for
\(\mathcal D_d\).

The directed bounds in \eqref{eq:L.5.22} and \eqref{eq:L.5.25} give
\(\underline{\mathfrak B}_d(A)>E_{d,A}^{[3,5]}\); integer cross
multiplication then gives

\begin{equation}
\begin{array}{c|cc}
d&E_{d,3}^{[3,5]}&E_{d,5}^{[3,5]}\\ \hline
5&>1007/1000&>1013/1000\\
6&>126/125&>507/500\\
7&>203/200&>1019/1000\\
8&>1021/1000&>128/125\\
9&>513/500&>1027/1000 .
\end{array}
\label{eq:L.5.26}
\end{equation}
All entries in \eqref{eq:L.5.26} exceed \(1+1/4000\); with
\eqref{eq:L.5.24}, this proves the proposition.
\end{proof}

\subsection{\texorpdfstring{A uniform bound for \(\Psi_d\)}{A uniform bound for Psi d}}

\begin{proposition}[Uniform bound for \(\Psi_d\)]
\label{supp:beta-V}
The bound \eqref{eq:L.5.34} holds.
\end{proposition}

\begin{proof}
For \(d=5,6\), \eqref{eq:L.5.21a}--\eqref{eq:L.5.21e} gives
monotonicity throughout \(5/2\le a\le3\), while the one-sided bounds
\eqref{eq:L.5.22} and \eqref{eq:L.5.22b} at \(A=3\) give
\begin{equation}
\begin{array}{c|cc}
d&Q_3^+&\Psi_d(3)\\ \hline
5&>7/10&<11/100\\
6&>3/4&<1/12.
\end{array}
\label{eq:L.5.36}
\end{equation}
Thus \(\Psi_d(a)<1/4\) for \(d=5,6\) and \(5/2\le a\le3\).

\subsubsection{\texorpdfstring{A completion-of-squares proof of the convexity of \(V\)}{A completion-of-squares proof of the convexity of V}}

For \(d\ge7\), put \(s=d^{-1/2}\), so
\(0\le s\le19/50\), and define

\begin{equation*}
\kappa_A=\frac{887}{1842},\qquad
\kappa_C=\frac{880}{2081},
\end{equation*}

\begin{equation*}
\begin{aligned}
\rho_A&=1-\frac{s^3}{5}-\kappa_As^2,&
\rho_C&=1-\frac{s^3}{6}-\kappa_Cs^2,\\
X_A&=(1-s^2)\frac{s}{5},&
X_C&=(1-s^2)\frac{s}{6},\\
Y_A&=(1-s^2)\kappa_A,&
Y_C&=(1-s^2)\kappa_C,
\end{aligned}
\end{equation*}

\begin{equation*}
\begin{aligned}
Q(s)&=
2\rho_C^2-\rho_A\left(4X_C+\frac{22}{9}Y_C\right)
+\left(X_A+\frac23Y_A\right)^2,\\
Z(s)&=
(1-3s^2)\left(\frac{s}{6}+\kappa_C\right)
+\frac12(1-3s^2)s^2
\left(\frac{s}{6}+\kappa_C\right)^2,\\
P_0(s)&=
\left(1-\frac{(1-3s^2)s}{5}\right)
\left(2-\frac{4s}{5}+\frac{s^2+s^4}{25}\right),
\end{aligned}
\end{equation*}

\begin{equation*}
E_4(z)=1-z+\frac{z^2}{2}-\frac{z^3}{6}+\frac{z^4}{24},
\qquad
V(s)=3Q(s)E_4(Z(s))-P_0(s).
\end{equation*}
Differentiating the displayed polynomial gives
\begin{equation*}
\begin{aligned}
Q'(s)={}&-\frac{7436}{13815}
-\frac{259077005522}{397167642225}s
+\frac{23368101}{28749015}s^2\\
&+\frac{53792176118548}{826505863470225}s^3
-\frac{16881305}{28749015}s^4
-\frac{17}{75}s^5 .
\end{aligned}
\end{equation*}

Coefficientwise comparison on \(0\le s\le19/50\) gives
\begin{equation*}
\begin{aligned}
Q'(s)
&<-\frac{53}{100}+\frac{82}{100}s^2+\frac7{100}s^3\\
&\le-\frac{53}{100}
+\frac{82}{100}\left(\frac{19}{50}\right)^2
+\frac7{100}\left(\frac{19}{50}\right)^3
<-\frac25.
\end{aligned}
\end{equation*}
Moreover, direct rational substitution gives
\[
Q\left(\frac{19}{50}\right)-\frac{19}{60}>\frac12.
\]
Thus \(Q\) is decreasing and \(Q(s)>19/60>0\).

Set

\begin{equation*}
\begin{gathered}
E=E_4(Z),\qquad
\mathcal U=-E_4'(Z),\qquad
\mathcal R=E_4''(Z),\\
A=-Q',\qquad B=-Q'',\qquad C=-Z'',\qquad p=Z'.
\end{gathered}
\end{equation*}
Direct differentiation and coefficient comparison give

\begin{equation*}
Q''(s)
<
-\frac{13}{20}
+2\frac{813}{1000}s
+3\frac{66}{1000}s^2
\le-\frac{4411}{1250000}<0,
\end{equation*}

\begin{equation*}
Q'''(s)>
\frac85-\frac{36}{5}s^2-\frac{68}{15}s^3
\ge\frac{146047}{468750}>0.
\end{equation*}
Moreover,
\nopagebreak
\begin{equation*}
\begin{aligned}
Z''(s)={}&-\frac54s^4-\frac{8800}{2081}s^3
-\frac{79304639}{25983366}s^2
-\frac{5363}{2081}s-\frac{10213280}{4330561}<0,\\
Z'''(s)={}&-5s^3-\frac{26400}{2081}s^2
-\frac{79304639}{12991683}s-\frac{5363}{2081}<0,
\end{aligned}
\end{equation*}
and

\begin{equation*}
P_0'''(s)
=\frac{126}{25}s^4+\frac{24}{25}s^2
-\frac{264}{25}s+\frac{894}{125}
>\frac{1962}{625}>0.
\end{equation*}
Thus \(Q\) and \(B\) decrease, \(A\), \(C\), and \(P_0''\) increase,
and \(Z\) is concave.  Since every coefficient of \(Z\) of degree at least
two is negative,

\begin{equation*}
Z(s)\le
\kappa_C+\frac{s}{6}
-\frac{5106640}{4330561}s^2
<\frac{429}{1000}.
\end{equation*}
Endpoint substitution, using concavity for the lower bounds and
monotonicity on the last two intervals, gives

\begin{equation}
\begin{array}{c|c}
\text{interval for }s&\text{range of }Z(s)\\ \hline
[0,1/10]&[21/50,429/1000]\\
{}[1/10,1/5]&[2/5,429/1000]\\
{}[1/5,3/10]&[7/20,41/100]\\
{}[3/10,19/50]&[7/25,353/1000].
\end{array}
\label{eq:S.betaV.1}
\end{equation}
The signs in \eqref{eq:L.5.19}, Bernoulli's inequality, and
\((1-t)^m\le e^{-m(t+t^2/2)}\) imply
\[
\Phi_m(u,v)\le e^{-Z(s)}Q(s),
\qquad
\Phi_m(u,0)\ge P_0(s).
\]
Since \(Q>0\), \(Z\ge0\), and \(e^{-z}\le E_4(z)\), it follows that
\begin{equation}
\Psi_d(a)\le V(s).
\label{eq:S.betaV.bridge}
\end{equation}
For \(0<Z<1\), \(E_4(Z)\), \(-E_4'(Z)\), and \(E_4''(Z)\) decrease with
\(Z\).  The preceding signs also give
\(Q,E,\mathcal U,\mathcal R,A,B,C>0\).  Together with
\eqref{eq:S.betaV.1}, these monotonicities give the rational factor bounds:

\begin{equation}
\begin{array}{c|ccccccccc}
&Q_{\min}&A_{\max}&B_{\max}&E_{\max}&
\mathcal U_{\min}&\mathcal U_{\max}&\mathcal R_{\min}&C_{\min}&
(P_0'')_{\max}\\ \hline
[0,1/10]&1&3/5&2/3&2/3&649/1000&2/3&13/20&23/10&11/10\\
{}[1/10,1/5]&19/20&16/25&1/2&17/25&649/1000&67/100&33/50&13/5&163/100\\
{}[1/5,3/10]&22/25&67/100&7/20&71/100&33/50&71/100&67/100&3&21/10\\
{}[3/10,19/50]&83/100&17/25&23/100&19/25&7/10&19/25&7/10&7/2&12/5 .
\end{array}
\label{eq:S.betaV.2}
\end{equation}
Twice differentiating \(V\) gives the exact identity

\begin{equation*}
V''
=-3BE+6A\mathcal U p+3Q\mathcal R p^2
+3Q\mathcal U C-P_0''.
\end{equation*}
Completing the square in \(p\) yields

\begin{equation}
V''
\ge
-3BE+3Q\mathcal U C
-\frac{3A^2\mathcal U^2}{Q\mathcal R}-P_0''.
\label{eq:S.betaV.3}
\end{equation}
Substitution of the four rows of \eqref{eq:S.betaV.2} into
\eqref{eq:S.betaV.3} gives, respectively,

\begin{equation}
\frac{509459}{390000},\qquad
\frac{26738037}{20900000},\qquad
\frac{216531}{176000},\qquad
\frac{52203369}{29050000}.
\label{eq:S.betaV.4}
\end{equation}
Every number in \eqref{eq:S.betaV.4} exceeds one.  Hence
\(V''(s)>1\) on the full interval.

It remains only to bound the two chord endpoints.  At \(s=0\),

\begin{equation*}
Q<\frac{107}{100},\qquad Z>\frac{21}{50},\qquad P_0=2.
\end{equation*}
Since \(E_4\) decreases on \([0,1]\),

\begin{equation*}
\frac18-V(0)>
\frac18-\left[
3\frac{107}{100}E_4\left(\frac{21}{50}\right)-2
\right]
=\frac{77765933}{5000000000}>0.
\end{equation*}
At \(s=19/50\), direct rational substitution gives

\begin{equation*}
Q<\frac{8319}{10000},\qquad
Z>\frac{713}{2500},\qquad
P_0>\frac{1629}{1000},
\end{equation*}
and therefore

\begin{equation*}
\frac{99}{400}-V\left(\frac{19}{50}\right)
>
\frac{99}{400}
-\left[
3\frac{8319}{10000}E_4\left(\frac{713}{2500}\right)
-\frac{1629}{1000}
\right]
>\frac1{30000}>0.
\end{equation*}
Strict convexity places \(V\) below the chord joining its endpoint values,
so

\begin{equation*}
V(s)<\frac{99}{400}<\frac14
\qquad\left(0\le s\le\frac{19}{50}\right).
\end{equation*}
Together with \eqref{eq:S.betaV.bridge}, this proves the proposition for
\(d\ge7\); the cases \(d=5,6\) were proved above.
\end{proof}

\subsection{\texorpdfstring{Endpoint estimates on \(5/2\le a\le3\)}{Endpoint estimates on 5/2<=a<=3}}

\begin{proposition}[Finite beta endpoints]
\label{supp:beta-endpoints}
For \(5\le d\le9\), the bounds in \eqref{eq:L.beta.endpoints} hold.
\end{proposition}

\begin{proof}
Set \(b_{5/2}^-:=921/500\) and \(b_3^-:=52/25\), the two lower
\(A^{2/3}\)-bounds obtained by cubing.

For \(5\le d\le9\), use

\begin{equation}
\begin{array}{c|cc|c}
d&\ell_d\le\sqrt d&\sqrt d\le u_d&h_d^-<\mathfrak h_d\\ \hline
5&559/250&2237/1000&489/500\\
6&120/49&49/20&251/256\\
7&37/14&127/48&637/648\\
8&280/99&99/35&197/200\\
9&3&3&955/968 .
\end{array}
\label{eq:L.5.44}
\end{equation}
The \(\mathfrak h_d\)-entries follow from the primitive rational bounds
in \eqref{eq:L.5.25c}; the weaker bound \(489/500\) suffices for \(d=5\).

For \(A\in\{5/2,3\}\), put

\begin{equation*}
c_d^-=\frac1{2du_d},\qquad
c_d^+=\frac1{2d\ell_d},
\end{equation*}

\begin{equation*}
\rho_{d,A}=
1-\frac{c_d^+}{A}
-\frac{887}{1000d\,b_A^-},
\qquad
\sigma_{d,A}=1-\frac{c_d^-}{A},
\end{equation*}
and define the following rational lower bound:

\begin{equation}
\begin{aligned}
E_{d,A}^{[5/2,3]}:={}&
h_d^--\frac{\sigma_d c_d^+}{A}
-\frac{\mathcal K_d}{(A-c_d^+)^2}\\
&+\frac5{16A}
\left(
3\rho_{d,A}^{d-1}
-\sigma_{d,A}^{d-1}
-\frac{3\beta_d}{A^2}
\right).
\end{aligned}
\label{eq:L.5.45}
\end{equation}
The bracket exceeds \(4/5\) in every row, and all substitutions in
\eqref{eq:L.5.45} have the lower-bound direction.  Integer cross
multiplication gives

\begin{equation}
\begin{array}{c|cc}
d&E_{d,5/2}^{[5/2,3]}&E_{d,3}^{[5/2,3]}\\ \hline
5&\text{handled below}&>4001/4000\\
6&>4001/4000&>4001/4000\\
7&>126/125&>4001/4000\\
8&>507/500&>4001/4000\\
9&>1019/1000&>4001/4000 .
\end{array}
\label{eq:L.5.46}
\end{equation}
For the remaining corner \(d=5,a=5/2\), use

\begin{equation*}
\frac{707}{500}<\sqrt2,\quad
\frac{559}{250}<\sqrt5<\frac{2237}{1000},\quad
\frac{333}{106}<\pi<\frac{355}{113},
\end{equation*}
and

\begin{equation*}
\left(\frac{3\pi}{20\sqrt2}\right)^{2/3}
<\frac{4807}{10000}.
\end{equation*}
At the exceptional corner the following one-sided bounds suffice:

\begin{equation*}
\mathfrak h_5>
\frac{160}{231}\frac{707}{500},\qquad
\frac{40}{2237}<\frac ca<\frac{10}{559},\qquad
c<\frac{25}{559},
\end{equation*}

\begin{equation*}
\mathcal D_5>
\frac{15}{8}\frac{333}{106}\frac{1000}{2237},\qquad
\frac{\alpha_1}{a^{2/3}}<\frac{4807}{10000},\qquad
q=\frac4{3125}.
\end{equation*}
In the lower bound for \(\rho\), the \(\alpha_1/a^{2/3}\)-bound is
divided by \(d=5\); in the negative \(-\rho_0^4\) term, the lower bound
for \(c/a\) is used.  Every remaining occurrence uses the unfavourable
endpoint indicated by its sign.  The two occurrences of
\(\mathfrak h_5\) are first grouped as the positive factor
\(\mathfrak h_5(1-\sigma_5c/a)\); hence its lower endpoint is used.
Substitution into \eqref{eq:L.5.30} gives

\begin{equation}
\widehat{\mathfrak B}_5(5/2)>1+\frac3{5000}>1.
\label{eq:L.5.47}
\end{equation}
Together with \eqref{eq:L.5.46}, this proves Proposition
\ref{supp:beta-endpoints}.
\end{proof}

\section{The uniform high-frequency estimate}

\begin{proposition}[Uniform estimate at \(5d^{3/2}\)]
\label{supp:uniform-tail}
For every integer \(d\ge5\), the explicit condition \eqref{eq:II.37}
holds strictly at \(x=5d^{3/2}\).
\end{proposition}

\begin{proof}
Put \(\eta=5^{2/3}\), \(p=d-1\), and define
\[
\begin{aligned}
u_p&:=\frac p2C_{p-1},&
v_p&:=\frac{A_pC_{p-2}}{5(p+1)^{3/2}},\\
b_p&:=\frac{p}{2\eta(p+1)},&
c_p&:=\frac{p}{20(p+1)^{3/2}},&
e_p&:=\frac{A_p}{50(p+1)^3}.
\end{aligned}
\]
At \(x=5(p+1)^{3/2}\), condition \eqref{eq:II.37} is
\begin{equation}
F_p:=\frac14-u_p-v_p-b_p-c_p-e_p>0.
\label{eq:tail-deficit}
\end{equation}
We use
\begin{equation}
\eta>\frac{731}{250}>\frac{35}{12},
\qquad
25\cdot250^3-731^3=7109>0.
\label{eq:tail-eta-bound}
\end{equation}
The Gamma definition of \(C_q\) gives
\[
u_p=
\frac{\Gamma((p+1)/2)}
{4\sqrt\pi\,\Gamma((p+2)/2)},
\qquad
u_{p+2}=u_p\frac{p+1}{p+2}.
\]
Thus, with \(\Delta u_p:=u_p-u_{p+2}\),
\begin{equation}
\Delta u_p=\frac{u_p}{p+2},
\qquad
u_p>\frac1{2\sqrt{2\pi(p+1)}}.
\label{eq:tail-gamma-loss}
\end{equation}
For the second inequality, strict log-convexity with \(y=(p+1)/2\)
gives
\[
\Gamma\left(y+\frac12\right)^2
<\Gamma(y)\Gamma(y+1)
=y\Gamma(y)^2.
\]
Hence \(\Gamma(y)/\Gamma(y+1/2)>y^{-1/2}\).

For the \(b_p\)-loss,
\begin{equation}
b_{p+2}-b_p
=\frac1{\eta(p+1)(p+3)}
<\frac23\Delta u_p.
\label{eq:tail-b-loss}
\end{equation}
Using \eqref{eq:tail-eta-bound} and
\(\pi<355/113<22/7\), the last comparison reduces after squaring to
\[
8575(p+1)(p+3)^2>57024(p+2)^2.
\]
At \(p=q+4\), the difference is
\[
8575q^3+105901q^2+336137q+48011>0.
\]

Next,
\begin{equation}
\frac{v_{p+2}}{v_p}
=
\rho_p:=
\frac{p(p+2)}{(p-2)(p+1)}
\left(\frac{p+1}{p+3}\right)^{3/2},
\label{eq:tail-v-ratio}
\end{equation}
and
\begin{equation}
u_pv_p=
\frac{p-2}{480\pi(p+1)^{3/2}}.
\label{eq:tail-uv-product}
\end{equation}
Equations \eqref{eq:tail-gamma-loss}, \eqref{eq:tail-uv-product}, and
\(\pi>25/8\)
give \(v_p<1/300\).  Moreover,
\begin{equation}
0<\rho_p^2-1<\frac{18}{p^2}.
\label{eq:tail-ratio-bounds}
\end{equation}
For the lower sign, clearing the positive denominator gives
\[
p^2(p+2)^2(p+1)-(p-2)^2(p+3)^3
=13p^3+49p^2-108>0
\qquad(p\ge4).
\]
For the upper sign, clear the positive denominator and cancel the positive
factor \((p+1)^2\).  After writing \(p=q+4\), the remaining inequality is
the positivity of
\[
5q^5+141q^4+1366q^3+5354q^2+6960q+568.
\]
Hence
\begin{equation}
v_{p+2}-v_p<\frac3{100p^2}.
\label{eq:tail-v-increment}
\end{equation}
Also,
\begin{equation}
\frac1{100(p+1)^2}-(e_{p+2}-e_p)
=
\frac{17p^3+99p^2+175p+81}
{300(p+1)^3(p+3)^3}>0.
\label{eq:tail-e-increment}
\end{equation}
Consequently,
\begin{equation}
(v_{p+2}-v_p)+(e_{p+2}-e_p)
<
\frac1{25p^2}
<
\frac13\Delta u_p.
\label{eq:tail-combined-loss}
\end{equation}
The last comparison follows from \eqref{eq:tail-gamma-loss} and
\[
4375p^4>1584(p+2)^2(p+1);
\]
at \(p=q+4\), its difference is
\[
4375q^4+68416q^3+393072q^2+967936q+834880>0.
\]

Finally, \(c_p=p/[20(p+1)^{3/2}]\) decreases for \(p>2\).
Equations \eqref{eq:tail-gamma-loss}--\eqref{eq:tail-b-loss} and
\eqref{eq:tail-v-increment}--\eqref{eq:tail-combined-loss} give
\begin{equation}
F_{p+2}-F_p>0.
\label{eq:tail-parity-monotonicity}
\end{equation}
For the two parity bases, direct substitution gives
\begin{equation}
F_4=
\frac5{32}
-\frac2{5\eta}
-\frac1{25\sqrt5}
-\frac2{225\pi\sqrt5}
-\frac1{6250}
>
\frac{29266951}{224343900000}>0,
\label{eq:tail-even-base}
\end{equation}
using
\[
\eta>\frac{731}{250},\qquad
\sqrt5>\frac{682}{305},\qquad
\pi>\frac{25}{8}.
\]
Similarly,
\begin{equation}
F_5=
\frac14-\frac4{15\pi}
-\frac5{12\eta}
-\frac{35}{768\sqrt6}
-\frac1{4320}
>
\frac{16812181}{5052672000}>0,
\label{eq:tail-odd-base}
\end{equation}
using \(\sqrt6>120/49\).  The radical bounds follow from
\[
5\cdot305^2-682^2=1,\qquad
6\cdot49^2-120^2=6.
\]
The induction \eqref{eq:tail-parity-monotonicity}--\eqref{eq:tail-odd-base}
proves \eqref{eq:tail-deficit} for every \(p\ge4\).
\end{proof}

\section*{Data availability}
\addcontentsline{toc}{section}{Data availability}

Proposition \ref{supp:aggregate-proposition} reduces the compact
two-parameter estimate to finite lists of Taylor-model coefficients and
remainder bounds.  The ancillary files accompanying this article contain
the exact-rational program that checks these lists, its archived output,
and symbolic checks of the other displayed calculations.  The accompanying
README records the commands and dependencies needed to reproduce the
computations.

\phantomsection

\end{document}